\documentclass[reqno, 
11pt]{amsart}
\usepackage{
amssymb,
amsmath,
amsthm,
eucal,
dsfont,
multicol,
mathrsfs,
tikz,
graphicx,
hyperref
}
\usepackage{lipsum}
\makeatletter\renewcommand*{\eqref}[1]{%
\hyperref[{#1}]{\textup{\tagform@{\!\!\ref*{#1}}}}%
}\makeatother %add when putting this file to arXiv

\makeatletter
 
  \@addtoreset{equation}{section}
 \makeatother
\theoremstyle{plain}
\newtheorem{theorem}{Theorem}[section]
\newtheorem{lemma}[theorem]{Lemma}
\newtheorem{proposition}[theorem]{Proposition}
\newtheorem{corollary}[theorem]{Corollary}
\theoremstyle{definition}
\newtheorem{definition}[theorem]{Definition}
\newtheorem{remark}[theorem]{Remark}

\newcommand{\norm}[1]{{\|#1\|}}

\def\supp{\mathop{\mathrm{supp}}\nolimits}

\def\Id{\mathop{\mathrm{Id}}\nolimits}

\def\Ker{\mathop{\mathrm{Ker}}\nolimits}
\def\ac{\mathop{\mathrm{ac}}\nolimits}

\def\Im{\mathop{\mathrm{Im}}\nolimits}

\def\sgn{\mathop{\mathrm{sgn}}\nolimits}
\def\BMO{{\mathop{\mathrm{BMO}}}}
\def\AB{{\mathop{\mathrm{AB}}}}
\def\R{{\mathbb{R}}}
\def\Z{{\mathbb{Z}}}
\def\N{{\mathbb{N}}}
\def\C{{\mathbb{C}}}

\def\F{{\mathcal{F}}}
\def\H{{\mathcal{H}}}

\def\D{{\mathcal{D}}}

\def\<{{\langle}}
\def\>{{\rangle}}

\def\ep{{\varepsilon}}

\DeclareMathOperator*{\slim}{s-lim}

\title[The $L^p$-boundedness of wave operators]{ $L^p$-boundedness of wave operators for  fourth order Schr\"odinger operators with zero resonances on $\mathbb{R}^3$}

\author{Haruya Mizutani}
\address[H. Mizutani]{Department of Mathematics, Graduate School of Science, Osaka University, Toyonaka, Osaka 560-0043, Japan}
\email{haruya@math.sci.osaka-u.ac.jp}
\author{Zijun Wan}
\address[Z. Wan]{Department of Mathematics, Central China Normal University, Wuhan, 430079, P.R. China}
\email{zijunwan@mails.ccnu.edu.cn}
\author{Xiaohua Yao}
\address[X. Yao]{School of Mathematics and Statistics,   Key Laboratory of Nonlinear Analysis  and Applications (Ministry of Education), Central China Normal University, Wuhan, 430079, P.R. China}
\email{yaoxiaohua@ccnu.edu.cn}
\begin{document}
\date{\today}

\begin{abstract}
%\begin{color}{red}
Let $H = \Delta^2 + V$ be the fourth-order Schr\"odinger operator on $\mathbb{R}^3$ with a real-valued fast-decaying potential $V$. If zero is neither a resonance nor an eigenvalue of $H$, then it was recently shown that the wave operators $W_\pm(H, \Delta^2)$ are bounded on $L^p(\mathbb{R}^3)$ for all $1 < p < \infty$ and unbounded at the endpoints $p=1$ and $p=\infty$.

This paper is to further establish the $L^p$-boundedness of $W_\pm(H, \Delta^2)$ that exhibit all types of singularities at the zero energy threshold. We first prove that $W_\pm(H, \Delta^2)$ are bounded on $L^p(\mathbb{R}^3)$ for all $1 < p < \infty$ in the first kind resonance case, and then proceed to establish for the second kind resonance case that they are bounded  on $L^p(\mathbb{R}^3)$ for all $1 < p < 3$, but not if $3 \le p \le \infty$. In the third kind resonance case, we also show that $W_\pm(H, \Delta^2)$ are bounded on $L^p(\mathbb{R}^3)$ for all $1<p<3$ and generically unbounded on $L^p(\R^3)$ for any $3\le p\le\infty$. Moreover, it is also shown that $W_\pm(H, \Delta^2)$ are bounded on $L^p(\R^3)$ for all $3\le p<\infty$ if in addition $H$ has the zero eigenvalue, but no $p$-wave zero resonances and all zero eigenfunctions are orthogonal to $x_ix_jx_kV$ in $L^2(\R^3)$ for all $i,j,k=1,2,3$ with $x=(x_1,x_2,x_3)\in \R^3$.

 These results describe precisely the validity of the $L^p$-boundedness of $W_\pm(H, \Delta^2)$ in $\mathbb{R}^3$ for all types of singularities at the zero energy threshold with some exceptions for the endpoint cases $p=1,\infty$. As an application, $L^p$-$L^q$ decay estimates are also derived for the fourth-order Schr\"odinger equations and Beam equations with zero resonance singularities.

%The proof of these results hinges crucially on the detailed analysis of the asymptotic expansions of the resolvent of $H$ at low energy for various types of zero resonance singularities, as well as the incorporation of crucial ideas by Yajima \cite{Yajima_2021arxiv} and Goldberg and Green %\end{color}
%\cite{GoGr21}.

\end{abstract}

\maketitle
\tableofcontents
\section{Introduction}

\subsection{Backgrounds}
We consider the fourth-order Schr\"odinger operator $$H=\Delta^2+V(x)$$  on $\mathbb R^3$ with a real-valued potentials $V(x)$ satisfying $|V(x)|\lesssim \<x\>^{-\mu}$ with a specific $\mu>0$ and $\<x\>=\sqrt{1+|x|^2}$. %\begin{color}{red}
	Under this condition, $H$ is self-adjoint on $L^2(\R^3)$ with domain $H^4(\R^3)$, the $L^2$-Sobolev space of order $4$, and its spectrum consists of discrete eigenvalues and the essential spectrum $[0,\infty)$ (with possible zero eigenvalue).  Throughout the paper we always assume that $H$ has no positive eigenvalue.
%\end{color}

If $\mu>1$, then it is well-known (see  e.g.  Kuroda \cite{Kuroda}) that the {\it wave operators}
\begin{align}\label{def-wave}
	W_\pm=W_\pm(H,\Delta^2) :=\slim_{t\to\pm\infty}e^{itH}e^{-it\Delta^2}
\end{align}
exist as partial isometries on  $L^2(\mathbb{R}^3)$ and are asymptotically complete, namely the range of $W_\pm$ coincides with the absolutely continuous spectral subspace $\H_{\ac}(H)$ of $H$. In particular, $W_\pm$ satisfy the  identities
$
W_\pm W_\pm^* =P_{\ac}(H)$, $W_\pm^*W_\pm=I
$
and exhibit {\it the intertwining property} $f(H)W_\pm=W_\pm f(\Delta^2)$ for any Borel measurable function $f$ on $\mathbb{R}$,
where $P_{\ac}(H)$ denotes the projection onto $\H_{\ac}(H)$. These formulas, in particular, imply
\begin{align}
	\label{intertwining_1}
	f(H)P_{\ac}(H)=W_\pm f(\Delta^2)W_\pm^*.
\end{align}
We refer to \cite{ReSi} for a detailed description on the mathematical scattering theory.

In this paper we are interested in the boundedness of the wave operators $W_\pm$ and $W_\pm^*$ on $L^p(\R^3)$ for $p\neq2$. One of  our motivations to study the $L^p$-boundedness is that, based on the equality \eqref{intertwining_1}, the $L^p$-boundedness of $W_\pm$ and $W_\pm^*$ can be directly employed to establish the $L^p$-$L^q$ estimates for the perturbed operator $f(H)$ from the same estimates for the free operator $f(\Delta^2)$. This can be seen as follows:
\begin{align}
	\label{Lp-bound of f(H)}
	\|f(H)P_{\ac}(H)\|_{L^p\to L^q}\le \|W_\pm\|_{L^q\to L^q}\ \|f(\Delta^2)\|_{L^p\to L^q}\|W_\pm^*\|_{L^{p}\to L^{p}}.
\end{align}
%In many cases, under appropriate conditions on $f$, it is feasible to establish $L^p$-$L^q$ bounds for $f(\Delta^2)$ using Fourier multiplier methods. Therefore, in order to use the  inequality \eqref{Lp-bound of f(H)}, it would be interesting  to prove the following $L^p$ estimates for $W_\pm$ and $W_\pm^*$:
%\begin{align}
%	\label{Lp-bound}\|W_\pm \phi\|_{L^p(\mathbb{R}^3)}\lesssim \|\phi\|_{L^p(\mathbb{R}^3)},
%	\quad \|W_\pm^* \phi\|_{L^p(\mathbb{R}^3)}\lesssim \|\phi\|_{L^p(\mathbb{R}^3)}.
%\end{align}
%\begin{color}{red}

If $\mu>9$ and zero is neither an eigenvalue nor a resonance of $H$ (see Definition \ref{definition1.1} below), Goldberg and Green \cite{GoGr21} established  for the first time that $W_\pm$ are bounded on $L^p(\R^3)$ for $1 < p < \infty$.
In the endpoint cases $p=1$ and $\infty$, we have proved in \cite{MWY23} that $W_\pm$ become unbounded on $L^1(\R^3)$ and $L^\infty(\R^3)$  assuming that $V$ has a compact support.
Furthermore, the weak $(1,1)$-estimates for $W_\pm$ were also established in the regular case if $\mu > 11$ in \cite{MWY23}. That is,
\begin{align}\label{Weakesti}
	\big|\{x \in \R^3; \ |W_\pm f(x)| \ge \lambda\}\big| & \lesssim \frac{1}{\lambda} \int_{\R^3} |f(x)|dx, \quad \lambda > 0.
\end{align}
As is known, the range of  values $p$ for which the wave operators are bounded on $L^p$ heavily depends on the existence/non-existence of zero resonances of $H$. In the present case, there are three kinds of zero resonances (the third one coincides with the existence of zero eigenvalues), refer to Definition \ref{definition1.1}. Therefore, the main object of the present paper is to establish the cases when zero is either an eigenvalue or a resonance of $H$.

%W here stress that, due to many intriguing  applications, it would be also crucial to establish the $L^p$-bounds \eqref{Lp-bound} even when zero is either an eigenvalue or a resonance of $H$.

Our results in the present paper are divided into three cases according to the types of the zero resonances. For  the first-kind resonance, we show that $W_\pm$ are bounded on $L^p(\R^3)$ for all $1<p<\infty$, which are the same results as in the regular case by Goldberg and Green \cite{GoGr21}. For the second and third kind resonance cases, the situation is more complicated. Indeed, for the case of the second kind resonance, we show that $W_\pm$ are bounded on $L^p(\R^3)$ for all $1<p<3$, but not for any $3\le p\leq\infty$. For the case of the third kind resonance, we first show that $W_\pm$ are bounded on $L^p(\R^3)$ for all $1<p<3$. Remarkably, we can also show in such a case that $W_\pm$ are bounded on $L^p(\R^3)$ for $3\le p<\infty$ if in addition $H$ has no $p$-wave zero resonance and all zero eigenfunctions $\phi$ satisfy the orthogonality conditions: $$\int_{\R^3} x_ix_jx_kV(x)\phi(x)dx=0$$ for all $i,j,k=1,2,3$ with $x=(x_1,x_2,x_3)\in \R^3$. Finally, without such additional conditions, they are shown to be unbounded on $L^p(\R^3)$ for $3\le p\leq\infty$.
%\end{color}

We here stress that, the proofs of these results significantly rely  on the thorough analysis of the asymptotic expansions of the resolvent of $H$ at low energy for all types of zero-resonance singularities, incorporating ideas from Yajima \cite{Yajima_2021arxiv} and Goldberg and Green \cite{GoGr21}.

\subsection{The main results and zero resonances}\label{main results}

In order to present our main results, we first introduce the definition of zero resonance  by Erdo\u gan, Green, and Toprak \cite{Erdogan-Green-Toprak}.
Recall that if there is a nontrivial solution $\phi\in L^2(\R^3)$ satisfying  $\Delta^2\phi+V\phi=0$ in the distribution sense, then  zero is an eigenvalue of $H$.  In general, there possibly exists a nontrivial resonant solution $\phi(x)$ in $L^2_{-s}(\R^3)$ with some $s>0$ but $\phi\notin L^2(\R^3)$  such that $H\phi(x)=0$. This may lead to challenges in achieving dispersive estimates or establishing $L^p$-boundedness for the wave operator of $H$.

Define $L^2_{-s}(\mathbb{R}^3):=\langle \cdot\rangle^{s} L^2(\R^3)$ for $s\in\mathbb{R}$ and  $ W_{-\sigma}(\mathbb{R}^3):=\bigcap_{s>\sigma}L^2_{-s}(\mathbb{R}^3)$ for $\sigma\in \R$. $W_{-\sigma} (\R^3)$ increases as $\sigma$ increases and  $L^2_{-\sigma}(\R^3)\subset W_{-\sigma}(\R^3)$. In particular, $L^2(\R^3)\subsetneq W_{-\frac12}(\R^3)\subsetneq W_{-\frac32}(\R^3)$.

 \begin{definition}\label{definition1.1}  The zero resonance of  $H=\Delta^2+V$ can be classified into the following four types:
\begin{itemize}
\item [(i)] Zero is the  first kind  resonance of $H$ if there exists a $0\neq \phi\in W_{-{3\over 2}}(\mathbb{R}^3)$  but no
 $0\neq \phi\in W_{-{1\over2}}(\mathbb{R}^3)$ such that $H\phi=0$ holds in the distributional sense;
%\vskip0.2cm
\item [(ii)] Zero is the  second kind resonance of $H$ if there exists a $0\neq \phi\in W_{-{1\over2}}(\mathbb{R}^3)$  but no
$0\neq \phi\in L^2(\mathbb{R}^3)$ such that $H\phi=0$ holds  in the distributional sense;
%\vskip0.2cm
\item [(iiii)] Zero is the  third kind resonance of $H$ if $H$ has a zero eigenvalue, namely there exists a  $0\neq\phi\in L^2(\mathbb{R}^3)$ such that $H\phi=0$ holds in the distributional sense;
%\vskip0.2cm
\item[(iv)] Zero is a  regular point of $H$ if   $H$  has neither zero eigenvalue nor zero resonance.
\end{itemize}
\end{definition}

An equivalent characterization of zero resonances of $H$ will be given by  Definition \ref{definition_resonance} in Section \ref{subsection_resonance}. We also recall the notions of $s$-wave and $p$-wave resonances.
%remark
%\begin{definition}
%\label{sp_wave}W
We say that zero is a  $s$-wave resonance if there exists $\phi\in W_{-{3\over 2}}(\R^3)\setminus W_{-{\frac 12}}(\R^3)$ such that $H\phi=0$, and  a   $p$-wave resonance if there exists $\phi\in W_{-{1\over 2}}(\R^3)\setminus L^2(\R^3)$ such that $H\phi=0$. Using these notions,
\begin{itemize}
\item Zero is the first kind  resonance of $H$ if there is only a $s$-wave zero resonance;
\item Zero is the second kind resonance of $H$ if there is a $p$-wave zero resonance, but no zero eigenvalue. In this case, there may or not may be a $s$-wave resonance;
\end{itemize}
%\end{definition}

  \begin{remark}\label{remark-eigenvalue} Some examples of zero resonance or eigenvalue are constructed as follows:
  	
  	(i)~(\text{The $s$-wave and  $p$-wave resonances})
  	If $\phi\in C^\infty(\mathbb{R}^3)$ is positive and $\phi(x)= c>0$ for $|x|>1$, then $H\phi(x)=0$ when the potential $V(x)$ is defined as $V(x)=-(\Delta^2\phi)/\phi(x)$. It can be readily verified that $V\in  C_0^\infty(\mathbb{R}^3)$ and $\phi\in  W_{-\frac{3}{2}}(\mathbb{R}^3)\setminus W_{-\frac{1}{2}}(\mathbb{R}^3)$. Consequently, zero is a resonance of the operator $H$, even in the presence of a compactly supported smooth potential.
  	
  Moreover, if we consider a positive smooth function $\phi(x)$ on $\mathbb{R}^3$ satisfying $\phi(x)=|x|^{-1}$ for $|x|>1$, then we still find $V\in  C_0^\infty(\mathbb{R}^3)$ such that $H\phi=0$ and $\phi\in  W_{-\frac{1}{2}}(\mathbb{R}^3)\setminus L^2(\mathbb{R}^3)$.
  	
%  	Assume that $\phi\in C^\infty(\R^3)$ is a positive function which is equal to $c>0$ for $|x|>1.$ Then $H\phi(x)=0$ when $\displaystyle V(x)=-(\Delta^2\phi)/\phi(x).$ One can easily check that $V\in  C_0^\infty(\R^3)$ and $\phi\in  W_{-\frac{3}{2}}(\R^3).$  Hence zero is a resonance of $H$ even with a compactly supported smooth potential. Besides, if take a positive smoothing function $\phi(x) $ on $\R^3$ satisfying $\phi(x)=|x|^{-1}$ as $|x|>1$,  then we also have $V\in  C_0^\infty(\R^3)$ such that $H\phi=0$ and  $\phi\in  W_{-\frac{1}{2}}(\R^3)$,  but still not in $L^2(\R^3)$.
  	
  	(ii)~(\text{The zero eigenvalue})
  	Unlike the case of zero resonance, the problem of zero eigenvalues for the operator $H$ is more subtle. For example, consider $\phi(x)=(1+x^2)^{-s}$ with $s>\frac{3}{4}$ and define $V(x)$ as $V(x)=-(\Delta^2\phi)/\phi(x)$. In this setup, it can be observed that $V(x)=O((1+|x|)^{-4})$ as $|x|\rightarrow \infty$, and $\phi\in L^2(\mathbb{R}^3)$ while satisfying $H\phi=0$. This demonstrates that zero is an eigenvalue of $H$ in the presence of such slowly decaying potential.
  	
  \end{remark}

%Assume that $\langle x\rangle ^{5+} V\in  L^\infty(\R^3)$. Then it was shown by \cite{Erdogan-Green-Toprak} that  any resonance solution $\phi$ of equation $\Delta^2\phi+V\phi =0$ in $W_{-{3\over 2}}(\R^3)$ can be expressed generally  as follows
%\begin{equation}\label{resonance solution}
%	\phi(x)=c_0+ \sum_{i=1}^3 c_i\ {x_i\over \langle x\rangle }+\sum_{1\le i\le j\le 3} c_{ij}\ {x_ix_j\over \langle x\rangle^3 }+ O_{L^2} (1),
%\end{equation}
%where $c_0$, $c_i \ (1\le i\le 3)$ and $c_{ij}\ (1\le i\le j\le 3)$  are some specific constants.
% Hence
%  \begin{itemize}
%	  \item If zero is a first kind resonance of $H$, then for any such nonzero solution $\phi$  in \eqref{resonance solution},  there exists some $c_i\neq 0$ for $0\le i\le 3$.
%	  % \vskip0.2cm
%	   \item  If zero is a second kind resonance of $H$, then  for any nonzero solution $\phi$ with $c_i=0$ for all $0\le i\le 3$ in \eqref{resonance solution},  there must exist some $c_{ij}\neq 0$ for $1\le i\le j\le 3$.
%	   % \vskip0.2cm
%	   \item  Finally, if zero is  an eigenvalue of $H$, then there exists some nonzero solution $\phi$ in \eqref{resonance solution} such that  $c_i=c_{ij}=0$ for all $i,j$.
%	    \end{itemize}
%According to the Fredholm theorem,  the eigenspace has finite dimension. Therefore, when  modulating these zero eigenfunctions, the codimension of space of the zero resonant functions is  10 at most.

We  emphasize that these zero resonances are closely linked to the asymptotic expansions of the resolvent $R_V(z)$ of the operator $H$ near zero energy. More specifically, the solutions of the Helmholtz equation $Hf=0$ (called zero resonant states) related to zero resonances can be characterized by the kernel subspaces of certain special bounded operators $T_0$,  $T_1$ and $T_2$ on $L^2(\R^3)$. These solutions are manifested in the asymptotic expansion through the projection operators $S_{j+1}$ for $j=0,1,2$ onto the kernel $\Ker (T_j)$, respectively (see Section \ref{subsection_resonance} below for the details).
Moreover,  it is worth noting that these asymptotic expansions play a crucial role in proving the $L^p$-boundedness of $W_\pm(H,\Delta^2)$, especially in the lower energy part. This will become clear through the stationary formula \eqref{stationary} of $W_\pm(H,\Delta^2)$, which will be presented later in Section \ref{regular3}-Section \ref{The third kind6}.

 Now we state our main results.
Let $\mathbb B(X,Y)$ be the space of bounded operators from Banach space $X$ to Banach space $Y$, namely $A\in \mathbb B(X,Y)$ if $ \|Af\|_{Y}\lesssim \|f\|_X$ for $ f\in X$ (or $f$ in some dense subset of $X$), and also set $\mathbb B(X)=\mathbb B(X,X)$.

  \begin{theorem}\label{theorem1.2}
  	Let  $H=\Delta^2+V$ with $|V(x)|\lesssim \<x\>^{-\mu}$ with some $\mu>0$ depending on the type of zero resonance.  Assume that $H$ has no positive embedded eigenvalues.
  	
  {\rm 	(i)}  If zero is  the first kind resonance of $H$ and $\mu>13$,  then  the wave operators $W_{\pm}(H,\Delta^2) \in \mathbb{B}(L^p(\R^3))$ for all $1<p<\infty$.

  	  {\rm (ii)}  If zero is the second kind resonance of $H$ and $\mu>21$,  then   $W_{\pm}(H,\Delta^2) \in \mathbb{B}(L^p(\R^3))$ for all $1<p<3$ and  $W_{\pm}(H,\Delta^2) \notin \mathbb{B}(L^p(\R^3))$ for any  $3\le p\leq\infty$.
  	
  	    {\rm (iii)}  If zero is the third kind resonance of $H$  (i.e. $H$ has a zero eigenvalue) and $\mu>25$,  then   $W_{\pm}(H,\Delta^2) \in \mathbb{B}(L^p(\R^3))$ for all $1<p<3$. Furthermore,  we have
  	    \begin{itemize}
  	  \item  If  $S_2L^2= S_3L^2=S_4L^2$, then $W_{\pm}(H,\Delta^2) \in \mathbb{B}(L^p(\R^3))$ for all  $3\le p<\infty$;
  	  \vskip0.2cm
  	     \item Otherwise, if  $S_2L^2\neq S_3L^2$\  or \ $S_3L^2\neq  S_4L^2$, then
  	     $W_{\pm}(H,\Delta^2) \notin \mathbb{B}(L^p(\R^3))$ for any $3\le p\leq\infty$.
  	
%  	     if and only if one of conditions that $S_3L^2\neq S_2L^2$ or $S_3L^2\neq S_4L^2$ holds,  That's $W_{\pm}(H,\Delta^2) \in \mathbb{B}(L^p(\R^3))$ for all  $1<p<\infty$ if and only if  $S_2L^2= S_3L^2=S_4L^2$.
  	     \end{itemize}
   Here $S_2L^2$, $ S_3L^2$ and $S_4L^2$ are suitable finite dimensional subspaces of $L^2(\R^3)$ defined in Definition \ref{definition_resonance} and \eqref{projectionS_4}  in Section 2 (see also Lemma \ref{lemma_resonance}) such that $S_4L^2 \subset S_3L^2 \subset S_2L^2$.
  	\end{theorem}

%\begin{color}{red}
\vskip0.3cm
%remark
\begin{remark}
\label{remark_orthogonal}
The condition $S_2L^2= S_3L^2=S_4L^2$ holds if and only if zero is an eigenvalue of $H$, but  not a $p$-wave resonance, and all zero eigenfunctions $\phi$ of $H$ satisfy
\begin{align}
\label{orthogonal}
\displaystyle \int_{\R^3}x_ix_jx_k V(x)\phi(x)dx=0
\end{align}
for all $i,j,k=1,2,3$. We refer to Lemma \ref{lemma_third_kind} for more details.
%\end{remark}
%\end{color}

%\vskip0.3cm
%\begin{remark}
	In particular, from  Lemma \ref{gongzhengkehua-1} and  Remark \ref{remark-eigenvalue},  it has been seen that the  $S_1L^2\neq S_2L^2$ or $S_2L^2\neq S_3L^2$ may hold, even for rapidly decaying potentials.

%Nevertheless, it remains uncertain for the authors whether there exists a nontrivial potential for which the relationship $S_2L^2= S_3L^2=S_4L^2$ holds. Equivalently, this  means  whether a potential $V$ can be found  such that each nontrivial solution of $H\phi=0$ in $W_{-1/2}(\R^3)$ must be bounded and satisfy that $\phi(x)=O(\<x\>^{-3})$ as $|x|$ goes to infinity.
	\end{remark}
\begin{remark}
We emphasize that the absence of embedded positive eigenvalues is a fundamental requirement for establishing $L^p$-bounds of wave operators. In fact, for any dimension $d\geq 1$, one can easily construct a potential $V\in C_0^\infty(\mathbb{R}^d)$ for which the operator $H=\Delta^2+V$ has some positive eigenvalues (see e.g. \cite[Section 7.1]{FSWY}). However, Feng et al. in \cite{FSWY} have proved that $H=\Delta^2 +V$ has no positive eigenvalues if $V\in C^1(\R^d)$ is bounded and satisfies the repulsive condition $$x\cdot \nabla V(x) \leq 0,\quad x\in \R^3.$$  Additionally, it is well-established by Kato \cite{Ka}, that the Schr\"odinger operator $-\Delta +V$ has no positive eigenvalues  when a bounded potential $V(x)=o(|x|^{-1})$ as  $|x|\to \infty$.
\end{remark}
As a consequence of Theorem \ref{theorem1.2} and \eqref{Lp-bound of f(H)}, we immediately obtain the following.

  \begin{corollary}Let $H$ and $V$ satisfy the same conditions as in Theorem \ref{theorem1.2} above.  Then
  	\begin{equation}\label{Schrodinger group}
  		\big\|e^{itH}P_{ac}(H)\big\|_{L^p(\R^3)\rightarrow L^{p'}(\R^3)}\lesssim |t|^{-\frac{3}{2}(\frac{1}{p}-\frac{1}{2})},
  	\end{equation}
  and
  \begin{equation}\label{wave group}
  	\big\|\cos(t\sqrt{H})P_{ac}(H)\big\|_{L^p(\R^3)\rightarrow L^{p'}(\R^3)}+\Big\|\frac{\sin(t\sqrt{H})}{t\sqrt{H}}P_{ac}(H)\Big\|_{L^p(\R^3)\rightarrow L^{p'}(\R^3)}\lesssim |t|^{-3(\frac{1}{p}-\frac{1}{2})},
  \end{equation}
  where $1<p\le 2$ in the regular or the first kind resonance cases and $\frac 32<p\le 2$ in the second or third kind resonance cases. Here $p'=p/(p-1)$ is the conjugate exponent of $p.$
  \end{corollary}

 When zero is a regular point or a resonance of the first kind, the decay estimates \eqref{Schrodinger group}  and \eqref{wave group} for all  $1\le p\le 2$, have been already established by Erdo\u gan et al.  \cite{Erdogan-Green-Toprak} and Chen et al \cite{CLSY} respectively, by different methods.  On the other hand, the estimates \eqref{Schrodinger group} and \eqref{wave group} for the case when zero is the second or third kind  resonances are new.

\subsection{Further related results}
\label{subsection_known_result}

For the classical Schr\"odinger operator $H=-\Delta+V(x)$, since the seminal work by Yajima in \cite{Yajima-JMSJ-95}, there has been a significant body of works focused on the $L^p$-boundedness of the wave operators $W_\pm$. This research has gained great importance due to its applications in linear and nonlinear dispersive equations. In particular, in the one-dimensional space ($d=1$), the wave operators $W_\pm$ are bounded on $L^p(\mathbb{R})$ for $1<p<\infty$ in both regular and zero resonance cases. However, they are generally unbounded for $p=1$ and $p=\infty$ (see, for example, \cite{ArYa,DaFa,Weder}).
In the regular case, for dimension $d=2$, the wave operators $W_\pm$ are bounded on $L^p(\mathbb{R}^2)$ for $1<p<\infty$, but the results at the endpoints are yet unknown (see \cite{Yajima-CMP-99, Jensen_Yajima_2D}). For dimensions $d\geq3$, the wave operators $W_\pm$ are bounded on $L^p(\mathbb{R}^d)$ for $1\leq p\leq\infty$ in the regular case (see, for example, \cite{BeSc,Yajima-1995,Yajima-JMSJ-95}). However, the presence of threshold resonances narrows down the range of values for $p$, depending on the dimension $d$ and asymptotic behaviors of zero resonant states and zero eigenfunctions (as discussed in \cite{EGG,Finco_Yajima_II,Goldberg-Green-Advance,Goldberg-Green-Poincare,Jensen_Yajima_4D, Yajima_2006,Yajima_2016,Yajima_2018,Yajima_2021arxiv,Yajima_2022arxiv}).

For the higher-order Schr\"odinger operator $H=(-\Delta)^m+V(x)$ on $\R^d$ with $m>1$,  many progresses have been made in recent years.  The first result in this direction is due to Goldberg--Green \cite{GoGr21} for the case $(m,d)=(2,3)$, where the $L^p$-boundedness of wave operators was proved for  $1<p<\infty$ if the zero energy is a regular point.  Additionally, for the case of $(m,d)=(2,1)$, Mizutani--Wan--Yao \cite{MiWYa} have obtained the $L^p$-boundedness for all $1<p<\infty$ in the  all kinds of resonances and counterexamples  at the endpoint $p=1,\infty$, as well as  weak-boundedness in the framework of $L^{1,\infty}$, $\H^1$ and $\BMO$. For $d>2m\ge4$, Erdo\u{g}an--Green \cite{Erdogan-Green21,Erdogan-Green22} proved the $L^p$-boundedness for all $1\le p\le \infty$ if the zero energy is a regular point and the potential $V(x)$ is sufficiently smooth.  Furthermore, for the case $d>4m-1$, Erdo\u{g}an--Goldberg--Green \cite{EGG23} provides examples of compactly supported non-smooth potential $V(x)$ for which the wave operators are not bounded on $L^p$ if $2d/(d-4m+1)<p\le \infty$. More recently, the case $d=2m=4$ was considered by Galtbayar--Yajima \cite{Yajima_2024JST} where the $L^p$-boundedness was proved for $1<p<p_0$ with suitable $p_0$ depending on the type of the singularity at the zero energy.

It can be observed from these works that the behavior of wave operators are roughly classified into three cases: $d<2m$, $d=2m$ and $d>2m$. When $d<2m$, as observed by \cite{GoGr21}, the resolvent has a singularity at the zero energy even in the free case and singular integrals similar to  Hilbert transform are appeared in the stationary representation of the low energy part of wave operators even in the regular case. It thus can be expect that the wave operators are generically not bounded on $L^p$ for $p=1,\infty$ in this case (see e.g. Mizutani--Wan--Yao \cite{MWY23}). On the other hand, when $d>2m$, the singularity at the zero energy of the resolvent is relatively mild, but the high energy part becomes  more complicated than the case $d<2m$ since the resolvent does not decay (or even can grow in higher space dimensions) in the high energy limit. The case $d=2m$ is critical in the sense that it has these difficulties in the low and high energy parts of the wave operators.
In this paper, our  focus is to further establish the $L^p$-bounds of wave operators  that exhibit all types of zero energy singularities  for the case $(m,d)=(2,3)$.

\subsection{Notations}
\label{subsection_notation}
Throughout the paper we use the following notations:
\begin{itemize}
	\item $A\lesssim B$ (resp. $A\gtrsim B$) means $A\le CB$ (resp. $A\ge CB$) with some constant $C>0$.
	\vskip0.2cm
	\item $L^p=L^p(\R^3),L^{1,\infty}=L^{1,\infty}(\R^3)$ denote the Lebesgue and weak $L^1$ spaces on $\R^3$, respectively.
	\vskip0.2cm
	\item $\<f,g\>_{L^2}=\int_{\R^3} f\overline gdx$ denotes the inner product in $L^2(\R^3)$.
\vskip0.1cm
\item $\mathcal{S}(\R^3)$ denotes the Schwartz class.
\vskip0.1cm
\item  The Fourier transform  of $f$ is  defined as
\begin{equation}\label{Fourier}\widehat{f}(\xi)=\mathcal{F}(f)(\xi)=\frac{1}{(2\pi)^{3}}\int_{\R^3}e^{-ix\xi}f(x)dx,\quad \xi\in \R^3. 
\end{equation}
We denote the inverse Fourier transform of $f$ by $\check{f}(\xi)$ or $\mathcal{F}^{-1}(f)(\xi)$.
\vskip0.1cm
\item  Let
$D=(D_1, D_2, D_3)$ denote the gradient operator, where $D_l=-i\partial/\partial x_l$ for $l=1,2,3.$ Moreover, $m(D)=\F^{-1}m(\xi)\F$ denotes the Fourier multiplier operator with symbol $m(\xi)$. 

\item Let $\widehat{\mathcal{R}_jf}(\xi)=\frac{\xi_j}{|\xi|} \widehat{f}(\xi) $ with $ j=1,2, 3$ be the Riesz transforms, and set $\mathcal R=(\mathcal R_1,\mathcal R_2,\mathcal R_3)$. Moreover, we set for short
\begin{align*}
	\mathcal{R}_{kl}=\mathcal{R}_k\mathcal{R}_l,\quad \mathcal{R}_{kls}=\mathcal{R}_k\mathcal{R}_l\mathcal{R}_s,
	\quad \mathcal{R}_{klst}=\mathcal{R}_k\mathcal{R}_l\mathcal{R}_s\mathcal{R}_t,	\quad \mathcal{R}_{klstm}=\mathcal{R}_k\mathcal{R}_l\mathcal{R}_s\mathcal{R}_t\mathcal{R}_m.
\end{align*}

\vskip0.1cm
%\item 
%We use the notations  $z=(z_1,z_2,z_3)\in \R^3,$\ \  $\omega=(\omega_1,\omega_2,\omega_3)\in S^2$  and the Riesz transforms $\mathcal R=(\mathcal R_1,\mathcal R_2,\mathcal R_3)$, where $S^2$ is the unit sphere of $\R^3$ and  the Riesz transform is defined by 
%$\widehat{\mathcal{R}_jf}(\xi)=\frac{\xi_j}{|\xi|} \widehat{f}(\xi) $ with $ j=1,2, 3.$
\item Let $v(x)=|V(x)|^{1/2}$ and $U(x)=\sgn V(x)$, i.e., $U(x)=1$ if $V(x)\ge0$, $U(x)=0$ if $V(x)<0$. 
\vskip0.1cm
	\item 
For $z=(z_1,z_2,z_3)\in \R^3$ and $k,l,s,t,m=1,2,3$, we set for short
\begin{align*}
	z_{kl}&=z_kz_l,\quad  z_{kls}=z_kz_lz_s,\quad  z_{klst}=z_kz_lz_sz_t,\quad  z_{klstm}=z_kz_lz_sz_tz_m;\\
	v_{k}(z)&=z_kv(z),\quad v_{kl}(z)=z_{kl}v(z),\quad v_{kls}(z)=z_{kls}v(z),\quad\\
	v_{klst}(z)&=z_{klst}v(z),\quad v_{klstm}(z)=z_{klstm}v(z).
\end{align*}
\vskip0.1cm
\item
For $\theta>0$ and $\rho\in \R^3$, we define the dilation operator $U_\theta$ by  
$
U_\theta u(x) = \theta^{-3} u(\theta^{-1} x)
$,  and the translation operator $J_\rho$ by  
$
J_\rho u(x) = u(x - \rho).  
$
\vskip0.1cm
\item The operators $P,Q,T,T_j,S_j$ and $\D_j$ are defined in Definition \ref{definition_resonance}. 
\vskip0.1cm
 \item A bounded operator $A\in \mathbb B(L^2)$ is said to be absolutely bounded ($A\in \AB(L^2)$ for short) if $$A=M_f+K$$ is a sum of a multiplication $M_f \in \mathbb B(L^2)$ by some $f\in L^\infty$  and an integral operator $K\in  \mathbb B(L^2)$ with kernel $K(x,y)$ such that the integral operator $|K|$ with kernel $|K(x,y)|$ is also bounded on $L^2$. Note that $A\in \AB(L^2)$ if  $A$ is a Hilbert--Schmidt operator on $L^2$. 
% An integral operator $T_K\in \mathbb B(L^2(\R^3))$ with kernel $K(x,y)$ is said to be absolutely bounded on $L^2(\R^3)$ if the integral operator $T_{|K|}$ with kernel $|K(x,y)|$ is also bounded on $L^2(\R^3)$.
\vskip0.1cm
\item   Let $\chi\in C_0^\infty (\R)$ be such that $0\le \chi(\lambda)\le1$, $\chi(\lambda)=1$ for $|\lambda|\le 1/2$ and $\chi(\lambda)=0$ for $|\lambda|\ge 1$.  For  any $a>0$, we set 
$$
\quad  \quad  \chi_{<a}(\lambda)=\chi(\lambda/a),\quad \chi_{\ge  a}(\lambda)=1-\chi_{<a}(\lambda),\quad \widetilde \chi_{<a}(x)=\chi_{<a}(|x|),\quad \widetilde \chi_{\ge a}(x)=\chi_{\ge a}(|x|).
$$
\end{itemize}

%remark\begin{remark}The above definition of  absolutely bounded operators is slightly different from the original one by Schlag \cite{Schlag_CMP2005} where \end{remark}

%\subsection{The outline of the proofs}

\section{The stationary formulas and asymptotic expansions}
\label{section2}
\subsection{The stationary formulas of wave operators}
\label{section_stationary}

First of all, we observe that it suffices to deal with $W_-$ only since \eqref{def-wave} implies $W_+f=\overline{W_-\overline f}$.
The starting point is the following well-known stationary representation of $W_-$ (see e.g. Kuroda \cite{Kuroda}):
\begin{align}
	\label{stationary}
	W_-=\Id-\frac{2}{\pi i}\int_0^\infty \lambda^3 R_V^+(\lambda^4)V\left(R_0^+(\lambda^4)-R_0^-(\lambda^4)\right)d\lambda.
\end{align}
To explain the formula \eqref{stationary}, we need to introduce some notations.  Let
\begin{align*}
	R_0(z)=(\Delta^2-z)^{-1},\quad
	R_V(z)=(H-z)^{-1},\quad z\in \C\setminus[0,\infty),
\end{align*}
be the resolvents of $\Delta^2$ and $H=\Delta^2+V(x)$, respectively. Then  $R^\pm_0(\lambda),R^\pm_V(\lambda)$  are defined as their proper boundary values (limiting resolvents) on $(0,\infty)$, namely
$$
R^\pm_0(\lambda)=\lim_{\ep \searrow 0}R_0(\lambda\pm i\ep),\quad R_V^\pm(\lambda)=\lim_{\ep \searrow 0}R_V(\lambda\pm i\ep),\quad \lambda>0.
$$
More specifically, the existence of $R^\pm_0(\lambda)$ as bounded operators from $L^2_s(\R^3)$ to $L^2_{-s}(\R^3)$ with $s>1/2$ follows from the limiting absorption principle for the resolvent $(-\Delta-z)^{-1}$ of the free Schr\"odinger operator $-\Delta$ (see e.g.  Agmon \cite{Agmon}) and the following equality:
\begin{equation}\label{decomposition}
R_0(z)=\frac{1}{2\sqrt z}\left[(-\Delta-\sqrt z)^{-1}-(-\Delta+\sqrt z)^{-1}\right],\quad z\in \C\setminus[0,\infty), \ \Im\sqrt z>0.
\end{equation}
%which is obtained by the identity $\Delta^2-z=(-\Delta-\sqrt z)(-\Delta+\sqrt z)$ and the first resolvent equation.
The existence of $R^\pm_V(\lambda)$ for $\lambda>0$ under the assumption of Theorem \ref{theorem1.2} has been also already shown (see  e.g. \cite{Agmon,FSY18, Kuroda}).

Formula \eqref{decomposition} also leads to the explicit  kernels of free resolvent $R_0^\pm(\lambda^4)$ of $\Delta^2$:
\begin{align}
	\label{free_resolvent}
	R_0^\pm(\lambda^4,x,y)=\frac{1}{8\pi\lambda^2|x-y|}\Big(e^{\pm i\lambda|x-y|}-e^{-\lambda|x-y|}\Big)=\frac{F_\pm(\lambda|x-y|)}{8\pi\lambda},
\end{align}
where $x, y\in \R^3$ and $F_\pm(s)=s^{-1}(e^{\pm is}-e^{-s})$.   As a consequence,  for each $\lambda>0$,
\begin{align}
\nonumber
	(R_0^+(\lambda^4)-	R_0^-(\lambda^4))f (x)
	&=\frac{i}{4\pi\lambda^2}\int_{\R^3}\frac{\sin\lambda|x-y|}{|x-y|}\ f(y)dx\\
\label{free_resolvent2}
	&=\frac{\pi i}{2\lambda}\int_{S^2} e^{i\lambda\omega x} \widehat{ f}(\lambda\omega)d\omega,
\end{align}
where $S^2$ is the unit sphere of $\R^3$. \eqref{free_resolvent2} immediately gives the following commuting property:
\begin{align}\label{free_resolvent4}
m(|D|)(R_0^+(\lambda^4)-	R_0^-(\lambda^4))f=(R_0^+(\lambda^4)-	R_0^-(\lambda^4))m(|D|)f= m(\lambda)(R_0^+(\lambda^4)-R_0^-(\lambda^4))f
\end{align}
for $f\in \mathcal S(\R^3)$ and thus, by the density argument, for all $f\in \D(m(|D|))$, the domain of $m(|D|)$, where  $m:(0,\infty)\to \C$ is piecewise continuous and of at most polynomial growth.

%\vskip0.5cm
\subsection{Some integral operators}
In this subsection, we give some lemmas on the $L^p$ boundedness of  some integral operators related with the wave operators $W_\pm$.  We first recall  the  classical Schur's test,  which says that  an operator $K \in \mathbb B\big(L^p(\R^3)\big)$ for all $1\le p\le \infty$ if its kernel $K(x,y)$ satisfies
\begin{align}\label{Schur test}
\sup_{x\in \R^3} \int_{\R^3}|K(x,y)|dy+\sup_{y\in \R^3} \int_{\R^3}|K(x,y)|dx<\infty.
\end{align}
In many situations faced here,  the kernels $K(x,y)$ decays slowly in $x$ or $y$ such that $K(x,y)$ \eqref{Schur test} does not satisfy \eqref{Schur test}  in which case we need to study these integral operators case by case. 

The next lemma will be frequently used in this paper. One can find its proof in the proof of \cite[Lemma 3.3]{GoGr21} (see pages 7--9 of \cite{GoGr21}). We also refer to \cite[Proposition 2.4 and Lemma 5.2]{MWY23}  for a different proof. % Although the statement is  contained in the proofs of \cite[Lemma 3.3]{GoGr21} and \cite[Proposition 5.2]{MWY23}, we provide a proof in Section \ref{subsection 2.1} for the sake of completeness.

\begin{lemma}\label{lem2.1}
	Let $z, w\in \R^3$ and $A_{z,w}$ be the integral operator with kernel 
\begin{align}
\label{Azw}	
A_{z, w}(x,y)=\int_0^\infty  \lambda^4 \chi_{<a}(\lambda)R_0^+(\lambda^4,x,z) \left (R_0^+(\lambda^4,w,y)-R_0^-(\lambda^4,w,y)\right)d\lambda, \ \ x,y\in\R^3. 
	\end{align}
Then  $A_{z,w}\in \mathbb{B}(L^p(\R^3))$ for any $z,w\in \R^3$ and $1<p<\infty$, and $$\sup_{z,w}\|A_{z,w}\|_{\mathbb{B}(L^p)}=\|A_{0,0}\|_{\mathbb{B}(L^p)}<\infty. $$ 
	\end{lemma}
		
%By   \eqref{free_resolvent}  and using translation invariance of $L^p$-norm,  we can first reduce the boundedness of $A_{z,w}$ into the bounds of $A_{0,0}$ with $z=w=0$, and then obtain the following type integral:
%\begin{align}\label{integral-A}
%	A_{0,0}(x,y)=\frac{1}{64\pi^2}\int_0^\infty \chi_{<a}(\lambda)\Big( \frac{e^{i\lambda|x|}-e^{-\lambda|x|}}{|x|}\Big)\Big( \frac{e^{i\lambda|y|}-e^{-i\lambda|y|}}{|y|}\Big)d\lambda,
%\end{align}
%which satisfies  the following estimates:
%$$64\pi^2A_{0,0}(x,y)=-4i\frac{|x|\ \chi_{\{||x|-|y||\ge1\}}}{|x|^4-|y|^4}+\Psi(x,y),\ \ x,y\in \R^3,$$
%with $\Psi(x,y)$ satisfying  Schur test condition \eqref{Schur test}.   Note that the kernel $\frac{|x|\ \chi_{\{||x|-|y||\ge1\}}}{|x|^4-|y|^4}$ does not  satisfy Schur test condition \eqref{Schur test},  we need to use  Calderon-Zygmund singular theory to prove the $L^p$ boundedness of $A_{0,0}$ for all $1<p<\infty$. Here we left the  proof of Lemma \ref{lem2.1} to Section \ref{Section_technical_lemma}  below.

 Besides, we also need to estimate the following two Fourier type integrals:
\begin{align}\label{E_0}
	E_0(x,y)=\frac{1}{(2\pi)^6}\int_{\R^6}e^{ix\xi-iy\eta}\ \frac{\widetilde \chi_{<4a}(\xi)\widetilde \chi_{<a}(\eta)}{(|\xi|^2+|\eta|^2)(|\xi|+|\eta|)}d\xi d\eta,
\end{align}
\begin{align}\label{E_1}
	E_1(x,y)=\frac{1}{(2\pi)^6}\int_{\R^6}e^{ix\xi-iy\eta}\ \frac{\widetilde \chi_{<4a}(\xi)\widetilde \chi_{<a}(\eta)}{(|\xi|^2+|\eta|^2)|\eta|}d\xi d\eta.
\end{align}

\begin{lemma}\label{lem2.2} The above $E_0(x,y)$ and $E_1(x,y)$ satisfy the following estimates:
	\begin{align}\label{eq2.12}
		\big|E_0(x,y)\big|\lesssim \< |x|+|y|\>^{-3},\\
	%\end{align}
%	\begin{align}
\label{eq2.13}
	\big|E_1(x,y)\big|\lesssim \<x\> ^{-1}\< |x|+|y|\>^{-2}.
\end{align}
As  a consequence,  $E_0\in \mathbb{B}(L^p(\R^3))$ for all  $1<p<\infty$  and $E_1\in \mathbb{B}(L^p(\R^3))$ for all  $1<p<3$. 
\end{lemma}

%We observe that the symbols in \eqref{E_0} and \eqref{E_1} do not exhibit smooth properties on the hyperplane set $\{(\xi, \eta) \in \R^6; \ \eta=0\}$. This, in turn, indicates potential technical challenges when applying conventional methods, such as Littlewood-Paley theory. However, the multiplication property of Fourier transform enables us to transform the target integrals into convolution integrals, facilitating the establishment of the desired estimates.

 The proof of Lemma \ref{lem2.2} is given in Section \ref{Section_technical_lemma}.
Remarkably, the estimates \eqref{eq2.12} and \eqref{eq2.13} are  optimal, because the $L^p$-boundedness of $T_K$ and $T_E$ provides the precise range for the $L^p$-boundedness of wave operators, as asserted in Theorem \ref{theorem1.2}.

Finally, we  collect some facts  on Fourier multiplier $m(D)$,  see  e.g. Grafakos \cite{Grafakos_Classical_III}.

 \begin{lemma}[Mikhlin's theorem]\label{lem2.3}
Let $k>d/2$ and  $m\in C^k(\R^d\setminus \{0\})$ satisfy the following estimates:
 	$$|\partial^\alpha_\xi m(\xi)|\le C_\alpha |\xi|^{-|\alpha|},\ \xi\neq 0,\  |\alpha|\le k.$$
 	Then $m(D)\in \mathbb{B}(L^p(\R^d))$ for all $1<p<\infty$.
 	 In particular,  Riesz transforms $\mathcal{R}=(\mathcal{R}_1, \mathcal{R}_2,\cdots, \mathcal{R}_d)$ are bounded on $L^p(\R^d)$ for all $1<p<\infty$, where
 	 $\widehat{\mathcal{R}_jf}(\xi)=\frac{\xi_j}{|\xi|} \widehat{f}(\xi) $ with $1 \le j\le d$.
 \end{lemma}

\subsection{A characterization of zero resonances}
\label{subsection_resonance}
Firstly, we provide some notations:  Let $$v(x)=|V(x)|^{1/2},\quad U(x)=\sgn V(x)=\begin{cases}1&\text{if}\ V(x)\ge0,\\0&\text{if}\ V(x)<0.\end{cases} $$ 
We also let $G_j$ with $j=1,3,4$  be the integral operators with kernels
$$
G_1(x,y)=|x-y|^2,\quad G_3(x,y)=|x-y|^4,\quad G_4(x,y)=-\frac{|x-y|^5}{4\pi\cdot 6!}
$$
Next  we recall an equivalent characterization of zero resonances of $H$ due to \cite{Erdogan-Green-Toprak}: %  by the kernel spaces of the following operators $T_j$ for $j=0,1,2$.
%%%%%%%%%%%%%%%%%%%%%%%%%%%%%%%%%%%%%%%%%%%%%%%%%%%%%%%%%%%%%%%%%%
\begin{definition}[{\cite[Definition 4.2]{Erdogan-Green-Toprak}}]
\label{definition_resonance}Define
\begin{align}
\label{PQT}
P:=\norm{V}_{L^1}^{-1}\<\cdot,v\>v,\quad Q:=I-P,\quad T:=U+v(\Delta^2)^{-1}v. 
\end{align}
	\begin{itemize}
		\item[(i)] We say that zero is a regular point  of $H$ if $T_0:=QTQ: QL^2\to QL^2$ is invertible on $QL^2(\mathbb{R}^3)$.  In this case, we define $\D_0:=T_0^{-1}:QL^2\to QL^2$. % as an operator on  $QL^2(\mathbb{R}^3)$.
		\vskip0.2cm
		\item[(ii)] Assume that $T_0$ is not invertible on $QL^2(\mathbb{R}^3).$ Let $S_1$ be the orthogonal projection onto the kernel of $T_0$.  Then $T_0+S_1$ is invertible on $QL^2(\mathbb{R}^3)$.  In this case, we define $\D_0=\big(T_0+S_1\big)^{-1}$ as an operator on $QL^2(\mathbb{R}^3)$, which does not conflict with the previous definition since $S_1=0$ in the regular case.
		Then we say zero is the first kind resonance of $H$ if
		\begin{equation}\label{T_1}
			T_1:= S_1TPTS_1-\frac{\|V\|_{L^1}}{3\cdot (8\pi)^2}S_1vG_1vS_1,
		\end{equation}
		is invertible on $S_1L^2(\mathbb{R}^3)$. We define $\D_1=T_1^{-1}$ as an operator on $S_1L^2(\mathbb{R}^3)$.
		\vskip0.2cm
		\item[(iii)] Assume $T_1$ is not invertible on $S_1L^2(\mathbb{R}^3).$ Let $S_2$ be the orthogonal  projection onto the kernel of $T_1.$ Then $T_1+S_2$ is invertible on $S_1L^2(\mathbb{R}^3).$ In this case, we define $\D_1=\big(T_1+S_2\big)^{-1}$ as an operator on $S_1L^2(\mathbb{R}^3)$.
%which does not conflict with the previous	definition since $S_2=0$ at  the first kind  resonance case.
		Then we say zero is the second kind  resonance of $H$  if
		\begin{equation}\label{T_2}
			T_2:= S_2vG_3vS_2+\frac{10}{3\|V\|_{L^1}} S_2(vG_1v)^2S_2-\frac{10}{3\|V\|_{L^1}}S_2vG_1vT\D_1TvG_1vS_2,
		\end{equation}
		is invertible on $S_2L^2(\mathbb{R}^3)$. We define $\D_2=T_2^{-1}$ as an operator on $S_2L^2(\mathbb{R}^3)$.
		\vskip0.2cm
		\item[(iv)] Finally, we say zero is the third kind  resonance of $H$ if $T_2$ is not invertible on $S_2L^2(\mathbb{R}^3)$. 
		In this case, the operator $T_3:=S_3vG_4vS_3$ is always invertible on $S_3L^2(\mathbb{R}^3)$ (see \cite[Lemma 7.6]{Erdogan-Green-Toprak}), where $S_3$ is the orthogonal  projection onto the kernel of $T_2.$ Let $\D_3=T_3^{-1}$ as an operator $T_3$ on $S_3L^2(\mathbb{R}^3)$. We define $\D_2=(T_2+S_3)^{-1}$ as an operator on $S_2L^2(\mathbb{R}^3)$. %,		which does not conflict with the previous definition since $S_3=0$ at the second kind  resonance case.
	\end{itemize}
\end{definition}

Remark that all of $T$ and $T_j$ are bounded self-adjoint operators on $L^2(\R^3)$ and so are $\D_j$ (if they exist) by the closed graph theorem. 

In shorthand, we have the following characterization of zero resonances of $H$: 
\begin{itemize}
	\item ~ Zero is a regular point of $H$ if and only if $S_1L^2=\{0\};$
	\item~  Zero is the first kind resonance of $H$ if and only if $S_1L^2\neq\{0\}$ and $S_2L^2=\{0\};$
	\item~ Zero is the second kind resonance of $H$ if and only if $S_2L^2\neq\{0\}$ and $S_3L^2=\{0\};$
	\item~ Zero is the third kind resonance of $H$ if and only if $S_3L^2\neq\{0\}.$
\end{itemize}
Since $(\Delta^2)^{-1}(x,y)=(-8\pi)^{-1}|x-y|$ (see Appendix \ref{Appendix A} below) and hence $$\left|v(x)(\Delta^2)^{-1}(x,y)v(y)\right|^2\lesssim \<x\>^{-\mu+2}\<y\>^{-\mu+2},$$  $v(\Delta^2)^{-1}v$ is a Hilbert--Schmidt operator on $L^2$ and hence $T$ is a compact perturbation of $U$ if $\mu>5$. Then the Fredholm alternative theorem ensures $S_1$ is finite rank and so are $S_2$ and $S_3$. Moreover, by definition, $S_jL^2=\Ker (T_{j-1})$ for $j=1,...,3$. The following lemma provides an explicit characterization of closed subspaces $QL^2$ and $S_jL^2$ of $L^2$. And more details about this lemma can be seen  in
{\cite [Section 7]{Erdogan-Green-Toprak}}.
%Using this and the following lemma, we can see that the definitions of $S_j$ and $S_jL^2$ given by \eqref{projectionQ}--\eqref{projectionS_3} in Section \ref{subsection_resolvent} are consistent with the above definitions in Definition \ref{definition_resonance}. 
%%%%%%%%%%%%%%%%%%%%%%%%%%%%%%%%%%%%%%%%%%%%%%%%%%%%%%%
\begin{lemma}\label{lemma_resonance}We have
	\begin{itemize}
	\item[(i)] $QL^2(\R^3)=\{f\in L^2\ | \ \langle v, f\rangle =0 \}$. 
		\item[(ii)]$S_1L^2(\R^3)=\{f \in QL^2(\R^3)\ | \ QTf=0\}$. As a result, $QTS_1= S_1TQ=0$.
		\item[(iii)]$S_2L^2(\R^3)=\{ f\in S_1L^2(\mathbb{R}^3)\ |\ PTf=0,\ \langle x_iv, f\rangle =0,\ i=1,2,3 \}$. In particular, $$TS_2=S_2T=0,\quad QvG_1vS_2=S_2vG_1vQ=0.$$
		\item[(iv)]$S_3L^2(\R^3)=
		\{ f\in S_2L^2(\mathbb{R}^3)\ |\ \langle x_ix_jv, f\rangle =0,\ i, j=1,2,3\}$.
%		\item[(iv)] $\mathop{\mathrm{Ker}} (T_3 )= \{ 0 \}$. That is,  $T_3=S_3vG_4vS_3$ is invertible on $S_3L^2(\mathbb{R}^3)$.
	\end{itemize}
\end{lemma}

%\begin{align}\label{projectionQ}QL^2&=\{f\in L^2\ | \ \langle v, f\rangle =0 \},\\\label{projectionS_1} S_1L^2&=\{  f\in  QL^2\ | \ QTf=0\},\\\label{projectionS_2}S_2L^2&=\{ f\in S_1L^2\ |\ PTf=0,\quad \langle x_iv, f\rangle =0,\quad i=1,2,3 \},\\\label{projectionS_3}S_3L^2&=\{ f\in S_2L^2\ | \ \langle x_ix_jv, f\rangle =0,\quad i, j=1,2,3\},\\\label{projectionS_4}S_4L^2&=\{ f\in S_3L^2\ | \ \langle x_ix_jx_kv, f\rangle =0,\quad i, j,k=1,2,3\}, \end{align}where \begin{align}\label{PQT}P:=\norm{V}_{L^1}^{-1}\<\cdot,v\>v,\quad Q:=I-P,\quad T:=U+v\ (\Delta^2)^{-1}v. \end{align}Let $S_j$ be the orthogonal (Riesz) projection onto $S_jL^2$ for each $j=1,...,4$, respectively 

Next, we present the following Lemma \ref{DTv} , which is crucial for proving the unboundedness results in Sections \ref{section 5.2-un} and \ref{section 6.2-un}.
\begin{lemma}\label{DTv}
	Let $T = U + v(\Delta^2)^{-1}v$ and $\D_1$ be the inverse of $T_1 + S_2$ on $S_1L^2(\mathbb{R}^3)$, where  $T_1$ and $S_2$  are given in Definition \ref{definition_resonance}. Then, the following inequality holds:
	$$
	1 - \|V\|_{L^1}^{-1}{\langle \D_1Tv, Tv \rangle} \geq 0.
	$$
\end{lemma}
The proof of Lemma \ref{DTv} is essentially similar to show
  $1-\frac{\<(U_1T_1U_1)^{-1}U_1Tv,\   U_1Tv\>}{\|V\|_{L^1}}\geq0,$ in \cite[Lemma 4.6--4.7]{Erdogan-Green-Toprak}.  
 In particular, $\|V\|_{L^1}^{-1},$ $A$ and  $B$ below correspond to $\alpha,$ $\mathcal{S}$ and  $\mathcal{T}$  in \cite[Lemma 4.7]{Erdogan-Green-Toprak}, respectively.
 \begin{proof}
 	Denote $\eta:=D_1Tv$ and $\xi:=S_1Tv.$ Then by $D_1=S_1D_1$ from \eqref{orthog-relation-1}, one has 
 	\begin{align*}
 		1 - \frac{\langle \D_1Tv, Tv \rangle}{\|V\|_{L^1}}=	1 - \frac{\langle S_1\D_1Tv, Tv \rangle}{\|V\|_{L^1}}=	1 - \frac{\langle \D_1Tv, S_1Tv \rangle}{\|V\|_{L^1}}=1 - \frac{\langle \eta, \xi \rangle}{\|V\|_{L^1}}.
 	\end{align*}
 	Hence, it suffices to prove that $1 - \|V\|_{L^1}^{-1}{\langle \eta, \xi \rangle}\geq0.$  By \eqref{T_1}, it follows that $T_1+S_2=A+B,$
 where  	operators $A$ and $B$ are given by
 		\begin{align*}
 		A= \Big(-\frac{\|V\|_{L^1}}{3\cdot (8\pi)^2}S_1vG_1vS_1+S_2\Big),\quad B=S_1TPTS_1.
 	\end{align*}
 	
 	Next we claim that the  operator $A$ is nonnegative, i.e. $\<Af, f\>\geq0$ for any $f\in L^2.$ Indeed 
 		\begin{align*}
 		\<Af, f\>= -\frac{\|V\|_{L^1}}{3\cdot (8\pi)^2}\<S_1vG_1vS_1f,  f\>+\<S_2f, f\>.
 	\end{align*}
 	Since $S_2$ is the orthogonal  projection, then $\<S_2f, f\>=\|S_2f\|^2\geq0.$ Furthermore, applying  the kernel $G_1(x,y)=|x-y|^2$ given by \eqref{def-Gk}, we obtain that
 	\begin{align*} 
 	\<S_1vG_1vS_1f,  f\>&=	\<G_1vS_1f, vS_1f\>\\
 	&=\int_{\R^3}\int_{\R^3}\big( |x|^2-2x\cdot y+|y|^2\big)v(y)(S_1f)(y)dy\big(v(x)\overline{S_1f}(x)\big)dx\\
 	&=\< |x|^2v, S_1f\>\<  S_1f, v\>+\<  S_1f, |y|^2v\>\<  v, S_1f\>-2\sum_{j=1}^3\<  S_1f, y_jv\>\< x_jv, S_1f\>.
 	\end{align*}
 	Note that $\<  S_1f, v\>=0$ from the cancellation property $v\perp S_1L^2$ (see Lemma \ref{lemma_resonance}). Thus,
 		\begin{align*} 
 		\<S_1vG_1vS_1f,  f\>
 		=-2\sum_{j=1}^3\<  S_1f, y_jv\>\< x_jv, S_1f\>=-2\sum_{j=1}^3|\< x_jv, S_1f\>|^2.
 	\end{align*}
As a result,  it follows that 
 	\begin{align*}
 	\<Af, f\>= \frac{2\|V\|_{L^1}}{3\cdot (8\pi)^2}\sum_{j=1}^3|\< x_jv, S_1f\>|^2+\|S_2f\|^2\geq0,\ \ \forall f\in L^2.
 \end{align*}
 
 Now, we show that $B\eta=\|V\|_{L^1}^{-1}\<\eta, \xi\>\xi,$ and $(T_1+S_2)\eta=\xi.$ In fact, since $B=S_1TPTS_1$ and $P=\norm{V}_{L^1}^{-1}\<\cdot,v\>v,$ then
 	\begin{align*}
 	B\eta= S_1TPTS_1\eta=\norm{V}_{L^1}^{-1}\<T_1S_1\eta,v\>S_1Tv=\norm{V}_{L^1}^{-1}\<\eta,S_1Tv\>S_1Tv=\norm{V}_{L^1}^{-1}\<\eta, \xi\>\xi.
 \end{align*}
Moreover,  taking into account that $\D_1$ is the inverse of $T_1 + S_2$ on $S_1L^2(\mathbb{R}^3),$ it concludes that 
 	\begin{align*}
 (T_1+S_2)\eta=(T_1+S_2)D_1Tv=S_1Tv=\xi.
 \end{align*}

 Considering that $A=T_1+S_2-B,$ $\<Af, f\>\geq0$ for any $f\in L^2,$ $B\eta=\|V\|_{L^1}^{-1}\<\eta, \xi\>\xi,$ and $(T_1+S_2)\eta=\xi$ all together, we get that
\begin{align*}
	0\leq\<\eta, A\eta\>&=\<\eta, (T_1+S_2)\eta-B\eta\>=\<\eta, \xi\>-\|V\|_{L^1}^{-1}\big|\<\eta, \xi\>\big|^2.
\end{align*}
 Then, $\<\eta, \xi\>\geq\|V\|_{L^1}^{-1}\big|\<\eta, \xi\>\big|^2\geq0.$ Thus,  we derive  that
 \begin{align*}
 \<\eta, \xi\>-\|V\|_{L^1}^{-1}\big|\<\eta, \xi\>\big|^2=\<\eta, \xi\>\Big(1-\|V\|_{L^1}^{-1}\<\eta, \xi\>\Big)\geq0,
 \end{align*}
 which implies that $1-\|V\|_{L^1}^{-1}\<\eta, \xi\>\geq0.$
 \end{proof}

So far, there are two definitions of zero resonances of $H$, Definitions \ref{definition1.1} and \ref{definition_resonance}.  These are in fact equivalent to each other as follows. Recall that $W_{-\sigma}(\mathbb{R}^3):=\bigcap_{s>\sigma}L^2_{-s}(\mathbb{R}^3)$. 

%remark\begin{remark}The subspaces $S_jL^2$ for $j=1,2,3$ can be characterized via the zero resonances and eigenfunctions $\phi$ in Definition \ref{definition1.1}. We refer to Lemma \ref{lemma_resonance} below  for details.\end{remark}

\begin{lemma}\label{gongzhengkehua-1}
Assume that $H=\Delta^2+V$ and  $|V(x)|\lesssim(1+|x|)^{-\beta}$ with $\beta>0$.
\begin{itemize}
\item[(i).] If $\beta> 11$, then $f\in S_1L^2\setminus \{0\}$ if and only if $f=Uv\phi$ with $0\neq \phi\in W_{-\frac{3}{2}}(\mathbb{R}^3)$  such that $H\phi=0$ in the distributional sense.
\item[(ii).] If $\beta> 19$, then $f\in S_2L^2\setminus \{0\}$ if and only if $f=Uv\phi$ with $0\neq\phi\in W_{-\frac{1}{2}}(\mathbb{R}^3)$  such that $H\phi=0$ in the distributional sense.
\item[(iii).] If $\beta> 23$, then $f\in S_3L^2\setminus \{0\}$ if and only if $f=Uv\phi$ with $0\neq\phi\in L^2$ such that $H\phi=0$ in the distributional sense.
\end{itemize}
\end{lemma}

\begin{proof}
An essentially same result, in terms of the $L^p$-spaces instead of weighted $L^2$-spaces $W_{-\sigma}$, has already been obtained by {\cite[Lemmas 7.1, 7.4 and 7.5]{Erdogan-Green-Toprak}}. Since the proof is also essentially analogous, we omit it. 
\end{proof}

Finally, we also introduce the following  closed subspace 
\begin{align}
\label{projectionS_4}
S_4L^2(\R^3):=\{ f\in S_3L^2\ | \ \langle x_ix_jx_kv, f\rangle =0,\quad i, j,k=1,2,3\}\subset L^2(\R^3). 
\end{align}
It is easy to see that $$S_4L^2\subset S_3L^2\subset S_2L^2\subset S_1L^2\subset QL^2.$$
Moreover, we can rephrase the condition $S_2L^2=S_3L^2=S_4L^2$ in Theorem \ref{theorem1.2} as in Remark \ref{remark_orthogonal}:
%remark
\begin{lemma}
\label{lemma_third_kind}
Suppose that zero is the third kind resonance of $H$. Then $S_2L^2=S_3L^2=S_4L^2$ if and only if $H$ has no $p$-wave zero resonance and \eqref{orthogonal} holds for all $i,j,k=1,2,3$.
\end{lemma}

\begin{proof}
By definition and Lemma \ref{gongzhengkehua-1} (i),(ii), the condition $S_2L^2=S_3L^2$ is nothing but the absence of  $p$-wave zero resonances. We next assume $S_3L^2=S_4L^2$ and take a zero eigenfunction $\phi$ of $H$. Then $f:=Uv\phi\in S_3L^2=S_4L^2$ by Lemma \ref{gongzhengkehua-1} (iii), so $\<x_ix_jx_kV,\phi\>=\<x_ix_jx_k v,f\>=0$ for all $i,j,k$. Conversely, if \eqref{orthogonal} holds for all $i,j,k$ and $f\in S_3L^2$, then there exists a zero eigenfunction $\phi$ by Lemma \ref{gongzhengkehua-1} (iii) so that $f=Uv\phi$ and $\<x_ix_jx_k v,f\>=\<x_ix_jx_kV,\phi\>=0$. This shows $S_3L^2=S_4L^2$ since $S_4L^2\subset S_3L^2$. 
\end{proof}

%vskip0.2cm
\subsection{Resolvent asymptotic  expansions}
\label{subsection_resolvent}

This subsection is  devoted to the study of asymptotic behaviors of the resolvent $R_V^\pm (\lambda^4)$ at low energy $\lambda\to 0^+$ when zero is a resonance of $H$.

 Recall $v(x)=|V(x)|^{1/2}$ and $U(x)=\sgn V(x)$, that is $U(x)=1$ if $V(x)\ge0$ and $U(x)=-1$ if $V(x)<0$. Let $M^\pm(\lambda)=U+vR_0^\pm(\lambda^4)v$ and $(M^{\pm})^{-1}(\lambda):=[M^\pm(\lambda)]^{-1}$.
%{lemma}
\begin{lemma}
	\label{lemma_2.1}
	For $\lambda>0$, $M^\pm(\lambda)$ is invertible on $L^2(\R^3)$ and $R_V^\pm(\lambda^4)V$ has the form
	\begin{align}
		\label{lemma_3_2_1}
		R_V^\pm(\lambda^4)V=R_0^\pm(\lambda^4)v (M^\pm(\lambda))^{-1}v,
	\end{align}

\end{lemma}

%proof
\begin{proof}
	Due to the absence of embedded positive eigenvalues of $H$, it is well-known that $M^\pm(\lambda)$ is invertible for all $\lambda>0$ (see Kuroda \cite{Kuroda}). Since $V=vUv$ and $1=U^2$, we have
	\begin{align*}
		R_V^\pm(\lambda^4)v
		&=R_0^\pm(\lambda^4)v-R_V^\pm(\lambda^4) v UvR_0^\pm(\lambda^4)v
		=R_0^\pm(\lambda^4)v\Big(1+UvR_0^\pm(\lambda^4)v\Big)^{-1}\\
		&=R_0^\pm(\lambda^4)v\Big(U+vR_0^\pm(\lambda^4)v\Big)^{-1}U^{-1}.
	\end{align*}
	Multiplying $Uv$ from the right, we obtain the desired formula for $R_V^\pm(\lambda^4)V$.
\end{proof}

%We now provide the asymptotic expansions of $\big(M^\pm(\lambda)\big)^{-1}$ in the presence of various types of zero resonances. 
Throughout this paper, we exclusively utilize $M^+(\lambda)$ only, so for simplicity, we denote it as $M(\lambda) = M^+(\lambda)$. The asymptotic expansion of $M^{-1}(\lambda)$ for the operator $H = \Delta^2 + V$ in $\R^3$ was initially obtained by Feng--Soffer--Yao \cite{FSY18} in the regular case. It was further established for all zero resonance cases by Erdo\u gan--Green--Toprak \cite{Erdogan-Green-Toprak}, where the decay estimates for $e^{itH}$ were derived using these asymptotic expansions.
However, in order to establish our results in  the paper, a more precise version of asymptotic expansions of $M^{-1} (\lambda)$ at $\lambda=0$ will be needed.

\begin{theorem}\label{thm-main-inver-M}
	Let $Q$ and $S_j$ for $j=1,2,3$ be the projections defined in Definition \ref{definition_resonance}.   Assume that $|V(x)| \lesssim \<x\>^{-\mu}$ with some $\mu >0$. Then there exists $\lambda_0>0$ such that
	the following asymptotic expansions of $	M^{-1}(\lambda)=(M(\lambda))^{-1}$  hold for $0<\lambda \le \lambda_0$ in $\mathbb{B}(L^2(\mathbb{R}^3))$:
	\begin{itemize}\label{projections}
		\item[(i)] If zero is   a  regular point of $H$ and $\mu> 9$, then
		\begin{equation}\label{thm-regularinver-M0}
			M^{-1}(\lambda)=QA_{0,1}^0Q+\lambda  A^0_{1,1}
			+\Gamma^0_2(\lambda);
		\end{equation}
	%\vskip0.2cm
		\item[(ii)] If zero is  the first kind resonance of $H$ and $\mu > 13$, then
		\begin{equation}\label{thm-resoinver-M1}
			M^{-1}(\lambda)
				= \lambda^{-1}S_1A_{-1,1}^1S_1+ \Big(S_1A^1_{0,1} +A^1_{0,2}S_1 +QA^1_{0,3}Q \Big)+\lambda  A^1_{1,1}+\Gamma^1_2(\lambda);
		\end{equation}
		%\vskip0.2cm
		\item[(iii)] If zero is the second kind resonance of $H$ and $\mu > 21$, then
		\begin{equation}\label{thm-resoinver-M2}
			\begin{split}
			M^{-1}(\lambda)
				&=\lambda^{-3}S_2A^2_{-3,1}S_2+\lambda^{-2}\Big(S_2A^2_{-2,1}S_1 + S_1A^2_{-2,2}S_2\Big)\\
				&+\lambda^{-1}\Big(S_2A^2_{-1,1}+A^2_{-1,2}S_2 +S_1A^2_{-1,3}S_1\Big)
		+\Big( S_1A_{0,1}^2 +A^2_{0,2}S_1 +QA^2_{0,3}Q\Big)\\
			&+\lambda  A^2_{1,1}+\Gamma^2_2(\lambda);
			\end{split}
		\end{equation}
		%\vskip0.2cm
		\item[(iv)] If zero is  the third kind resonance of $H$ and $\mu > 25$, then
		\begin{equation}\label{thm-resoinver-M3}
			\begin{split}
			M^{-1}(\lambda)
				&=\lambda^{-4} S_3 A_{-4,1}^3 S_3+\lambda^{-3} S_2A^3_{-3,1}S_2+\lambda^{-2}\Big(S_2A^3_{-2,1}S_1 + S_1A^3_{-2,2}S_2\Big)\\
			&+\lambda^{-1}\Big(S_2A^3_{-1,1}+ A^3_{-1,2}S_2+S_1A^3_{-1,3}S_1\Big)
				+\Big( S_1A_{0,1}^3 +A^3_{0,2}S_1 +QA^3_{0,3}Q\Big)\\&+ \lambda A^3_{1,1}+\Gamma^3_2(\lambda).
			\end{split}
		\end{equation}
	\end{itemize}
	Here in the statements (i)-(iv) above,
\begin{itemize}
	\item $ A_{i,j}^k $ denote $\lambda$-independent absolutely bounded operators on $ L^2(\mathbb{R}^3)$, where $ k $, $ i $, and $ j $ represent the resonance type, the degree of $\lambda$, and the order of operators with the same power of $\lambda$, respectively. (Here, $ k = 0$ corresponds to the regular case.)
	\item $\Gamma^k_2(\lambda)$ denote  $\lambda$-dependent absolutely  bounded operators  on $L^2(\mathbb{R}^3)$ satisfying
	\begin{align}
		\label{lemma_3_3_4}
		\norm{|\partial_\lambda^\ell \Gamma^k_2(\lambda)|}_{L^2\to L^2}\le C_{\ell, k} \lambda^{2-\ell},\quad 0<\lambda\le\lambda_0,\ \ell=0,1,2,\ k=0, 1,2,3.
	\end{align}
\end{itemize}
\end{theorem}

\begin{proof}
The proof will be given in Appendix \ref{Appendix A}. 
\end{proof}

\begin{remark}
To obtain the positive results in Theorem \ref{theorem1.2} (the $L^p$-boundedness of $W_\pm$), it is enough to establish associated boundedness properties of $A_{i,j}^k$. However, to prove the negative results (unboundedness of  of $W_\pm$), the specific expressions of the operators $A^2_{-3,1}$, $A^2_{-2,1}$, $A^2_{-1,1}$, $A^3_{-4,1}$, $A^3_{-3,1}$, $A^3_{-2,1}$, and $A^3_{-1,1}$ will play essential roles. 
%For the purpose of this paper, we only need to focus on the boundedness of most operators $A_{i,j}^k$ but disregarding their  expressions, except for the operators $A^2_{-3,1}$, $A^2_{-2,1}$, $A^2_{-1,1}$, $A^3_{-4,1}$, $A^3_{-3,1}$, $A^3_{-2,1}$, and $A^3_{-1,1}$, whose  specific expressions play an important role in  establishing  the unboundedness of wave operators $W_\pm$ on $L^p(\mathbb{R}^3)$ for $3 \leq p \le \infty$ in the second and third types of resonances, see \eqref{A321,1-biaojisecond}  and \eqref{A4321-1biaojithird} in Appendix \ref{Appendix A}.
\end{remark}

\section{The general strategy and a review of the regular case}\label{regular}
\label{regular3}

In this section, we recall very briefly a general strategy of the proof of $L^p$ boundedness of wave operator $W_-$ in the regular case by Goldberg and Green \cite{GoGr21} (see also \cite{MWY23}). For a fixed  $a>0$, we first decompose the stationary formula \eqref{stationary} of $W_-$ into the low and high energy parts
\begin{align}
	\label{lower}
	\Omega^{<a}&=\frac{2}{\pi i}\int_0^\infty \lambda^3 \chi_{<a}(\lambda) R_V^+(\lambda^4)V\left(R_0^+(\lambda^4)-R_0^-(\lambda^4)\right)d\lambda,
\end{align}
\begin{align}
	\label{higher}
	\Omega^{\ge a}&=\frac{2}{\pi i}\int_0^\infty \lambda^3 \chi_{\ge a}(\lambda) R_V^+(\lambda^4)V\left(R_0^+(\lambda^4)-R_0^-(\lambda^4)\right)d\lambda.
\end{align}
so that
$W_-=\Id-\left(\Omega^{<a}+	\Omega^{\ge a}\right)$, where $\chi_{<a}$ (resp. $\chi_{\ge a}$) is a cut-off function supported in $[0,a]$ (resp. $[a/2,\infty)$) given in Section \ref{subsection_notation}. It was proved by \cite[Proposition 4.1]{GoGr21} that, for any $a>0$, $\Omega^{\ge a}$ is bounded on $L^p(\R^3)$ for all $1\le p\le \infty$ if $V(x)=O(\<x\>^{-\mu})$ with $\mu>5$ without any condition on the zero energy. For the low energy part $\Omega^{< a} $, by  Lemma \ref{lemma_2.1} we  can write
\begin{align}
	\label{lower1}
	\Omega^{<a}&=\frac{2}{\pi i}\int_0^\infty \lambda^3 \chi_{<a}(\lambda) R_0^+(\lambda^4)v M^{-1}(\lambda)v\left(R_0^+(\lambda^4)-R_0^-(\lambda^4)\right)d\lambda.
\end{align}
For sufficiently small $a$, we will make use of asymptotic expansion of $M^{-1}(\lambda)$ in Theorem \ref{thm-main-inver-M} to further decompose \eqref{lower1} into several parts depending on the types of zero resonances. 

 If zero is a regular point of $H$, then  we can put $M^{-1}(\lambda)=QA_{0,1}^0Q+\lambda  A^0_{1,1}+\Gamma^0_2(\lambda)$ into \eqref{lower1} so that it is split into the sum of the following  three operators:
\begin{align}
	\label{T_B}
T_B&=\frac{2}{\pi i}\int_0^\infty \lambda^3 \chi_{<a}(\lambda) R_0^+(\lambda^4)v B v\left(R_0^+(\lambda^4)-R_0^-(\lambda^4)\right)d\lambda 
\end{align}
with $B=QA_{0,1}^0Q,\lambda  A^0_{1,1,},\Gamma^0_2(\lambda)$. Then it has been proved by Goldberg and Green  \cite{GoGr21} that $T_B$ is bounded on $L^p(\R^3)$ for all $1<p<\infty$ and all these three choices of $B$. Hence, the wave operator $W_-$ is bounded on $L^p(\R^3)$ for all $1<p<\infty$ in the regular case. 

The proofs of the $L^p$-boundedness for $T_B$ in \cite{GoGr21} do not depend on the specific expressions of the operators $A^0_{0,1}$, $A^0_{1,1}$ and $\Gamma^0_2(\lambda)$. Instead, the proofs rely on the absolute boundedness of these operators on $L^2$, the property (i) in Lemma \ref{gongzhengkehua-1} and \eqref{lemma_3_3_4}. In  other words, if $B$ is one of the forms
$$
QAQ,\quad \lambda A,\quad \Gamma_2(\lambda)
$$
with some $\lambda$-independent $A\in \AB(L^2)$ and some $\Gamma_2(\lambda)\in \AB(L^2)$ (see  Section \ref{subsection_notation} for the definition of $\AB(L^2)$) satisfying the same estimate as in \eqref{lemma_3_3_4}, then the argument of \cite{GoGr21} (see also \cite{MWY23}) works well to conclude  $T_B\in \mathbb B(L^p(\R^3))$ for all $1<p<\infty$. In the following argument, we will thus omit the proofs for such cases. %Additional details can be found in \cite{GoGr21, MWY23}.

We end this section with the following remark, which is   frequently used in the rest of the paper. 
\begin{remark}\label{real-valued basis}
	Recall that $ S_jL^2 $ for $ 1 \leq j \leq 4 $ are finite-dimensional subspaces of $QL^2 $ with $ S_{j+1}L^2 \subset S_jL^2 $. Let $ n_j = \dim S_jL^2 $, where $ n_4 \leq n_3 \leq n_2 \leq n_1 < \infty $.  Let $\{f_j\}_{j=1}^{n_1}$ be an orthonormal basis of $S_1L^2$ such that, for each $k=2,3,4$, $\{f_j\}_{j=1}^{n_k}$ is an orthonormal basis of $S_kL^2$ (which is possible since $S_{j+1}L^2$ is a closed subspace of $S_jL^2$ for $j=1,2,3$). Since $v$ is real-valued, by Lemma \ref{lemma_resonance} and \eqref{projectionS_4}, one can observe that $S_jL^2$ are invariant under the complex conjugation: 
	$$
	S_jL^2=\{f\in L^2\ |\ \overline f\in S_jL^2\}. 
	$$
	Hence we may assume without loss of generality that $f_j$ are real-valued. 
\end{remark}

\section{The first kind resonance case}
\label{the first kind4}
This section is devoted to the proof of  the $L^p$ boundedness of wave operator $W_-$ for all $1<p<\infty$ in the first kind resonance case. As  mentioned above,  it  suffices to deal with the low energy part $\Omega^{<a}$. Recalling the expansion of $M^{-1}$ in the first kind case in Theorem \ref{thm-main-inver-M}: 
\begin{align}
	\label{first resonance}
	M^{-1}(\lambda)
	= \lambda^{-1}S_1A_{-1,1}^1S_1+ S_1A^1_{0,1} +A^1_{0,2}S_1 +QA^1_{0,3}Q +\lambda  A^1_{1,1}+\Gamma^1_2(\lambda),
	\end{align}
we split $\Omega^{<a}$ into the six operators $T_B$ given by \eqref{T_B} with $B$ being one of six terms in this formula \eqref{first resonance}. As mentioned above, the cases with $B=QA^1_{0,3}Q$,  $\lambda  A^1_{1,1}$ and $\Gamma^1_2(\lambda)$ can be dealt with by the same method as in the regular case. Hence we  only consider the case with $B=\lambda^{-1}S_1A_{-1,1}^1S_1$,  $S_1A^1_{0,1}$ and  $A^1_{0,2}S_1 $, which is given by the following  Propositions \ref{proposition_4_1} and \ref{proposition_4_3}.

Before providing Propositions  \ref{proposition_4_1} and \ref{proposition_4_3}, we recall some notations from Section \ref{subsection_notation}:

\begin{itemize}
	\item 
 For $z=(z_1,z_2,z_3)\in \R^3$ and $k,l=1,2,3$, we set
%\begin{equation}
	\begin{align*}
		z_{kl}=z_kz_l,\quad 	v_{k}(z)=z_kv(z),\quad v_{kl}(z)=z_{kl}v(z).
	\end{align*}
	\item  Let $\mathcal R_j=\F^{-1}|\xi|^{-1}\xi_j\F$ be the Riesz transforms, $\mathcal R=(\mathcal R_1,\mathcal R_2,\mathcal R_3)$ and $\mathcal{R}_{kl}=\mathcal{R}_k\mathcal{R}_l $.
\item Let $
U_\theta u(x) = \theta^{-3} u(\theta^{-1} x)$ for $\theta>0$ and 
$J_\rho u(x) = u(x - \rho)
$ for $\rho\in \R^3$. 
%\item We denote by $\mathop{\mathrm{AB}}(L^2)$ the family of absolutely bounded operators on $L^2$ (see Section \ref{subsection_notation}).  
\end{itemize}

%Throughout the following Propositions \ref{proposition_4_1} and \ref{proposition_4_3}, $A$ always denotes an

%Recall that $T_B$ is the operator defined by \eqref{T_B}. 

%%%%%%%%%%%%%%%%%%
 \begin{proposition}	\label{proposition_4_1}
Let $A\in \AB(L^2)$ and $B=\lambda^{-1}S_1AS_1$. Then $T_B\in \mathbb B(L^p(\R^3))$ for all $1<p<\infty$. %: 	\begin{align*}%\label{K_{-1,1}^1} 	K^1_{-1,1}:=\frac{2}{\pi i}\int_0^\infty \lambda^2 \chi_{<a}(\lambda)\ R_0^+(\lambda^4)vS_1A_{-1,1}^1S_1 v\big(R_0^+(\lambda^4)-R_0^-(\lambda^4)\big )d\lambda. 	\end{align*}
 	\end{proposition}

\begin{proof} For short we set $e_{\omega}(z):=e^{i\omega\cdot z}$, where $\omega\cdot z=\omega_1z_1+\omega_2z_2+\omega_3z_3$ with $z\in \R^3$ and $\omega=(\omega_1,\omega_2,\omega_3)\in S^2$. Let $K_\lambda=R_0^+(\lambda^4)-R_0^-(\lambda^4)$, and $K_\lambda(x,y)$ denote the kernel of $K_\lambda$. 
 %Then	\begin{align}		\label{eq3.7}		(K^1_{-1,1}f)(x)=\frac{2}{\pi i}\int_0^\infty \lambda^2 \chi_{<a}(\lambda)\Big[R_0^+(\lambda^4)vS_1A_{-1,1}^1S_1 v\big(R_0^+(\lambda^4)-R_0^-(\lambda^4)\big )f\Big](x)d\lambda.		 \end{align}
By  the formula  \eqref{free_resolvent2} and the property $S_1v=0$ (see Lemma \ref{lemma_resonance}), we have for $f\in \mathcal S(\R^3)$, 
\begin{align*}
\nonumber
&[S_1vK_\lambda f](x)
=\frac{\pi i}{2\lambda}\int_{S^2} [S_1v(e_{\lambda \omega}-1)](x) \widehat{ f}(\lambda\omega)d\omega\\
%\label{eq3.2}
&=-\frac{\pi}{2}\sum_{k=1}^3[S_1v_k](x)\Big(\int_{S^2} \omega _k\widehat{ f}(\lambda\omega)d\omega\Big)
-\frac{\pi i}{2}\sum_{k,l=1}^3\lambda\int_0^1(1-\theta) \int_{S^2} [S_1v_{kl}e_{\lambda \theta \omega}](x)\omega_{kl} \widehat{ f}(\lambda\omega)d\omega d\theta,
\end{align*}
where the second equality follows from  the Taylor expansion: 
 \begin{align}
 \label{taylor}
 e^{i\lambda\omega\cdot z}=1+ i\lambda\omega\cdot z-(\lambda\omega\cdot z)^{2}  \int_0^1(1-\theta)e^{i\lambda \theta\omega\cdot z}d\theta. 
 \end{align}
%Recall that the notations \eqref{z} and \eqref{v_k}. Then $S_1v(R_0^+(\lambda^4)-R_0^-(\lambda^4))f$ is rewritten as\begin{align}	\label{eq3.2}-\frac{\pi}{2}\sum_{k=1}^3S_1v_k\Big(\int_{S^2} \omega _k\widehat{ f}(\lambda\omega)d\omega\Big)-\frac{\pi i}{2}\sum_{k,l=1}^3\lambda\int_0^1(1-\theta) S_1v_{kl}\int_{S^2} e_{\lambda \theta \omega}\omega_{kl} \widehat{ f}(\lambda\omega)d\omega d\theta,\end{align}
Since $\omega _k\widehat{ f}(\lambda\omega)=\widehat{\mathcal R_k f}(\lambda \omega),$  $\omega_{kl} \widehat{ f}(\lambda\omega)=\widehat{\mathcal R_{kl} f}(\lambda \omega),$  we hence can write
\begin{align}
	\label{eq4.10}
T_B=i\sum_{k=1}^{3} \Phi_k^1-\sum_{k,l=1}^{3} \Phi^1_{kl},
\end{align}
where%, with the notation $g_{*}(x)=(S_1AS_1v_*)(x)$ for short
\begin{align}
\label{Phi_k^1}
\Phi_k^1f(x)&=\int_0^\infty \lambda^2 \chi_{<a}(\lambda)\left[R_0^+(\lambda^4)vS_1AS_1v_k\right](x)\Big(  \int_{S^2} \widehat{\mathcal R_k f}(\lambda\omega)d\omega\Big )d\lambda,\\
\nonumber
\Phi^1_{kl}f(x)&=\int_0^1(1-\theta)\int_0^\infty \lambda^3 \chi_{<a}(\lambda)\left[R_0^+(\lambda^4)vS_1AS_1 v_{kl}  \int_{S^2} e_{\lambda \theta \omega}\widehat{\mathcal R_{kl} f}(\lambda\omega)d\omega \right](x)d\lambda d\theta.
\end{align}
\\
%\underline{(i) The proof of $\Phi^1_{kl}\in \mathbb{B}(L^p)$}.
\underline{(i) The proof of $\Phi^1_{kl}\in \mathbb{B}(L^p)$}. 
Recall that $S_1$ has finite rank (see Section \ref{subsection_resonance}) and $A\in \AB(L^2)$, which implies $S_1AS_1\in \AB(L^2)$. Let $(S_1AS_1)(x,y)$ denote its integral kernel and $|S_1AS_1|$ denote the integral operator with kernel $|(S_1AS_1)(x,y)|$. By \eqref{free_resolvent2}, we have 
%\begin{align}\label{riesz}	\int_{S^2} e_{\lambda \theta \omega}(z)\widehat{\mathcal R_{kl} f}(\lambda\omega)d\omega=\frac{2 \lambda}{\pi i}\left[(R_0^+(\lambda^4)-	R_0^-(\lambda^4))\mathcal R_{kl} f\right](\theta z),\end{align} where note that $e_{\lambda \theta \omega}(z)=e_{\lambda \omega}(\theta z)$. Hence
, for $\rho\in \R^3$, 
\begin{align*}
\Big(S_1AS_1v_{kl}  \int_{S^2} e_{\lambda \theta \omega}\widehat{\mathcal R_{kl} f}(\lambda\omega)d\omega\Big)(\rho)
&=\frac{2 \lambda}{\pi i}\int_{\R^3}(S_1AS_1)(\rho,z)v_{kl}(z)\left(K_\lambda \mathcal R_{kl} f\right)(\theta z)dz\\
&=\frac{2 \lambda}{\pi i}\int_{\R^6}(S_1AS_1)(\rho,z)v_{kl}(z)K_\lambda(\theta z,y)(\mathcal R_{kl} f)(y)dydz
\end{align*}
which, together with the above formula of $\Phi^1_{kl}$ yields
 \begin{align*}%\label{\Phi_k^1l}
\Phi^1_{kl}f(x)
&=\frac{2}{\pi i} \int_0^1(1-\theta)\int_0^\infty \lambda^4 \chi_{<a}(\lambda)\int_{\R^3}R_0^+(\lambda^4,x,\rho) v(\rho)\\
&\quad \times \int_{\R^6}(S_1AS_1)(\rho,z)v_{kl}(z)K_\lambda (\theta z,y)(\mathcal R_{kl} f)(y)dydzd\rho d\theta\\
&=\frac{2}{\pi i} \int_0^1(1-\theta)\int_{\R^6}v(\rho)(S_1AS_1)(\rho,z)v_{kl}(z)\\
&\quad \times \left[\int_{\R^3}\left(\int_0^\infty \lambda^4 \chi_{<a}(\lambda) R_0^+(\lambda^4,x,\rho)K_\lambda (\theta z,y)d\lambda\right) (\mathcal R_{kl} f)(y)dy\right]dzd\rho d\theta\\
%&=\frac{2}{\pi i} \int_0^1(1-\theta)\int_{\R^6}v(\rho)(S_1AS_1)(\rho,z)v_{kl}(z)\left[\int_{\R^3}A_{\rho,\theta z}(x,y) (\mathcal R_{kl} f)(y)dy\right]dzd\rho d\theta\\
&=\frac{2}{\pi i} \int_0^1(1-\theta)\int_{\R^6}v(\rho)(S_1AS_1)(\rho, z)v_{kl}(z) (A_{\rho, \theta z}\mathcal R_{kl} f)(x) dz d\rho d\theta,
\end{align*}
where  $A_{\rho, \theta z}$ is given by \eqref{Azw}. Minkowski's inequality,  Lemmas \ref{lem2.1} and  \ref{lem2.3}) thus show
\begin{align*}%\label{eq4.13}
\|\Phi^1_{kl}f\|_{L^p}
&\lesssim \int_0^1\int_{\R^6}\left|v(\rho)(S_1AS_1)(\rho, z)v_{kl}(z)\right|\|A_{\rho,\theta z}\|_{\mathbb B(L^p)}\|\mathcal R_{kl}\|_{\mathbb{B}(L^p)} \|f\|_{L^p}dzd\rho d\theta\\
&\lesssim\bigg(\int_{\R^6}\left|v(\rho)(S_1AS_1)(\rho, z)v_{kl}(z)\right|d\rho dz\bigg)  \sup_{\rho,w}\|A_{\rho,w}\|_{\mathbb{B}(L^p)}\|\mathcal R_{k}\|_{\mathbb{B}(L^p)}\|\mathcal R_{l}\|_{\mathbb{B}(L^p)} \|f\|_{L^p}\nonumber \\
&\lesssim \|v\|_{L^2}\||S_1AS_1|\|_{\mathbb B(L^2)}\|z_kz_lv\|_{L^2}\|A_{0,0}\|_{\mathbb{B}(L^p)}\|f\|_{L^p}\\
&\lesssim \|f\|_{L^p}
\end{align*}
for all $1<p<\infty$. %Since $S_1AS_1$ is a finite rank operator, if $|V(x)|\lesssim \<x\>^{-\mu}$ with $\mu>7/2$ then$$\int_{\R^6}\big|v(\rho)(S_1AS_1)(\rho, z)z_{kl} v(z)\big|d\rho dz\le \|v\|_{L^2}\|S_1AS_1\|_{\mathrm{HS}}\|z_{kl}v\|_{L^2}<\infty,$$where $\|\cdot\|_{\mathrm{HS}}$ denotes the Hilbert--Schmidt norm. 
Hence, $\Phi^1_{kl}\in \mathbb{B}(L^p)$ for all $1<p<\infty$.
\vspace{0.3cm}
\\
%\underline{(ii) The proof of  $\Phi_k^1\in \mathbb{B}(L^p)$}
\underline{(ii) The proof of  $\Phi_k^1\in \mathbb{B}(L^p)$}. %Using the orthonormal basis $\{f_j\}_{j=1}^{n_1}$ of $S_1L^2$, we write$$S_1AS_1v_k=\sum_{j=1}^{n_1}c_j^k f_j,\quad c_j^k:=\< S_1AS_1v_k, f_j\>. $$
%Then $\Phi_k^1$ is written in the form\begin{align}\label{Phi_k}	\Phi_k^1f(x)=\sum_{j=1}^{n_1}c_j^k\int_0^\infty \lambda^2 \chi_{<a}(\lambda)\Big(R_0^+(\lambda^4)vf_j\Big)(x)\Big(  \int_{S^2} \widehat{ \mathcal{R}_kf}(\lambda\omega)d\omega\Big )d\lambda.\end{align}
We decompose it into the high and low frequency parts: $$\Phi_k^1=\widetilde \chi_{\ge 4a}(D)\Phi_k^1+\widetilde \chi_{< 4a}(D)\Phi_k^1,$$ where $\widetilde \chi_{\ge 4a}(\xi)=\chi_{\ge 4a}(|\xi|)$ is defined in Section \ref{subsection_notation}. % and show that both $\widetilde \chi_{\ge 4a}(D)\Phi_k^1\in \mathbb{B}(L^p)$  and $\widetilde \chi_{< 4a}(D)\Phi_k^1\in \mathbb{B}(L^p)$ for all $1<p<\infty$.
%\vspace{0.3cm}
For the high frequency part $\widetilde \chi_{\ge 4a}(D)\Phi_k^1$, since 
$$
(|\xi|^4-\lambda^4)^{-1}=|\xi|^{-4}+\lambda^4|\xi|^{-4}(|\xi|^4-\lambda^4)^{-1}
$$
for $|\xi|\ge 4a$ and $0<\lambda<a$, we find
$\widetilde \chi_{\ge 4a}(D)R_0^+(\lambda^4)=m(D)+\lambda^4m(D)R_0^+(\lambda^4)$, where $m(\xi)=\widetilde\chi_{\ge 4a}(\xi)|\xi|^{-4}$. Hence we can write 
$
\widetilde \chi_{\ge 4a}(D)\Phi_k^1=\Phi_k^{1,1}+\Phi_k^{1,2}
$
with 
\begin{align}
\label{Phi_k^{1,1}}
\Phi_k^{1,1}f(x)&=(m(D)vS_1AS_1v_k)(x) \Big( \int_0^\infty \lambda^2 \chi_{<a}(\lambda) \int_{S^2} \widehat{\mathcal R_k f}(\lambda\omega)d\omega d\lambda\Big ),\\
\label{Phi_k^{1,2}}
\Phi_k^{1,2}f(x)&=\int_0^\infty \lambda^6 \chi_{<a}(\lambda) \left(m(D)R_0^+(\lambda^4)vS_1AS_1v_k\right)(x)\Big(  \int_{S^2} \widehat{\mathcal R_k f}(\lambda\omega)d\omega\Big )d\lambda.
\end{align}
For the part $\Phi_k^{1,1}$, by letting $\eta=\lambda\omega$ and using the unitarity of $\F$ on $L^2$, we have
$$
\int_0^\infty \lambda^2 \chi_{<a}(\lambda)\Big(  \int_{S^2} \widehat{\mathcal R_k f}(\lambda\omega)d\omega\Big )d\lambda=\int_{\R^3}\widetilde \chi_{<a}(\eta)\widehat{\mathcal R_k f}(\eta)d\eta%=\<\widehat{\mathcal R_k f},\widetilde \chi_{<a}\>
=\frac{1}{(2\pi)^{3/2}}\<\mathcal R_k f,\F^{-1}\widetilde \chi_{<a}\>.
$$
Since $\partial_\xi^\alpha m\in L^1$ for any $\alpha\in \Z^3_+$, $|\F^{-1}m(x)|\le C_N\<x\>^{-N}$ for any $N\ge0$. We thus have
\begin{align*}
%\label{Phi_1}
\|\Phi_k^{1,1}f\|_{L^p}&\lesssim \|(\F^{-1}m)*(vS_1AS_1v_k)\|_{L^p}|\<\mathcal R_kf,\F^{-1}\widetilde \chi_{<a}\>|\\
&\lesssim \|\F^{-1}m\|_{L^p}\|vS_1AS_1v_k\|_{L^1}\|\mathcal R_kf\|_{L^p}\|\mathcal F^{-1}\widetilde \chi_{<a}\|_{L^{p'}}\\
%&\le \|\F^{-1}m\|_{L^p} \|v\|_{L^2}\||S_1AS_1|\|_{\mathbb B(L^2)}\|z_kv\|_{L^2}\|\mathcal R_k\|_{\mathbb B(L^p)}\|f\|_{L^p}\|F^{-1}\widetilde \chi_{<a}\|_{L^{p'}}\\
&\lesssim \|f\|_{L^p},
\end{align*}
where we also used the fact $\widetilde \chi_{<a}\in \mathcal S(\R^3)$. 
For the part $\Phi_k^{1,2}$, we calculate
\begin{align*}
\left(R_0^+(\lambda^4)vS_1AS_1v_k\right)(x)
&=\int_{\R^3}R_0^+(\lambda^4,x,\rho)(vS_1AS_1v_k)(\rho)d\rho,\\
\lambda^2\widetilde \chi_{<4a}(\lambda)  \int_{S^2} \widehat{\mathcal R_k f}(\lambda\omega)d\omega
&=\frac{2}{\pi i}\left(K_\lambda |D|^3\widetilde \chi_{<a}(D)\mathcal R_kf\right)(0),
\end{align*}
where we have used \eqref{free_resolvent2} with $x=0$ and \eqref{free_resolvent4} for the second formula. Using these formulas and the fact $\widetilde \chi_{<a}=\widetilde \chi_{<a}\widetilde \chi_{<4a}$,  we can follow the way to calculate $\Phi_{kl}^1$ above to compute 
\begin{align*}
%\widetilde \Phi_k^{1,2}&=
&\int_0^\infty\lambda^6\widetilde \chi_{<a}\left(R_0^+(\lambda^4)vS_1AS_1v_k\right)(x)\Big(\int_{S^2} \widehat{\mathcal R_k f}(\lambda\omega)d\omega\Big )d\lambda\\
&=\frac{2}{\pi i}\int_{\R^3}(vS_1AS_1v_k)(\rho)
\left(\int_0^\infty\lambda^4\widetilde \chi_{<a}(\lambda)R_0^+(\lambda^4,x,\rho)\left(K_\lambda |D|^3\widetilde \chi_{<a}(D)\mathcal R_k f\right)(0)d\lambda \right)d\rho\\
&=\frac{2}{\pi i}\int_{\R^3}(vS_1AS_1v_k)(\rho)\left(A_{\rho,0}|D|^3\widetilde \chi_{<a}(D)\mathcal R_k f\right)(x)d\rho
\end{align*}
with $A_{\rho,0}$ being defined in  Lemma \ref{lem2.1}, which shows
$$
\Phi_k^{1,2}f(x)=\frac{2}{\pi i}\int_{\R^3}(vS_1AS_1v_k)(\rho)\left(m(D)A_{\rho,0}|D|^3\widetilde \chi_{<a}(D)\mathcal R_k f\right)(x)d\rho. 
$$
Since $m(D)$, $|D|^3\widetilde \chi_{<4a}(D)$ and $\mathcal{R}_k$ are bounded on $L^p(\R^3)$ by Lemma \ref{lem2.3}, and $A_{\rho,0}\in \mathbb{B}(L^p)$ uniformly in $\rho$ by Lemma \ref{lem2.1}  for all $1<p<\infty$, we obtain
$$
\|\Phi_k^{1,2}f\|_{L^p}\le \|vS_1AS_1v_k\|_{L^1}\|m(D)\|_{\mathbb{B}(L^p)} \sup_{\rho}\|A_{\rho,0}\|_{\mathbb{B}(L^p)}\||D|^3\widetilde \chi_{<4a}(D)\mathcal{R}_kf\|_{L^p}
\lesssim\|f\|_{L^p}.
$$
Thus $\widetilde \chi_{\ge 4a}(D)\Phi_k^1\in \mathbb{B}(L^p)$  for all $1<p<\infty$.

It remains to deal with the part $\widetilde \chi_{<4a}(D)\Phi_k^1$. Since $S_1AS_1v_k\in S_1L^2$ and $S_1L^2\perp\{v\}$ (see Lemma \ref{lemma_resonance}), we can write
\begin{align*}
\F[vS_1AS_1v_k](\xi)
&=\frac{1}{(2\pi)^3}\int_{\R^3}\big(e^{-iy\xi}-1\big)v(y)(S_1AS_1v_k)(y)dy\\
&=-i\sum_{l=1}^3\frac{1}{(2\pi)^3}\int_{\R^3}\left(\int_0^1 e^{-i\theta y\xi}d\theta\right)\xi_lv_l(y)(S_1AS_1v_k)(y)dy\\
&=-i\sum_{l=1}^3\frac{1}{(2\pi)^3}\int_0^1\int_{\R^3}e^{-iy\xi}\xi_lv_l(\theta^{-1}y)(S_1AS_1v_k)(\theta^{-1}y)\theta^{-3}dyd\theta\\
&=\sum_{l=1}^3 \int_0^1(-i\xi_l)\F[U_\theta v_lS_1AS_1v_k](\xi)d\theta,
\end{align*}
where $U_\theta: u(x)\mapsto \theta^{-3}u(\theta^{-1}x)$ is a dilation. We set $h_\theta (x)=(U_\theta v_lS_1AS_1v_k)(x)$ for short. Then
\begin{align*}
\left[R_0^+(\lambda^4)vS_1AS_1v_k\right](x)
&=-i\sum_{l=1}^3\int_0^1 R_0^+(\lambda^4)D_lh_\theta(x) d\theta\\
&=-i\sum_{l=1}^3\int_0^1  \Big\{\lambda \mathcal R_l R_0^+(\lambda^4)+\mathcal R_l R_0^+(\lambda^4)(|D|-\lambda)\Big\}
 h_\theta(x) d\theta,
\end{align*}
where we used the identity  $D_l=\lambda |D|^{-1}D_l+(|D|-\lambda)|D|^{-1}D_l$ in the last line. Plugging this formula into the definition of $\widetilde \chi_{<4a}(D)\Phi_k^1$ implies 
$$
\widetilde \chi_{<4a}(D)\Phi_k^1=-i\sum_{l=1}^3\left( \widetilde \Phi^{1,1}_k+\widetilde \Phi^{1,2}_k\right)
$$
with
\begin{align}\label{1stphi11}
\widetilde \Phi^{1,1}_kf(x)&=\int_0^1\int_0^\infty \lambda^3\chi_{<a}(\lambda) \left(\mathcal R_l \widetilde \chi_{<4a}(D)R_0^+(\lambda^4)h_\theta\right)(x)\Big(\int_{S^2} \widehat{\mathcal R_k f}(\lambda\omega)d\omega\Big )d\lambda d\theta,\\
\label{1stphi12}
\widetilde \Phi^{1,2}_kf(x)&=\int_0^1\int_0^\infty \lambda^2\chi_{<a}(\lambda) \left(\mathcal R_l \widetilde \chi_{<4a}(D)R_0^+(\lambda^4)(|D|-\lambda)h_\theta\right)(x)\Big(\int_{S^2} \widehat{\mathcal R_k f}(\lambda\omega)d\omega\Big )d\lambda d\theta.
\end{align}
By the same argument as that for $\Phi^{1,1}_k$ above, one can write
$$
\widetilde \Phi^{1,1}_kf(x)=\int_0^1\int_{\R^3}h_\theta (\rho)\left(\mathcal R_l\widetilde \chi_{<4a}(D)A_{\rho,0}\mathcal R_k f\right)(x)d\rho d\theta.
$$
Therefore, noting that $U_\theta$ leaves the $L^1$-norm invariant, we have for all $1<p<\infty$, 
\begin{align*}
\|\widetilde \Phi^{1,1}_kf\|_{L^p}
\lesssim \|\mathcal R_l\widetilde \chi_{<4a}(D)\|_{\mathbb B(L^p)} \left(\int_0^1\|h_\theta\|_{L^1}d\theta\right) \sup_\rho \|A_{\rho,0}\mathcal R_k f\|_{L^p}
\lesssim \|f\|_{L^p}.
\end{align*}

Finally, for the term $\widetilde \Phi^{1,2}_k$, %we first write\begin{align*}\left(R_0^+(\lambda^4)(|D|-\lambda)U_\theta v_lg_k\right)(x)=\frac{1}{2\lambda^2}\lim_{\ep\to 0^+}\int_{\R^6}e^{i(x-y)\xi}\left(\frac{1}{|\xi|^2-(\lambda+i\ep)^2}-\frac{1}{|\xi|^2+\lambda^2}\right)(|\xi|-\lambda)\end{align*}
since the Fourier symbol of $R_0^+(\lambda^4)(|D|-\lambda)$ satisfies
\begin{align*}
\lim_{\ep\to 0^+}\frac{1}{2\lambda^2}\left(\frac{1}{|\xi|^2-(\lambda+i\ep)^2}-\frac{1}{|\xi|^2+\lambda^2}\right)(|\xi|-\lambda)%=\frac{(2\lambda^2+2i\ep \lambda-\ep^2)(|\xi|-\lambda)}{(|\xi|^2-(\lambda+i\ep)^2)(|\xi|^2+\lambda^2)}
=\frac{1}{(|\xi|^2+\lambda^2)(|\xi|+\lambda)},
\end{align*}
one  can write
\begin{align*}
&\int_0^\infty \lambda^2\chi_{<a}(\lambda) \left(\widetilde \chi_{<4a}(D)R_0^+(\lambda^4)(|D|-\lambda)h_\theta\right)(x)\Big(\int_{S^2} \widehat{\mathcal R_k f}(\lambda\omega)d\omega\Big)\\
&=\int_{\R^6}e^{i(x-\rho)\xi}h_\theta(\rho)\left(\int_0^\infty\int_{S^2} \lambda^2 \frac{\chi_{<a}(\lambda)\widetilde \chi_{<4a}(\xi)}{(|\xi|^2+\lambda^2)(|\xi|+\lambda)} \widehat{\mathcal R_k f}(\lambda\omega)d\omega  d\lambda \right)d\rho d\xi\\
&=\int_{\R^6}e^{i(x-\rho)\xi}h_\theta(\rho)\left(\int_{\R^3} \frac{\widetilde \chi_{<a}(\eta)\widetilde \chi_{<4a}(\xi)}{(|\xi|^2+|\eta|^2)(|\xi|+|\eta|)} \widehat{\mathcal R_k f}(\eta)d\eta \right)dyd\xi\\
&=\int_{\R^3}h_\theta(\rho)\int_{\R^3}\left(\int_{\R^6} e^{i(x-\rho)\xi-iy\eta}\frac{\widetilde \chi_{<a}(\eta)\widetilde \chi_{<4a}(\xi)}{(|\xi|^2+|\eta|^2)(|\xi|+|\eta|)}d\eta d\xi\right)(\mathcal R_k f)(y)dyd\rho \\
&=\int_{\R^3}h_\theta(\rho)(J_\rho E_0\mathcal R_kf)(x)d\rho, 
\end{align*}
where $J_\rho: u(x)\mapsto u(x-\rho)$ is the translation by $\rho$ and $E_0$ has been defined by \eqref{E_0}. Hence
$$
\widetilde \Phi^{1,2}_kf(x)=\int_0^1\int_{\R^3}h_\theta(\rho)(\mathcal R_lJ_\rho E_0\mathcal R_kf)(x)d\rho d\theta,\quad h_\theta(x) =U_\theta[v_lS_1AS_1v_k](x). 
$$
This formula, combined with Lemmas \ref{lem2.2} and \ref{lem2.3}, shows
$$
\|\widetilde \Phi^{1,2}_kf\|_{L^p}\lesssim \int_0^1\|h_\theta\|_{L^1}d\theta\sup_\rho\|\mathcal R_lJ_\rho E_0\mathcal R_k\|_{\mathbb B(L^p)}\|f\|_{L^p}\lesssim \|f\|_{L^p}. 
$$
This completes the proof of $\Phi^1_k\in \mathbb B(L^p(\R^3))$ for all $1<p<\infty$, and thus the proof of the proposition. 
\end{proof}

%remark
\begin{remark}\label{remark4.2}
Let $\alpha\geq0$ and $g(z)$ be  a polynomial of degree $s$ on $\R^3$. Let $g(\mathcal{R})$ be the corresponding polynomial of Riesz transforms $\mathcal{R}=(\mathcal{R}_1, \mathcal{R}_2, \mathcal{R}_3)$. It is clear that $g(\mathcal{R})\in \mathbb B(L^p(\R^3))$ for all $1<p<\infty$.  Let $B=K+F(x)\in \mathop{\mathrm{AB}}(L^2)$ with an integral operator $K\in \mathbb B(L^2)$ with kernel $K(x,y)$ and a multiplication operator $F\in L^\infty$. Let $|K|\in \mathbb B(L^2)$ be the integral operator with kernel $|K(x,y)|$. 
Consider the following operators: 
\begin{align}
\label{Rema4.2ii}
\Theta f(x)&=\int_0^\infty \lambda^2 \chi_{<a}(\lambda)\left(R_0^+(\lambda^4)vS_1Bgv\right)(x)\left(  \int_{S^2}\widehat{g(\mathcal{R}) f}(\lambda\omega)d\omega\right )d\lambda,\\
\label{Rema4.2i}
\Phi f(x)&=\int_0^1(1-\theta)^\alpha \int_0^\infty \lambda^3 \chi_{<a}(\lambda)\bigg[R_0^+(\lambda^4)vB gv  \int_{S^2} e_{\lambda\theta \omega}\widehat{g(\mathcal{R}) f}(\lambda\omega)d\omega\bigg](x)d\lambda d\theta.
\end{align}
 Then $\Theta,\Phi\in \mathbb{B}(L^p(\R^3))$ for all $1<p<\infty$ provided $|V(x)|\lesssim \<x\>^{-\mu}$ with $\mu>\max\{3+2s,5\}$. 
 
 Indeed, for the operator $\Theta$, one can follow the proof of $\Phi_k^1\in \mathbb B(L^p(\R^3))$ in the proof of Proposition \ref{proposition_4_1} where no specific structure of $B$, except for its absolute boundedness on $L^2$, was used. 

For $\Phi$, as in the analysis of $\Phi^1_{kl}$ in the proof of Proposition \ref{proposition_4_1}, we can write
$$
\Phi f(x)=\frac{2}{\pi i} \int_0^1(1-\theta)^\alpha \int_{\R^6}v(\rho)B(\rho, z) g(z) v(z) (A_{\rho, z\theta}g(\mathcal{R}) f)(x) dz d\rho d\theta
$$
provided that $B=K$, that is $F\equiv0$. Then Lemma \ref{lem2.1} yields
\begin{align}
\label{remark_4_2_1}
\|\Phi f\|_{L^p}\lesssim \|v\|_{L^2}\||K|\|_{\mathbb B(L^2)}\|gv\|_{L^2}\|A_{0,0}\|_{\mathbb B(L^p)}\|g(\mathcal R)\|_{\mathbb B(L^p)}\|f\|_{L^p}\lesssim \|f\|_{L^p}.
\end{align}
The same argument also works well even if  $F\not\equiv0$ by letting $B(\rho,z):=F(z)\delta_{z}(\rho)+K(\rho,z)$ with Dirac's delta $\delta_z$ centered at $z$, in which case \eqref{remark_4_2_1} still holds with $\||K|\|_{\mathbb B(L^2)}$ replaced by $\|F\|_{L^\infty}+\||K|\|_{\mathbb B(L^2)}$. 
\end{remark}

%\subsection{The $L^p$-boundedness of $\Omega_{<a}$ in the second and third kinds of  resonance }

  Next, we consider  the remained two terms associated with $S_1A_{0,1}^1 $ and $A_{0,2}^1 S_1$ in  \eqref{first resonance}. 
   %  the following two operators related with $S_1A_{0,1}^1 $ and $A_{0,2}^1 S_1$, respectively:
 %	\begin{align*}%\label{K_{0,1}^1}	&K^1_{0,1}=\frac{2}{\pi i}\int_0^\infty \lambda^3 \chi_{<a}(\lambda)\ R_0^+(\lambda^4)vS_1A_{0,1}^1 v\big(R_0^+(\lambda^4)-R_0^-(\lambda^4)\big )d\lambda,\\
% \label{K_{0,2}^1} &K^1_{0,2}=\frac{2}{\pi i}\int_0^\infty \lambda^3 \chi_{<a}(\lambda)\ R_0^+(\lambda^4)vA_{0,2}^1S_1 v\big(R_0^+(\lambda^4)-R_0^-(\lambda^4)\big )d\lambda.\end{align*}

%%%%%%%%%%%%%%%%%
\begin{proposition}\label{proposition_4_3}
Let $A\in \AB(L^2)$, $B\in \{S_1A,AS_1\}$. Then $T_{B}\in \mathbb{B}(L^p(\R^3))$ for all $1<p<\infty$.
%Let  $K^1_{0,1}$  and  $ K^1_{0,2}$ be the operators given in \eqref{K_{0,1}^1} and \eqref{K_{0,2}^1}, respectively. Then $K^1_{0,1},\ K^1_{0,2} \in \mathbb{B}(L^p(\R^3))$ for all $1<p<\infty$.
\end{proposition}

\begin{proof}The proof is similar to that of Proposition \ref{proposition_4_1}. For the case $B=S_1A$, by using \eqref{free_resolvent2} and the Taylor expansion $e^{i\lambda \omega\cdot z}=1+i\lambda \omega\cdot z\int_0^1e^{i\lambda\theta\omega\cdot z}d\theta$, we can write $T_{S_1A}=\Phi^1_0+i \sum_{k=1}^3 \Phi_k^1$, where 
\begin{align*}
\Phi^1_0 f(x)&=\int_0^\infty\lambda^2 \chi_{<a}(\lambda)\big(R_0^+(\lambda^4)vS_1A v\big)(x) \Big(\int_{S^2} \widehat{ f}(\lambda\omega)d\omega\Big) d\lambda,\\
\Phi_k^1f(x)&=\int_0^1 \int_0^\infty \lambda^3 \chi_{<a}(\lambda)\left[ R_0^+(\lambda^4)vS_1A v_k\int_{S^2} e_{\lambda\theta \omega}\widehat{ \mathcal{R}_kf}(\lambda\omega)d\omega\right] (x)d\lambda d\theta.
\end{align*}
% \begin{align}\label{eq4.32}	 	K^1_{0,1}f(x)&=\int_0^\infty\lambda^2 \chi_{<a}(\lambda)\big[ R_0^+(\lambda^4)vS_1A_{0,1}^1 v\big](x) \Big(\int_{S^2} \widehat{ f}(\lambda\omega)d\omega\Big) d\lambda\nonumber\\	 	&\ \ \ \ \ \ \ + i \sum_{k=1}^3\int_0^1 \int_0^\infty \lambda^3 \chi_{<a}(\lambda)\Big[ R_0^+(\lambda^4)vS_1A_{0,1}^1 \Big(vz_k\int_{S^2} e^{i\lambda\omega\theta z}\widehat{ \mathcal{R}_kf}(\lambda\omega)d\omega\Big)\Big] (x)d\lambda d\theta\nonumber\\	 	&=:  Tf(x)+i \sum_{k=1}^3 \Phi_k^1f(x).	 \end{align}
It thus follows from Remark \ref{remark4.2} that $\Phi^1_0,\Phi_k^1 \in \mathbb{B}(L^p)$ for all $1<p<\infty$. 

Similarly, if $B=AS_1$, then we use the formula \eqref{free_resolvent2} and the fact $S_1v=0$ to derive% (ii) The proof of   $K^1_{0,2} \in \mathbb{B}(L^p(\R^3))$ for all $1<p<\infty.$Putting the formula  \eqref{free_resolvent2} and  the following  equality:
                         $$S_1v e_{\lambda\omega}=i \lambda \sum_{k=1}^3 \omega_kS_1v_k\int_0^1 e_{\lambda\theta \omega} d\theta.$$
Using this formula and a similar computations to that for $\Phi^1_{kl}$, one can write $T_{AS_1}$ as 
\begin{align*}%\label{eq4.36}
	T_{AS_1}f(x)&= i \sum_{k=1}^3\int_0^1 \int_0^\infty \lambda^3 \chi_{<a}(\lambda)\bigg[ R_0^+(\lambda^4)vAS_1v_k\int_{S^2} e_{\lambda\theta \omega}\widehat{ \mathcal{R}_kf}(\lambda\omega)d\omega\Big)\bigg] (x)d\lambda d\theta
\end{align*}
which is bounded on $L^p$ for all $1<p<\infty$ by Remark \ref{remark4.2}. %,  one has  $	K^1_{0,2} \in \mathbb{B}(L^p(\R^3))$ for all $1<p<\infty$.  Finally, summing up the  arguments in (i) and (ii) above,  we have finished the whole proof of Proposition \ref{proposition_4_3}.
\end{proof}

As mentioned in the beginning of this section, all of the terms in \eqref{first resonance}, except for $B=\lambda^{-1}S_1A^1_{-1,1}S_1,S_1A^1_{0,1},A^1_{0,2}S_1$, can be dealt with by the same argument as in the regular case. For these remained three terms, we can apply Propositions \ref{proposition_4_1} and \ref{proposition_4_3}. Hence we have completed the proof of Theorem \ref{theorem1.2} (i).

\section{The second kind resonance case}
\label{the second kind5}
In this section, we suppose zero is a resonance of the second kind and prove Theorem \ref{theorem1.2} (ii). 
%%%%%%%%%%%%%%%%%%%%%%%
\subsection{The boundedness for $1<p<3$ }
We first prove $W_-$ is bounded on $L^p$ for all $1< p<3$. Recall the expansion \eqref{thm-resoinver-M2} of  $M^{-1}(\lambda)$ for the resonance of the second kind in Theorem \ref{thm-main-inver-M}:
\begin{align}
	\label{second resonance}
	M^{-1}(\lambda)
	=&\lambda^{-3}S_2A^2_{-3,1}S_2+\lambda^{-2}\Big(S_2A^2_{-2,1}S_1 + S_1A^2_{-2,2}S_2\Big)
	+\lambda^{-1}\Big(S_2A^2_{-1,1}+A^2_{-1,2}S_2 \nonumber\\ &\ \ \ \ \  \ \ \ \ \ \ +S_1A^2_{-1,3}S_1\Big)
	+\Big( S_1A_{0,1}^2 +A^2_{0,2}S_1 +QA^2_{0,3}Q\Big)
	+\lambda  A^2_{1,1}+\Gamma^2_2(\lambda).
\end{align}
As above, we decompose the low energy part of $W_-$ into several integral operators $T_B$ given by \eqref{T_B} using this expansion. Then, it is enough to deal with the five terms corresponding with 
\begin{align}
\label{first_kind}
B=\lambda^{-3}S_2A^2_{-3,1}S_2,\ \lambda^{-2}S_2A^2_{-2,1}S_1,\ \lambda^{-1}S_2A^2_{-1,1},\ \lambda^{-2}S_1A^2_{-2,2}S_2,\ \lambda^{-1}A^2_{-1,2}S_2;
\end{align} otherwise the proof is exactly same as in the previous cases. Precisely, we have $T_B\in \mathbb B(L^p(\R^3))$ for all $1<p<\infty$ and for $B=QA^2_{0,3}$, $\lambda  A^2_{1,1}$ and $\Gamma^2_2$ by the same argument as in the regular case,  and for $B=\lambda^{-1}S_1A^2_{-1,3}S_1$, $S_1A_{0,1}^2$ and $A^2_{0,2}S_1$ by applying Propositions \ref{proposition_4_1} and \ref{proposition_4_3}, respectively. 

The $L^p$-boundedness of $T_B$ for the case \eqref{first_kind} is given by the following Propositions \ref{proposition_5_1}-\ref{proposition_5_3}. As in the previous section, we frequently use the notations in Section \ref{subsection_notation} in the following argument. 
%Before providing Propositions  \ref{proposition_5_1}-\ref{proposition_5_3}, we recall some notations (See Section \ref{subsection_notation}):

%\begin{itemize}	\item 	For $z=(z_1,z_2,z_3)\in \R^3$ and $k,l,s,t=1,2,3$,  set	\begin{align*}		&z_{kl}=z_kz_l,\quad z_{kls}=z_kz_lz_s,\quad z_{klst}=z_kz_lz_sz_t;\\		&	v_{k}(z)=z_kv(z),\quad v_{kl}(z)=z_{kl}v(z),\quad v_{kls}(z)=z_{kls}v(z),\quad v_{klst}(z)=z_{klst}v(z).	\end{align*}	\item  Denote the Riesz transforms $\mathcal R=(\mathcal R_1,\mathcal R_2,\mathcal R_3) $  and  set	 \begin{align*}		%\label{RKL}		\mathcal{R}_{kl}=\mathcal{R}_k\mathcal{R}_l,\quad \mathcal{R}_{kls}=\mathcal{R}_k\mathcal{R}_l\mathcal{R}_s,\quad \mathcal{R}_{klst}=\mathcal{R}_k\mathcal{R}_l\mathcal{R}_s\mathcal{R}_t.	\end{align*}\item	For fixed  $\theta \in \mathbb{R} \setminus \{0\}$, define the dilation operator $U_\theta$ by  	$	U_\theta u(x) = \theta^{-3} u(\theta^{-1} x).	$ For fixed  $\rho \in \mathbb{R}^3$,  denote $J_\rho$  the translation operator by  	$	J_\rho u(x) = u(x - \rho).  	$\end{itemize}

%%%%%%%%%%%%%
\begin{proposition}	\label{proposition_5_1}
Let $B=\lambda^{-3}S_2AS_2$ with $A\in \AB(L^2)$. Then $T_B\in \mathbb B(L^p(\R^3))$ for $1<p<3$. %The following operator $K^2_{-3,1}$ is bounded on $L^p(\R^3))$ for all $1<p<3$: 	\begin{align}\label{K_{-3,1}^2}		K^2_{-3,1}:=\frac{2}{\pi i}\int_0^\infty  \chi_{<a}(\lambda)\ R_0^+(\lambda^4)vS_2A_{-3,1}^2S_2 v\big(R_0^+(\lambda^4)-R_0^-(\lambda^4)\big )d\lambda.	\end{align}
\end{proposition}

\begin{proof}The general strategy is similar to that in the previous section. The only difference is that, since $\lambda^{-3}S_2AS_2$ is more singular at $\lambda=0$ than before, which causes the restriction $p<3$. Due the same reason, we also need the cancelation property of $S_2$ against not only $v(z)$ but also its first moment $zv(z)$, that is $S_2v=S_2v_k=0$ for all $k=1,2,3$, where $v_k=z_kv$. We thus use 
the Taylor expansion
	\begin{align}\label{taylor3}
		e^{i\lambda\omega\cdot z}=\sum_{\ell=0}^3\frac{i^{\ell}}{\ell !}(\lambda\omega\cdot z)^{\ell}
		  +\frac{1}{6}(\lambda\omega\cdot z)^{4}  \int_0^1(1-\theta)^3e^{i\lambda \theta\omega\cdot z}d\theta,\quad z\in \R^3,\ \omega\in S^2,
	\end{align}
and the formula \eqref{free_resolvent2} to rewrite $S_2v(R_0^+(\lambda^4)-	R_0^-(\lambda^4))f$ as 
 \begin{align*}\nonumber
&-\frac{\pi i}{4}\sum_{k,l=1}^3\lambda S_2v_{kl}\left(\int_{S^2}\widehat{\mathcal R_{kl}f}(\lambda\omega)d\omega\right)
			+\frac{\pi}{12}\sum_{k,l,s=1}^3\lambda^2S_2v_{kls}\left(\int_{S^2}\widehat{\mathcal{R}_{kls}f}(\lambda\omega)d\omega\right)\\
	& \ \ \ \ \ \ \ \ \ \ +\frac{\pi}{12}\sum_{k,l,s,t=1}^3\lambda^3\int_0^1(1-\theta)^3S_2v_{klst}
			\int_{S^2} e_{\lambda\theta\omega}\widehat{\mathcal{R}_{klst}f}(\lambda\omega)d\omega\big) d\theta,
\end{align*}
where $e_{\omega}(z)=e^{i\omega\cdot z}$, $v_k=z_kv$, $v_{kl}=z_kz_lv$ and so on. 
Plugging this formula into \eqref{T_B} gives
	\begin{align}\label{K2-31-unbound}
		T_B
		=-\frac{1}{2}\sum_{k,l=1}^3\Phi^2_{kl}-\frac{i}{6}\sum_{k,l,s=1}^3\Phi^2_{kls}-\frac{i}{6}\sum_{k,l,s,t=1}^3\Phi^2_{klst},
	\end{align}
where
\begin{align}
\label{Phi^2_{kl}}
\Phi^2_{kl}f(x)=&\int_0^\infty \lambda\chi_{<a}(\lambda)\big( R_0^+(\lambda^4)vS_2AS_2v_{kl}\big)(x)
		\Big(\int_{S^2}\widehat{\mathcal R_{kl}f}(\lambda\omega)d\omega\Big) d\lambda,\\
\label{Phi^2_{kls}}
\Phi^2_{kls}f(x)=&\int_0^\infty \lambda^2\chi_{<a}(\lambda)\big( R_0^+(\lambda^4)vS_2AS_2v_{kls}\big)(x)
		\Big(\int_{S^2}\widehat{\mathcal{R}_{kls}f}(\lambda\omega)d\omega\Big) d\lambda,\\
\label{Phi^2_{klst}}
\Phi^2_{klst}f(x)=&\int_0^1(1-\theta)^3\int_0^\infty \lambda^3\chi_{<a}(\lambda)
		\bigg[ R_0^+(\lambda^4)vS_2AS_2v_{klst}\int_{S^2} e_{\lambda\theta\omega}\widehat{\mathcal{R}_{klst}f}(\lambda\omega)d\omega\bigg](x)d\lambda d\theta.
\end{align}
Remark \ref{remark4.2} then implies 
 $ \Phi^2_{klst},\Phi^2_{kls}\in \mathbb B(L^p(\R^3))$ for all $1<p<\infty$ and  $1\le k,l,s,t\le 3$. 

It remains to deal with the most singular part $\Phi^2_{kl}$ with respect to $\lambda$. As in the previous case, we decompose $\Phi^2_{kl}=\widetilde \chi_{\ge 4a}(D)\Phi^2_{kl}+\widetilde \chi_{<4a}(D) \Phi^2_{kl}$. For $\widetilde \chi_{\ge 4a}(D)\Phi^2_{kl}$, taking into account the fact that we have $\lambda$ rather than $\lambda^2$ in the integrand of  $\Phi^2_{kl}$, we follow the way to deal with the operator $\widetilde \chi_{\ge 4a}(D)\Phi^1_{k}$ in the proof of Proposition \ref{proposition_4_1} to write $\widetilde \chi_{\ge 4a}(D)\Phi^2_{kl}=\Phi^{2,1}_{kl}+\Phi^{2,2}_{kl}$ with
\begin{align}
\nonumber
\Phi^{2,1}_{kl}f(x)
&=(m(D)vS_2AS_2v_{kl})(x) \Big( \int_0^\infty \lambda \chi_{<a}(\lambda) \int_{S^2} \widehat{\mathcal R_{kl} f}(\lambda\omega)d\omega d\lambda\Big )\\
\label{Phi^{2,1}_{kl}}
&=(m(D)vS_2AS_2v_{kl})(x)\int_{\R^3}\widetilde m(\eta)\widehat{\mathcal R_{kl}f}(\eta)d\eta,\\
\nonumber
\Phi^{2,2}_{kl}f(x)
&=\int_0^\infty \lambda^5 \chi_{<a}(\lambda) \left(m(D)R_0^+(\lambda^4)vS_2AS_2v_{kl}\right)(x)\Big(  \int_{S^2} \widehat{\mathcal R_{kl} f}(\lambda\omega)d\omega\Big )d\lambda\\
\label{Phi^{2,2}_{kl}}
&=\frac{2}{\pi i}\int_{\R^3}(vS_2AS_2v_{kl})(\rho)\left(A_{\rho,0}|D|^2\widetilde \chi_{<a}(D)\mathcal R_{kl}f\right)(x)d\rho,
\end{align}
where $m(\xi)=|\xi|^{-4}\widetilde \chi_{\ge 4a}(\xi)$ and $\widetilde m(\eta)=|\eta|^{-1}\widetilde \chi_{<a}(\eta)$. It will be seen in Proposition \ref{L1,L2,L3} below that 
\begin{equation*}
\begin{split}
|\F\widetilde m(y)|\lesssim \<y\>^{-2}\in L^{q},\quad 3/2<q<\infty. 
\end{split}
\end{equation*}
Hence $\Phi^{2,1}_{kl}\in \mathbb B(L^p(\R^3))$ for $1<p<3$ as follows:  
$$
\|\Phi^{2,1}_{kl}f\|_{L^p}\lesssim \|\F^{-1}m\|_{L^p}\|vS_2AS_2v_{kl}\|_{L^1}\|\mathcal R_{kl}f\|_{L^p}\|\F\widetilde m\|_{L^{p'}}\lesssim \|f\|_{L^p}. 
$$
Moreover, the same argument as that for $\Phi_{k}^{1,2}$ in 
Proposition \ref{proposition_4_1} implies $\Phi^{2,2}_{kl}\in \mathbb B(L^p(\R^3))$ for all $1<p<\infty$. 

It remains to  deal with $\widetilde \chi_{< 4a}(D)\Phi^2_{kl}$. Since $S_2AS_2v_{kl}\in S_2L^2$ and $S_2L^2\perp\{v,z_1v,z_2v,z_3v\}$ (see Lemma \ref{lemma_resonance}), by  the Taylor expansion of $e^{-i z\xi}$, one has
\begin{align*}
\F[S_2AS_2v_{kl}](\xi)
%&=\frac{1}{(2\pi)^3}\int_{\R^3}\big(e^{-i\rho\xi}-1+i\sum_{t=1}^3\rho_t\xi_t\big)
%		(vf_j)(\rho)d\rho \nonumber\\
&=\frac{-1}{(2\pi)^3}\sum_{s,t=1}^3 \int_0^1(1-\theta)\int_{\R^3}\xi_s\xi_t e^{-i\theta y\xi}v_{st}(y)(S_2AS_2v_{kl})(y)dy d\theta\\
&=-\sum_{s,t=1}^3\int_0^1(1-\theta)\xi_s\xi_t\F[U_\theta v_{st}S_2AS_2v_{kl}](\xi)d\theta
\end{align*}
and hence, with $\widetilde h_\theta(x)=(U_\theta v_{st}S_2AS_2v_{kl})(x)$ and the formula $D_sD_t=\lambda^2 \mathcal R_{st}+(|D|^2-\lambda^2)\mathcal R_{st}$, 
$$
\left[R_0^+(\lambda^4)S_2AS_2v_{kl}\right](x)=-\sum_{s,t=1}^3\int_0^1(1-\theta)\mathcal R_{st}\left\{\lambda^2 R_0^+(\lambda^4)+R_0^+(\lambda^4)(|D|^2-\lambda^2)\right\}\widetilde h_\theta(x) d\theta,
$$
where $
U_\theta u(x) = \theta^{-3} u(\theta^{-1} x)
$. Taking into account the fact that we have $\lambda$ rather than $\lambda^2$ in the integrand of  $\Phi^2_{kl}$, we plug this into the above formula of $\Phi^2_{kl}$ and follow the way to calculate $\widetilde \Phi^{1,1}_k$ and $\widetilde \Phi^{1,2}_k$ defined by \eqref{1stphi11} and  \eqref{1stphi12}, respectively,  to observe that $\widetilde \chi_{< 4a}(D)\Phi^2_{kl}$ is the sum of 
%\begin{align*}\widetilde \chi_{< 4a}(D)\Phi^2_{kl}=-\sum_{s,t=1}^3\left(\widetilde \Phi^{2,1}_{kl}+\widetilde \Phi^{2,2}_{kl}\right)\end{align*}with 
\begin{align*}
\widetilde \Phi^{2,1}_{kl} f(x)&=\int_0^1(1-\theta)\int_{\R^3}\widetilde h_\theta (\rho)\left(\mathcal R_{st}\widetilde \chi_{<4a}(D)A_{\rho,0}\mathcal R_{kl} f\right)(x)d\rho d\theta,\\
\widetilde \Phi^{2,2}_{kl} f(x)&=\int_0^1\int_{\R^3}\widetilde h_\theta(\rho)(\mathcal R_{st}J_\rho E_1\mathcal R_{kl}f)(x)d\rho d\theta,
\end{align*}
where $
J_\rho u(x) = u(x - \rho)
$ and $E_1$ is given by \eqref{E_1}. Note that we have used the following identities to derive $\widetilde \Phi^{2,2}_{kl}$ (from which it is clear why $E_1$ appears instead of $E_0$): 
\begin{align*}
\lim_{\ep\to 0^+}\frac{1}{2\lambda^2}\left(\frac{1}{|\xi|^2-(\lambda+i\ep)^2}-\frac{1}{|\xi|^2+\lambda^2}\right)(|\xi|^2-\lambda^2)%=\frac{(2\lambda^2+2i\ep \lambda-\ep^2)(|\xi|-\lambda)}{(|\xi|^2-(\lambda+i\ep)^2)(|\xi|^2+\lambda^2)}
=\frac{1}{|\xi|^2+\lambda^2},
\end{align*}
and 
$$
\int_0^\infty\int_{S^2} \frac{\lambda \chi_{<a}(\lambda)\widetilde \chi_{<4a}(\xi)}{|\xi|^2+\lambda^2} \widehat{\mathcal R_{kl} f}(\lambda\omega)d\omega  d\lambda
=\int_{\R^3}\frac{\widetilde \chi_{<a}(\eta)\widetilde \chi_{<4a}(\xi)}{(|\xi|^2+|\eta|^2)|\eta|} \widehat{\mathcal R_{kl} f}(\eta)d\eta. 
$$
By the same argument as that for $\widetilde \Phi^{1,1}_k$ and $\widetilde \Phi^{1,2}_k$ with $E_0$ replaced by $E_1$ and using Lemma \ref{lem2.2}, we have $\widetilde \Phi^{2,1}_{kl}\in \mathbb{B}(L^p(\R^3))$ for all $1<p<\infty$ and $\widetilde \Phi^{2,2}_{kl}\in \mathbb{B}(L^p(\R^3))$ for all $1<p<3$. This shows $\widetilde \chi_{<4a}(D) \Phi^2_{kl}\in  \mathbb{B}(L^p(\R^3))$ for all  $1<p<3$ and completes the proof of the proposition. %In summary, since each $\Phi^2_{kls},$ $\Phi^2_{klst}\in  \mathbb{B}(L^p(\R^3))$ for all  $1<p<\infty$ and each $\Phi^2_{kl}f\in  \mathbb{B}(L^p(\R^3))$ for all  $1<p<3$, then by \eqref{K2-31-unbound}, $K^2_{-3,1}\in \mathbb B(L^p(\R^3))$ for all $1<p<3$. 
\end{proof}

 We next deal with operators associated with $B=\lambda^{-2}S_2A^2_{-2,1}S_1$ and $\lambda^{-1}S_2A^2_{-1,1}$. %Let\begin{align}\label{K2-21}K^2_{-2,1}&:=\frac{2}{\pi i}\int_0^\infty \lambda \chi_{<a}(\lambda)R_0^+(\lambda^4)vS_2A^2_{-2,1}S_1 v\big(R_0^+(\lambda^4)-R_0^-(\lambda^4)\big )d\lambda,\\\label{K2-11}	K^2_{-1,1}&:=\frac{2}{\pi i}\int_0^\infty \lambda^2 \chi_{<a}(\lambda)R_0^+(\lambda^4)v S_2A^2_{-1,1}v\big(R_0^+(\lambda^4)-R_0^-(\lambda^4)\big )d\lambda.\end{align}

%proposition
\begin{proposition}\label{proposition_5_2}
Let $A\in \AB(L^2)$ and $B\in \{\lambda^{-2}S_2AS_1,\lambda^{-1}S_2A\}$. Then $T_B \in \mathbb{B}(L^p(\R^3))$ for all $1<p<3$.
\end{proposition}

\begin{proof}Suppose first $B=\lambda^{-2}S_2AS_1$. By the Taylor expansion of $e^{i\lambda\omega\cdot z}$:
	\begin{align}\label{taylor(1-theta)2}
		e^{i\lambda\omega\cdot z}=\sum_{\ell=0}^2\frac{i^{\ell}}{\ell !}(\lambda\omega\cdot z)^{\ell}
		-\frac{i}{2}(\lambda\omega\cdot z)^{3}\int_0^1(1-\theta)^2e^{i\lambda \theta\omega\cdot z}d\theta, 
	\end{align}
the property $S_1(v)=0$ and the formula  \eqref{free_resolvent2},
$S_1v(R_0^+(\lambda^4)-	R_0^-(\lambda^4))f$ is rewritten as 
	\begin{align}\label{S1v(1-theta)2}
-\frac{\pi}{2}\sum_{k=1}^3 & S_1v_k\int_{S^2}\widehat{\mathcal{R}_{k}f}(\lambda\omega)d\omega
-\frac{\pi i}{4}\sum_{k,l=1}^3\lambda S_1v_{kl}
		\int_{S^2}\widehat{\mathcal R_{kl}f}(\lambda\omega)d\omega\nonumber\\
		&+\frac{\pi}{4}\sum_{k,l,s=1}^3\lambda^2\int_0^1(1-\theta)^2S_1v_{kls}
		\int_{S^2} e_{\lambda \theta\omega}\widehat{\mathcal{R}_{kls}f}(\lambda\omega)d\omega d\theta.
	\end{align}
	Putting \eqref{S1v(1-theta)2} into \eqref{T_B}, we have
	\begin{align*}%\label{K2-21unbound}
	T_{B}
%&=i\sum_{k=1}^3\sum_{j=1}^{n_2}c_j^{k}\int_0^\infty \lambda \chi_{<a}(\lambda)\big(R_0^+(\lambda^4)vf_j\big)(x)\Big(  \int_{S^2} \widehat{ \mathcal{R}_{k}f}(\lambda\omega)d\omega\Big )d\lambda\nonumber\\
%		& \  -\frac{1}{2}\sum_{k,l=1}^3\sum_{j=1}^{n_2}c_j^{kl}\int_0^\infty \lambda^2 \chi_{<a}(\lambda)\big(R_0^+(\lambda^4)vf_j\big)(x)\Big(  \int_{S^2} \widehat{ \mathcal R_{kl}f}(\lambda\omega)d\omega\Big )d\lambda\nonumber\\
%		& \  -\frac{i}{2}\frac{2}{\pi i}\sum_{k,l,s=1}^3\int_0^1(1-\theta)^2
%		\int_{\R^6}v(\rho)(S_2A^2_{-2,1}S_1)(\rho, z)\  z_{kls} v(z) \Big(A_{\rho, \theta z}(\mathcal{R}_{kls} f)\Big)(x) dz d\rho d\theta\nonumber\\
		=i\sum_{k=1}^3\Psi^2_k-\frac{1}{2}\sum_{k,l=1}^3\Psi^2_{kl}-\frac{i}{2}\sum_{k,l,s=1}^3\Psi^2_{kls},
	\end{align*}
	where%, with $c_j^{k}=\langle S_2AS_1(vz_k), f_j \rangle$ and $c_j^{kl}=\langle S_2AS_1(vz_{kl}), f_j \rangle$, 
\begin{align}
\label{Psi^2_k}
\Psi^2_kf(x)&=\int_0^\infty \lambda\chi_{<a}(\lambda) \left( R_0^+(\lambda^4)vS_2AS_1v_{k}\right)(x)
		\left(\int_{S^2}\widehat{\mathcal R_{k}f}(\lambda\omega)d\omega\right) d\lambda,\\%\sum_{j=1}^{n_2}c_j^{k}\int_0^\infty \lambda \chi_{<a}(\lambda)\big(R_0^+(\lambda^4)vf_j\big)(x)\Big(  \int_{S^2} \widehat{ \mathcal{R}_{k}f}(\lambda\omega)d\omega\Big )d\lambda,\\
\nonumber
\Psi^2_{kl}f(x)&=\int_0^\infty \lambda^2\chi_{<a}(\lambda)\left( R_0^+(\lambda^4)vS_2AS_1v_{kl}\right)(x)
		\left(\int_{S^2}\widehat{\mathcal{R}_{kl}f}(\lambda\omega)d\omega\right) d\lambda,\\%\sum_{j=1}^{n_2}c_j^{kl}\int_0^\infty \lambda^2 \chi_{<a}(\lambda)\big(R_0^+(\lambda^4)vf_j\big)(x)\Big(  \int_{S^2} \widehat{ \mathcal R_{kl}f}(\lambda\omega)d\omega\Big )d\lambda,\\
\nonumber
\Psi^2_{kls}f(x)&=\int_0^1(1-\theta)^3\int_0^\infty \lambda^3\chi_{<a}(\lambda)
		\left[ R_0^+(\lambda^4)vS_2AS_1v_{kls}\int_{S^2} e_{\lambda \theta\omega}\widehat{\mathcal{R}_{kls}f}(\lambda\omega)d\omega\right](x)d\lambda d\theta.
\end{align}
$\Psi^2_{k}$, $\Psi^2_{kl}$ and $\Psi^2_{kls}$  have almost  analogous structures to that of $\Phi^2_{kl}$, $\Phi^2_{kls}$ and $\Phi^2_{klst}$ in Proposition \ref{proposition_5_1}, respectively (see \eqref{Phi^2_{kl}}--\eqref{Phi^2_{klst}}. Hence, one can prove by the same argument that $\Psi^{2}_{k}\in \mathbb B(L^p(\R^3))$ for all $1<p<3$ and $\Psi^2_{kl},\Psi^2_{kls}\in \mathbb B(L^p(\R^3))$ for $1<p<\infty$. 
%Hence, this $T_B$ has an essentially analogous structure to the one , so it is bounded on $L^p(\R^3)$ foe all $1<p<\infty$. 
%Hence, the same argument as that for $\Phi^2_{kl},\Phi^2_{kls},\Phi^2_{klst}$ in the previous proposition shows $T_B\in \mathbb{B}(L^p(\R^3))$ for all $1<p<\infty$. 

Similarly, to deal with the case $B=\lambda^{-1} S_2A$, we use \eqref{taylor} and \eqref{free_resolvent2} to observe
	\begin{align*}
	%\label{v(1-theta)}
v(R_0^+(\lambda^4)-	R_0^-(\lambda^4))f
&=\frac{\pi i}{2\lambda}v\int_{S^2}\widehat{f}(\lambda\omega)d\omega
		-\frac{\pi}{2}\sum_{k=1}^3 v_k\int_{S^2}\widehat{\mathcal{R}_{k}f}(\lambda\omega)d\omega\nonumber\\
		&-\frac{\pi i}{2}\sum_{k,l=1}^3\lambda v_{kl}\int_0^1(1-\theta)
		\int_{S^2} e_{\lambda \theta\omega}\widehat{\mathcal R_{kl}f}(\lambda\omega)d\omega d\theta.
	\end{align*}
Thus $T_{\lambda^{-1} S_2A}=\widetilde{\Psi}^2+i\sum_{k=1}^3\widetilde{\Psi}^2_{k}-\sum_{k,l=1}^3\widetilde{\Psi}^2_{kl}$ with
%	\begin{align}\label{K2-11unbound}T_B=\widetilde{\Psi}^2+i\sum_{k=1}^3\widetilde{\Psi}^2_{k}-\sum_{k,l=1}^3\widetilde{\Psi}^2_{kl},	\end{align}with 
%where, with $\tilde{c}_j=\langle S_2A^2_{-1,1}(v), f_j \rangle$ and $\tilde{c}_j^{k}=\langle S_2A^2_{-1,1}(vz_k), f_j \rangle$, 
\begin{align}
\label{widetildePsi^2}
\widetilde{\Psi}^2f(x)&=\int_0^\infty \lambda\chi_{<a}(\lambda)\left( R_0^+(\lambda^4)vS_2Av\right)(x)
		\left(\int_{S^2}\widehat{f}(\lambda\omega)d\omega\right) d\lambda,\\
%=\sum_{j=1}^{n_2}\tilde{c}_j\int_0^\infty \lambda \chi_{<a}(\lambda)\big(R_0^+(\lambda^4)vf_j\big)(x)\Big(  \int_{S^2} \widehat{f}(\lambda\omega)d\omega\Big )d\lambda,\\
\nonumber
\widetilde{\Psi}^2_{k}f(x)&=\int_0^\infty \lambda^2\chi_{<a}(\lambda)\left( R_0^+(\lambda^4)vS_2Av_{k}\right)(x)
		\left(\int_{S^2}\widehat{\mathcal{R}_{k}f}(\lambda\omega)d\omega\right) d\lambda,\\
		%\sum_{j=1}^{n_2}\tilde{c}_j^{k}\int_0^\infty \lambda^2 \chi_{<a}(\lambda)\big(R_0^+(\lambda^4)vf_j\big)(x)\Big(  \int_{S^2} \widehat{ \mathcal{R}_{k}f}(\lambda\omega)d\omega\Big )d\lambda,\\
\nonumber
\widetilde{\Psi}^2_{kl}f(x)&=\int_0^1(1-\theta)^3\int_0^\infty \lambda^3\chi_{<a}(\lambda)
		\left[ R_0^+(\lambda^4)vS_2Av_{kl}\int_{S^2} e_{\lambda \theta\omega}\widehat{\mathcal{R}_{kl}f}(\lambda\omega)d\omega\right](x)d\lambda d\theta. %\int_0^1(1-\theta)		\int_{\R^6}v(\rho)(S_2A^2_{-1,1})(\rho, z)\  z_{k}z_{l} v(z) \Big(A_{\rho, \theta z}(\mathcal R_{kl} f)\Big)(x) dz d\rho d\theta. 
\end{align}
Again, by  the same argument as above, we can prove $\widetilde{\Psi}^{2}\in \mathbb B(L^p(\R^3))$ for all $1<p<3$ and 
 $\widetilde{\Psi}^2_{kl},\widetilde{\Psi}^2_{k}\in \mathbb B(L^p(\R^3))$ for all  $1<p<\infty$. This completes the proof. 
\end{proof}

Finally,  we deal with the terms corresponding with $\lambda^{-2}S_1A^2_{-2,2}S_2$ and $\lambda^{-1}A^2_{-1,2}S_2$. %Let\begin{align}\label{K2-22}K^2_{-2,2}&:=\frac{2}{\pi i}\int_0^\infty \lambda \chi_{<a}(\lambda)\ R_0^+(\lambda^4)vS_1A^2_{-2,2}S_2 v\big(R_0^+(\lambda^4)-R_0^-(\lambda^4)\big )d\lambda,\\\label{K2-12}	K^2_{-1,2}&:=\frac{2}{\pi i}\int_0^\infty \lambda^2 \chi_{<a}(\lambda)\ R_0^+(\lambda^4)v A^2_{-1,2}S_2v\big(R_0^+(\lambda^4)-R_0^-(\lambda^4)\big )d\lambda.\end{align}

%proposition
\begin{proposition}\label{proposition_5_3}
Let $A\in \AB(L^2)$ and $B\in \{\lambda^{-2}S_1AS_2,\lambda^{-1}AS_2\}$. Then $T_B \in \mathbb{B}(L^p(\R^3))$ for all $1<p<\infty$.
\end{proposition}

\begin{proof}The proof is almost analogous to that of $\Psi^2_{kl}$ and $\Psi^2_{kls}$ in the previous proposition. Indeed, by the Taylor expansion \eqref{taylor(1-theta)2}, the properties $S_2(v)=S_2(z_jv)=0$ for $j=1,2,3$, and \eqref{free_resolvent2}, we can observe that $S_2v(R_0^+  (\lambda^4)-	R_0^-(\lambda^4))f$ is a linear combination of 
	\begin{align*}
		\lambda S_2v_{kl}\int_{S^2}\widehat{\mathcal R_{kl}f}(\lambda\omega)d\omega,\quad \lambda^2\int_0^1(1-\theta)^2S_2v_{kls}\int_{S^2} e_{\lambda \theta\omega}\widehat{\mathcal{R}_{kls}f}(\lambda\omega)d\omega d\theta,\quad k,l,s=1,2,3. 
\end{align*}They are essentially same as the ones on second and third terms of \eqref{S1v(1-theta)2}, respectively. 
Hence, we see that $T_{\lambda^{-2}S_1AS}f$ is a linear combination of the integrals
%	\begin{equation*}		\begin{split}			K^2_{-2,2}			&:=-\frac{1}{2}\sum_{k,l=1}^3\Phi^2_{kl}-\frac{i}{2}\sum_{k,l,s=1}^3\Phi^2_{kls},		\end{split}	\end{equation*}
\begin{align*}
&\int_0^\infty \lambda^2 \chi_{<a}(\lambda)\left(R_0^+(\lambda^4)vS_1AS_2v_{kl}\right)(x)\left(  \int_{S^2} \widehat{ \mathcal R_{kl}f}(\lambda\omega)d\omega\right )d\lambda,\\
&\int_0^1(1-\theta)^3\int_0^\infty \lambda^3\chi_{<a}(\lambda)
		\left[ R_0^+(\lambda^4)vS_1AS_2v_{kls}\int_{S^2} e_{\lambda \theta\omega}\widehat{\mathcal{R}_{kls}f}(\lambda\omega)d\omega\right](x)d\lambda d\theta.%\int_0^1(1-\theta)^2		\int_{\R^6}v(\rho)(S_1A^2_{-2,2}S_2)(\rho, z)\  z_{kls} v(z) \Big(A_{\rho, \theta z}(\mathcal{R}_{kls} f)\Big)(x) dz d\rho d\theta,
\end{align*}
They are essentially same as the entries of $\Psi^2_{kl}$ and $\Psi^2_{kls}$. Similarly, since $S_2v(R_0^+  (\lambda^4)-	R_0^-(\lambda^4))f$ can be written in the form
$$
-\frac{\pi i}{2}\sum_{k,l=1}^3\lambda \int_0^1(1-\theta)S_2v_{kl}\int_{S^2}e_{\lambda \theta\omega}\widehat{\mathcal R_{kl}f}(\lambda\omega)d\omega d\theta,
$$
$T_{\lambda^{-1}AS_2}f(x)$ is of the form
$$
-\frac{\pi i}{2}\sum_{k,l=1}^3\int_0^1(1-\theta)\int_0^\infty \lambda^3\chi_{<a}(\lambda)\left[R_0^+(\lambda^4)AS_2v_{kl}\int_{S^2}e_{\lambda \theta\omega}\widehat{\mathcal R_{kl}f}(\lambda\omega)d\omega \right](x) d\lambda d\theta
$$
 which is essentially same as $\widetilde \Psi^2_{kl}f(x)$ in the previous proposition. Hence, we have the proposition by the same argument as above based on Remark \ref{remark4.2}. 
\end{proof}

\subsection{The unboundedness for $3\le p\le \infty $}\label{section 5.2-un}

 In this subsection, we show that $W_-$ is unbounded on $L^p(\R^3)$ for any $3\leq p\le \infty$ in the case when zero is  the second kind resonance of $H$. At first, note that if $W_-\notin \mathbb{B}(L^p(\R^3))$ for some $3\leq p<\infty$, then $W_-\notin \mathbb{B}(L^\infty(\R^3))$;  otherwise, interpolating with the case $1<p<3$, one has $W_-\in \mathbb{B}(L^p(\R^3))$ for all $1< p<\infty$, leading to a contradiction.

Building upon the above argument, it becomes clear that most terms associated with \eqref{second resonance} are bounded on $L^p$ for  all $1 < p < \infty$. However, there are three specific terms corresponding to $$B=\lambda^{-3}S_2A^2_{-3,1}S_2,\ \lambda^{-2}S_2A^2_{-2,1}S_1,\ \lambda^{-1}S_2A^2_{-1,1},$$ more precisely the operators $$-\frac12\sum_{k,l=1}^3\Phi^2_{kl},\quad i\sum_{k=1}^3\Psi^2_k,\quad \widetilde \Psi^2$$ in Propositions \ref{proposition_5_1} and \ref{proposition_5_2}, 
for which we have established so far the $L^p$-boundedness only for $1 < p < 3$ (without using any specific structure of $A^2_{-3,1},A^2_{-2,1},A^2_{-1,1}$). Then we shall prove in fact the sum of them is not bounded on $L^p$ if $3\le p<\infty$ by using their specific structures. To this end, 
we observe that one can write 
$$
S_2A^2_{-1,1}=S_2H^2_{-1,1} + S_2H^2_{-1,2}Q
$$ with an explicit operator $H^2_{-1,1}$ and $H^2_{-1,2}\in \mathop{\mathrm{AB}}(L^2)$ (see \eqref{H^2} below). The term associated with $\lambda^{-1}S_2H^2_{-1,2}Q$ has essentially the same structure as that of the one associated with $\lambda^{-1}S_1A^1_{-1,1}S_1$, so can be shown to be bounded on $L^p$ for all $1 < p < \infty$ as in the previous section.  As a result, it suffices to establish that the following operator $\Lambda^2$ is unbounded on $L^p$ for any $3 \leq p < \infty$:
\begin{align}
	\label{Lambda2}	
	\Lambda^2=-\frac12\sum_{k,l=1}^3\Phi^2_{kl}(A^2_{-3,1})+i\sum_{k=1}^3\Psi^2_k(A^2_{-2,1})+\widetilde \Psi^2(H^2_{-1,1}),
%	=\frac{2}{\pi i}\int_0^\infty \lambda^3 \chi_{<a}(\lambda)R_0^+(\lambda^4)vBv\left(R_0^+(\lambda^4)-R_0^-(\lambda^4)\right )d\lambda.
\end{align}
where $\Phi^2_{kl}(A^2_{-3,1})$, $\Psi^2_k(A^2_{-2,1})$ and $\widetilde \Psi^2(H^2_{-1,1})$ denote the operators $\Phi^2_{kl}$, $\Psi^2_k$ and $\widetilde \Psi^2$ given by \eqref{Phi^2_{kl}}, \eqref{Psi^2_k} and \eqref{widetildePsi^2} with $A=A^2_{-3,1}$, $A^2_{-2,1}$ and $H^2_{-1,1}$, respectively. 
%%%%%%%%%%%%%%%%%%%%%%%%%%%%%%
\begin{proposition}	\label{proposition_5_4}
$\Lambda^2\notin \mathbb B(L^p(\R^3))$ for any $3\leq p<\infty$.
\end{proposition}
To prove Proposition \ref{proposition_5_4}, we  need to provide the following two Lemmas  \ref{Dg real}-\ref{calcuate xishu2}.
 \begin{lemma}\label{Dg real}
If $g\in L^2(\R^3)$ is real-valued, then  $\D_2g$ and $\D_1g$ are also real-valued, where $\D_1,\D_2\in \mathbb B(L^2(\R^3))$ are defined in  Definition \ref{definition_resonance} (ii) and (iii), respectively.
 \end{lemma}
 
\begin{proof}We first prove that $\D_2g$ is real-valued. To this end, we first observe that, for any $f\in L^2$, 
\begin{align}
\label{proposition_5_4_T_2}
\overline{S_2f}=S_2\overline{f},\quad \overline{T_2f}=T_2\overline f,
\end{align}
where $S_2,T_2$ are defined in Definition  \ref{definition_resonance}. Indeed, since $S_2L^2$ is invariant under the complex conjugation, namely $f\in S_2L^2$ if and only if $\overline f\in S_2L^2$, and hence has a real-valued orthonormal basis $\{f_j\}_{j=1}^{n_2}$ (See Remark \ref{real-valued basis}), we have
$$
\overline{S_2f}%=\sum_{j=1}^{n_2}\langle S_2f, f_j\rangle f_j =\sum_{j=1}^{n_2}\langle f, S_2f_j\rangle f_j
=\sum_{j=1}^{n_2}\overline{\langle f, f_j\rangle f_j}=\sum_{j=1}^{n_2}\langle \overline{f}, f_j\rangle f_j=S_2\overline{f}.
$$
To show the second part of \eqref{proposition_5_4_T_2}, we observe that
\begin{align}\label{T_2fbiaoda}
\nonumber
		T_2f(x)=&2\langle S_2f, |y|^2v\rangle S_2(|x|^2v)+\sum_{i,j=1}^34\langle S_2f,  y_iy_jv\rangle S_2(x_ix_jv)\\
		&+\frac{10}{3}\left(1-\frac{\langle \D_1Tv, Tv\rangle}{\|V\|_{L^1}}\right)\langle S_2 f, |y|^2v\rangle S_2(|x|^2v). 
\end{align}
Indeed, if we write the first term of the RHS in \eqref{T_2} by expanding $G_3(x,y)=|x-y|^4$ as 
%Specifically, note that $G_3(x,y)=|x-y|^4$ from \eqref{def-Gk} and $S_2(v)=S_2(x_jv)=0,$ then 
\begin{align*} 
&S_2vG_3vS_2f(x)\\
&=	(S_2v)(x)\int_{\R^3}\left\{ |x|^4-4x\cdot y(|x|^2+|y|^2)+2|x|^2|y|^2+4(x\cdot y)^2+|y|^4\right\}v(y)(S_2f)(y)dy,
\end{align*}
then the three terms associated with $|x|^4,4x\cdot y(|x|^2+|y|^2)$ and $|y|^4$ vanish identically thanks to the cancellation properties (see Lemma \ref{lemma_resonance}): 
\begin{align}
\label{cancelation}
S_2v=S_2v_k=0,\quad v,v_k\perp S_2L^2\quad \text{for}\quad k=1,2,3
\end{align}
Hence
\begin{align}
\label{T2-1}
S_2vG_3vS_2f=2\langle S_2f, |y|^2v\rangle S_2(|x|^2v)+\sum_{i,j=1}^34\langle S_2f,  y_iy_jv\rangle S_2(x_ix_jv).
\end{align}
Similarly, for the second term of  the RHS in \eqref{T_2}, since $G_0(x,y)=|x-y|^2$ and thus
\begin{align}\label{T2-2}	 
	G_1vS_2f(x)
	=\int_{\R^3}\big( |x|^2-2x\cdot y+|y|^2\big)v(y)(S_2f)(y)dy=\langle S_2f, |y|^2v\rangle
\end{align}
as above, recalling $v^2=V$, we have
\begin{align}\label{T2-3}
 S_2(vG_1v)^2S_2f(x)&= \langle S_2f, |y|^2v\rangle (S_2vG_1V)(x)\nonumber\\
&=\langle S_2f, |y|^2v\rangle (S_2v)(x)\int_{\R^3}\big( |x|^2-2x\cdot y+|y|^2\big)V(y)dy\nonumber\\
&=\|V\|_{L^1} \langle S_2f, |y|^2v\rangle S_2(|x|^2v).
\end{align}
Utilizing \eqref{T2-2} and \eqref{cancelation}, we also derive
\begin{align}\label{T2-4}
	S_2vG_1vT\D_1TvG_1vS_2f(x)&= \langle S_2f, |y|^2v\rangle (S_2v)(x)\int_{\R^3}\big( |x|^2-2x\cdot y+|y|^2\big)v(y)(T\D_1Tv)(y)dy\nonumber\\
&	=\langle S_2f, |y|^2v\rangle\langle T\D_1Tv, v\rangle S_2(|x|^2v)\nonumber\\
&=\langle S_2f, |y|^2v\rangle\langle \D_1Tv, Tv\rangle S_2(|x|^2v),
	\end{align}
where the last equality follows since $T= U+v(\Delta^2)^{-1}v$ is self-adjoint.
By \eqref{T_2} and combining with \eqref{T2-1}, \eqref{T2-3} and \eqref{T2-4}, we obtain \eqref{T_2fbiaoda}. Since $\overline{S_2f}=S_2\overline{f}$ and  $\langle \D_1Tv, Tv\rangle $  is real-valued  due to $\D_1$  being self-adjoint, it follows from \eqref{T_2fbiaoda} that $\overline{T_2f}=T_2\overline{f}$ for any $f\in L^2$. 
 
We now come to prove that $\D_2g$  is real-valued for any real-valued $g\in L^2$. Recall that $\D_2:S_2L^2\to S_2 L^2$ is the inverse of $T_2:S_2L^2\to S_2 L^2$. In particular, one has
$
S_2=T_2\D_2=\D_2T_2
$. 
Let $\D_2g=\varphi+i\psi$ with real-valued $\varphi,\psi\in S_2L^2$. Then, it follows from  \eqref{proposition_5_4_T_2} that $T_2\varphi,T_2\psi$ are real-valued. Hence,
$$
S_2g=T_2\D_2g=T_2(\varphi+i\psi)=T_2\varphi+iT_2\psi,
$$
which implies $T_2\psi=0$ since $S_2g$ is real-valued by \eqref{proposition_5_4_T_2}. We thus conclude $$
\psi=S_2\psi=\D_2T_2\psi=0
$$
and $\D_2g$ is real-valued. 

The proof for $\D_1g$ is essentially same as that for $\D_2g$ and one needs to check that $S_1,T_1$ satisfy the same property as \eqref{proposition_5_4_T_2}, where $T_1$ is \eqref{T_1}. To this end, it is enough to rewrite $T_1$ as
\begin{align*}
T_1f=\|V\|_{L^1}^{-1}\langle S_1f, Tv\rangle S_1(Tv)+\frac{2\|V\|_{L^1}}{3\cdot(8\pi)^2}\sum_{i=1}^3\langle S_1f,  y_iv\rangle S_1(x_iv)
\end{align*} 
by similar computations as above based on the properties $S_1v=0$ and $v\perp S_1L^2$. The remaining part is also same as above, and we thus omit it. 
\end{proof}

\begin{lemma}\label{calcuate xishu2}
	Let $\{f_j\}_{j=1}^{n_2}$ be a  real-valued  orthonormal basis of $S_2L^2$  and operators $A_{-3,1}^2, $  $A^2_{-2,1},$ 
	$H^2_{-1,1}$ be defined by \eqref{A^2}-\eqref{H^2}. Define	constants $c_{{j}}^{kl},$ $c_{{j}}^{k}$ and $c_{{j}}$  by 
		\begin{align*}
			&c_{{j}}^{kl}=\langle S_2A_{-3,1}^2S_2v_{kl}, f_{j} \rangle ,\ \ 
			 c_{{j}}^{k}=\langle S_2A^2_{-2,1}S_1v_k, f_{j} \rangle ,\ \ 
			c_{{j}}=\langle S_2H^2_{-1,1}v, f_{j} \rangle, \ \ k,l=1,2,3.
	\end{align*}
	Then the following identities hold:
		\begin{align}
			\label{2nd-xishu}
			&c_{{j}}^{kl}=4\pi\cdot5!(1-i)\langle \D_2v_{kl}, f_{j}\rangle,\ \
			 c_{{j}}^{k}=20i\langle \D_1(Tv), v_k\rangle  \langle \D_2(|z|^2v) , f_{j} \rangle,\nonumber \\
			&c_{{j}}=80\pi(1+i)\left(1-\frac{\langle \D_1Tv, Tv\rangle}{\|V\|_{L^1}}\right)\langle \D_2(|z|^2v) , f_{j} \rangle,
	\end{align}
	where the operators appearing in \eqref{2nd-xishu} are defined in Definition \ref{definition_resonance}.
\end{lemma}

\begin{proof}
It will be seen in the item (iii) in the proof of Theorem \ref{thm-main-inver-M} in Appendix \ref{expansion 6.2} that 
\begin{align*}
A^2_{-3,1}&=4\pi\cdot5!(1-i)\D_2,\quad
A^2_{-2,1}=20i\D_2vG_1vT\D_1,\\
H^2_{-1,1}&=80\pi(1+i)\|V\|^{-1}_{L^1}(\D_2vG_1v-\D_2vG_1vT\D_1T). 
\end{align*}

For $c_{{j}}^{kl}$, since $S_2\D_2=\D_2S_2=\D_2$ (see \eqref{orthog-relation-1}), we thus have $$c_{{j}}^{kl}=4\pi\cdot5!(1-i)\langle \D_2v_{kl}, f_{j}\rangle.$$
	
For $c_{{j}}^{k}$, since $ G_1(z,y)=|z-y|^2$, $S_2v=S_2x_lv=0$ and $S_2\D_2=\D_2S_2=\D_2$,  we obtain
	\begin{align}\label{cjk2}	 	 
		\D_2vG_1vT\D_1(v_k)&=\D_2v\int_{\R^3}\big( |z|^2-2z\cdot y+|y|^2\big)v(y)\big(T\D_1(v_{k})\big)(y)dy\nonumber\\
		&=	\big\langle \D_1(Tv), v_k\big\rangle  \D_2(|z|^2v).
	\end{align}	
Since $S_t\D_t= \D_ tS_t =\D_t$ for $t=1, 2$ (see \eqref{orthog-relation-1}),  we obtain the desired formula for $c_{{j}}^{k}.$
	
	It remains to deal with $c_{j}$. As above, we calculate 	\begin{align*}
	&	\D_2vG_1v^2=\D_2v\int_{\R^3}\big( |z|^2-2z\cdot y+|y|^2\big)v^2dy=\|V\|_{L^1}\D_2(|z|^2v),\\
	&\D_2vG_1vT\D_1Tv=\D_2v\int_{\R^3}\big( |z|^2-2z\cdot y+|y|^2\big)v(y)\big(T\D_1Tv\big)(y)dy=\big\langle \D_1(Tv), Tv\big\rangle\D_2(|z|^2v).
		\end{align*}
		Hence, we derive that 
			\begin{align*}
		\big(\D_2vG_1v-\D_2vG_1vT\D_1T\big)(v)=\|V\|_{L^1}\left(1-\frac{\langle \D_1Tv, Tv\rangle}{\|V\|_{L^1}}\right)\D_2(|z|^2v).
		\end{align*}
Combining with the above formula of $H^2_{-1,1}$, we obtain the desired result for $c_{j}$. 
\end{proof}

Next,  we are devoted to the  proof of Proposition  \ref{proposition_5_4}.
\begin{proof}[Proof of \rm{ Proposition  \ref{proposition_5_4}}]
Assume for contradiction $\Lambda^2\in \mathbb B(L^{p_0}(\R^3))$ for some $3\le p_0<\infty$. 
As above, we write $\Lambda^2=\widetilde \chi_{\ge 4a}(D)\Lambda^2+\widetilde \chi_{< 4a}(D)\Lambda^2$. Since $\widetilde \chi_{\ge 4a}(D)\in \mathbb B(L^p(\R^3))$ for all $1<p<\infty$ and all $a>0$, it is enough to show $\widetilde \chi_{\ge 4a}(D)\Lambda^2\notin \mathbb{B}(L^{p_0}(\R^3))$ for an $a>0$ small enough to achieve a contradiction. 

To this end, we first remove from $\widetilde \chi_{\ge 4a}(D)\Lambda^2$ several harmless terms. Recall that, by the same argument as for $\widetilde \chi_{\ge 4a}(D)\Phi^1_{k}$ in the proof of Proposition \ref{proposition_4_1}, we can rewrite $\widetilde \chi_{\ge 4a}(D)\Phi^2_{kl}(A^2_{-3,1})$ as
$$
\widetilde \chi_{\ge 4a}(D)\Phi^2_{kl}(A^2_{-3,1})=\Phi^{2,1}_{kl}(A^2_{-3,1})+\Phi^{2,2}_{kl}(A^2_{-3,1}),
$$
where $\Phi^{2,1}_{kl}$ and $\Phi^{2,2}_{kl}$ are given by \eqref{Phi^{2,1}_{kl}} and \eqref{Phi^{2,2}_{kl}} (with $A=A^2_{-3,1}$). In particular, $\Phi^{2,2}_{kl}(A^2_{-3,1})\in \mathbb B(L^p(\R^3))$ for all $1<p<\infty$. Since $\Psi^2_k(A^2_{-2,1})$ and $\widetilde \Psi^2(H^2_{-1,1})$ also have  similar structures to $\Phi^2_{kl}(A^2_{-3,1})$ (see \eqref{Phi^2_{kl}}, \eqref{Psi^2_k} and \eqref{widetildePsi^2}), we can obtain by exactly the same argument that
\begin{align*}
\widetilde \chi_{\ge 4a}(D)\Psi^2_k(A^2_{-2,1})&=\Psi^{2,1}_k(A^2_{-2,1})+\Psi^{2,2}_k(A^2_{-2,1}),\\
\widetilde \chi_{\ge 4a}(D)\widetilde \Psi^2(H^2_{-1,1})&=\widetilde \Psi^{2,1}(H^2_{-1,1})+\widetilde \Psi^{2,2}(H^2_{-1,1}),
\end{align*}
where $\Psi^{2,2}_k(A^2_{-2,1}),\widetilde \Psi^{2,2}(H^2_{-1,1})\in \mathbb B(L^p(\R^3))$ for all $1<p<\infty$ and 
\begin{align*}
\Psi^{2,1}_k(A^2_{-2,1})f(x)&=(m(D)vS_1A^2_{-2,1}S_1v_k)(x)\int_{\R^3}\widetilde m(\eta)\widehat{\mathcal R_kf}(\eta)d\eta,\\
\widetilde{\Psi}^{2,1}(H^2_{-1,1})f(x)&=(m(D)vS_2H^2_{-1,1}v)(x)\int_{\R^3}\widetilde m(\eta)\widehat{f}(\eta)d\eta,
\end{align*}
where  $m(\xi)=|\xi|^{-4}\widetilde\chi_{\ge 4a}(\xi)$  and  $\widetilde m(\eta)=|\eta|^{-1}\widetilde \chi_{<a}(\eta)$. Therefore,  it is enough to show 
$$
\widetilde \Lambda^2:=-\frac12\sum_{k,l=1}^3\Phi^{2,1}_{kl}(A^2_{-3,1})+i\sum_{k=1}^3\Psi^{2,1}_k(A^2_{-2,1})+\widetilde{\Psi}^{2,1}(H^2_{-1,1})\notin \mathbb{B}(L^{p_0}(\R^3)). 
$$

Next, we expand $S_2A^2_{-3,1}S_2v_{kl}$, $S_1A^2_{-2,1}S_1v_k$ and $S_2H^2_{-1,1}v$ by using the real-valued orthonormal basis $\{f_j\}_{j=1}^{n_2}$ of $S_2L^2$  (See Remark \ref{real-valued basis}) as:
\begin{align*}
S_2A^2_{-3,1}S_2v_{kl}=\sum_{j=1}^{n_2}c^{kl}_jf_j,\quad S_1A^2_{-2,1}S_1v_k=\sum_{j=1}^{n_2}c^k_jf_j,\quad S_2H^2_{-1,1}v=\sum_{j=1}^{n_2}c_jf_j,
\end{align*}
where 
$
c_{j}^{kl}=\langle S_2A_{-3,1}^2S_2v_{kl}, f_j \rangle$, $c_{j}^{k}=\langle S_2A^2_{-2,1}S_1v_k, f_j \rangle$ and $c_{j}=\langle S_2H^2_{-1,1}v, f_j \rangle$. It follows from these expressions that
\begin{align}
\nonumber
\widetilde \Lambda^2f&=\sum_{j=1}^{n_2}\ell_j(f) m(D)vf_j, \quad \ell_j(f)=\int_{\R^3}\F(h_j)(y)f(y)dy,\\
\nonumber
h_j(\eta)&=\widetilde m(\eta)\phi_{j}(\eta),\quad \phi_{j}(\eta)=-\frac{1}{2}\sum_{k,l=1}^3c_{j}^{kl}\frac{\eta_k\eta_l}{|\eta|^2}+i\sum_{k=1}^3c_{j}^{k}\frac{\eta_k}{|\eta|}+c_{j},
\end{align}
where recall that $m(\xi)=|\xi|^{-4}\widetilde\chi_{\ge 4a}(\xi)$  and  $\widetilde m(\eta)=|\eta|^{-1}\widetilde \chi_{<a}(\eta)$. We also note that
$$
\int_{\R^3}\widetilde m(\eta)\widehat{\mathcal R_{kl}f}(\eta)d\eta=\int_{\R^3}\F\left(\widetilde m(\eta)\eta_k\eta_l |\eta|^{-2}\right)(y)f(y)dy. 
$$

Now we shall show $\widetilde \Lambda^2\notin \mathbb{B}(L^{p_0}(\R^3))$ by contradiction.  Assume that $\widetilde \Lambda^2\in \mathbb{B}(L^{p_0}(\R^3))$. 

Under the above assumption, we claim that $\ell_j(f)$ is a continuous functional on $L^{p_0}(\R^3)$ for all $j=1,...,n_2$.  Indeed, taking into account that  $\{m(D)vf_j\}_{j=1}^{n_2}$ are mutually linearly independent on $L^{p_0}(\R^3)$ for any $a>0$ small enough (see Lemma \ref{lemma_technical}), we know 
by the Hahn-Banach theorem, there exist a sequence of  functions $\{g_k\}_{k=1}^{n_2}$ in $ L^{p_0'}(\R^3)$ with $p_0'=p_0/(p_0-1)$ such that 
$$\big\langle m(D)vf_j, g_k \big\rangle=0\  \text{for}\  j\neq k,
\ \ \text{and}\ \ \big\langle m(D)vf_k, g_k \big\rangle\neq0.$$
It yields that  $\big\langle\widetilde \Lambda^2f, g_k\big\rangle=\big\langle \ell_k(f) m(D)vf_k, g_k\big\rangle$. Since $\widetilde \Lambda^2\in \mathbb{B}(L^{p_0}(\mathbb{R}^3))$, one has
$$\big|\ell_k(f)\big| \left| \big\langle  m(D)vf_k, g_k\big\rangle\right|=\left|\big\langle \ell_k(f) m(D)vf_k, g_k\big\rangle\right|=\left| \big\langle\widetilde \Lambda^2f, g_k\big\rangle\right|\lesssim \|f\|_{L^{p_0}} \|g_k\|_{L^{p_0'}}.$$
Furthermore,  combining with $\big\langle m(D)vf_k, g_k \big\rangle\neq0,$ we have $|\ell_k(f)|\lesssim\|f\|_{L^{p_0}}$ for all $k=1,...,n_2.$  Therefore, $\ell_j(f)$ is a continuous functional on $L^{p_0}(\R^3)$  for all $j=1,...,n_2$. 
 
As a result,  by the duality, $\F h_{j}\in L^{p_0'}(\R^3)$ and hence $h_{j}\in L^{p_0}(\R^3)$ by Hausdorff-Young's inequality, for all $j=1,...,n_2.$ Since $|\eta|^{-1}\widetilde\chi_{<a}(\eta)\notin L^{p_0}(\R^3)$ if $p_0\ge3$ and every $\phi_{j}$ is homogeneous of degree zero and continuous on $\R^3\setminus\{0\}$, $\phi_{j}$ must vanish identically in order to ensure $h_{j}\in L^{p_0}(\R^3)$ for all $j=1,...,n_2,$ namely:
$$\phi_{j}(\eta)=-\frac{1}{2}\sum_{k,l=1}^3c_{j}^{kl}\frac{\eta_k\eta_l}{|\eta|^2}+i\sum_{k=1}^3c_{j}^{k}\frac{\eta_k}{|\eta|}+c_{j}\equiv0,\quad j=1,...,n_2,$$
where 
$
c_{j}^{kl}=\langle S_2A_{-3,1}^2S_2v_{kl}, f_j \rangle$, $c_{j}^{k}=\langle S_2A^2_{-2,1}S_1v_k, f_j \rangle,$  $c_{j}=\langle S_2H^2_{-1,1}v, f_j \rangle.$   Specifically, by virtue of Lemma \ref{calcuate xishu2}, we have
	\begin{align*}
		%\label{second-xishu}
	&c_{{j}}^{kl}=4\pi\cdot5!(1-i)\langle \D_2v_{kl}, f_{j}\rangle,\ \
	c_{{j}}^{k}=20i\langle \D_1(Tv), v_k\rangle  \langle \D_2(|z|^2v) , f_{j} \rangle,\nonumber \\
	&c_{{j}}=80\pi(1+i)\left(1-\frac{\langle \D_1Tv, Tv\rangle}{\|V\|_{L^1}}\right)\langle \D_2(|z|^2v) , f_{j} \rangle,
\end{align*}
where  $Tv$ is real-valued  since $T= U+v(\Delta^2)^{-1}v$ and $\langle \D_1Tv, Tv\rangle$ is also real-valued since $\D_1$ is self-adjoint. Moreover, by Lemma \ref{Dg real} and  the fact that $f_{j}$ are real-valued (See Remark \ref{real-valued basis}), the three terms $\langle \D_2v_{kl} , f_{j} \rangle$, $\langle \D_1(Tv), v_k\rangle$ and $\langle \D_2(|z|^2v) , f_{j} \rangle$  are real-valued for all $j,k,l$. Hence
\begin{align*}
\Im \phi_{j}(\eta)
=\sum_{k,l=1}^32\pi\cdot5!\<\D_2v_{kl} , f_{j} \>\frac{\eta_k\eta_l}{|\eta|^2}
			+80\pi\left(1-\frac{\langle \D_1Tv, Tv\rangle}{\|V\|_{L^1}}\right)\<\D_2(|z|^2v) , f_{j} \>:=\sum_{k,l=1}^3b_{j,kl}\frac{\eta_k\eta_l}{|\eta|^2}
\end{align*}
with $b_{j,kl}=2\pi\cdot5!\<\D_2(vz_{kl}),f_j\>$ for $k\neq\ell$ and 
\begin{align}\label{b_{j,kk}}
	b_{j,kk}=2\pi\cdot5!\<\D_2(vz_k^2) , f_{j} \>+80\pi\bigg(1-\frac{\langle \D_1Tv, Tv\rangle}{\|V\|_{L^1}}\bigg)\<\D_2(v|z|^2) , f_{j} \>.
\end{align}
Since $\phi_j\equiv0$ and $b_{{j},kl}=b_{{j},lk}$ for all $j,k,l$, we have $b_{{j},kl}=0$ for all $j,k,l$ (see Lemma \ref{equiv0}). In particular, by \eqref{b_{j,kk}},
\begin{align}\label{Dvz=0}
	3\<\D_2(vz_k^2) , f_{j} \>=-\bigg(1-\frac{\langle \D_1Tv, Tv\rangle}{\|V\|_{L^1}}\bigg)\<\D_2(v|z|^2),f_{j}\>
\end{align}
for all $j,k$. Since the RHS is independent of $k$, so is $\<\D_2(vz_k^2) , f_{j} \>$. Therefore, 
\begin{equation}\label{Dvz=3Dvzk}
	\<\D_2(v|z|^2) , f_{j} \>=\sum_{k=1}^3\<\D_2(vz_k^2) , f_{j} \>=3\<\D_2(vz_k^2) , f_{j} \>.
\end{equation}
Plugging \eqref{Dvz=3Dvzk} into  \eqref{Dvz=0}, we obtain that 
$$\<\D_2(vz_k^2) , f_{j} \>=-\bigg(1-\frac{\langle \D_1Tv, Tv\rangle}{\|V\|_{L^1}}\bigg)\<\D_2(vz_k^2),f_{j}\>.$$
Now it follows from Lemma \ref{DTv} that $ 1-\|V\|_{L^1}^{-1}\langle \D_1Tv, Tv\rangle\geq0$, which implies $\<\D_2(vz_k^2) , f_j \>=0$ for all $j,k$. 

Thus, by $\Im\phi_j\equiv0$ and Lemma \ref{equiv0}, we find $
\<\D_2(vz_{kl}) , f_{j} \>=0
$ for all $j,k,l$. Since $\D_2$ is invertible on $S_2L^2$ and $\{f_j\}$ forms an orthonormal basis of  $S_2L^2$, $vz_{kl}$ must vanish identically on $S_2L^2$, or equivalently, $S_2v_{kl}=0$ for all $k,l$. By virtue of Lemma \ref{lemma_resonance}, this means $S_2L^2=S_3L^2$ which clearly contradicts with the fact that $S_2L^2\neq \{0\}$ and  $S_3L^2=\{0\}$ under the assumption that zero is  the second kind resonance of $H$ (see Subsection \ref{subsection_resonance}). 

Hence $\widetilde \Lambda^2\notin \mathbb{B}(L^{p_0}(\R^3))$ by contradiction. Therefore, $\Lambda^2\notin \mathbb B(L^p(\R^3))$ for any $3\le p<\infty$. 
\end{proof}

%\begin{color}{red}
With Propositions \ref{proposition_5_1}--\ref{proposition_5_4} at hand, 
we have finished the proof of Theorem \ref{theorem1.2} for the second kind resonance case (ii).
%\end{color}

%%%%%%%%%%%%%%%
\section{The third kind resonance case}
\label{The third kind6}
In this section, we prove the third kind  resonance case (iii) of Theorem \ref{theorem1.2}. Although the proof  involves somewhat more complicated computations than the second kind resonance case, the strategy of the proof is almost analogous. 

%The section is divided into two subsections to address the boundedness and unboundedness of wave operators.
\subsection{The boundedness for $1<p<3$}
Here we investigate the $L^p$ boundedness of the wave operators $W_-$ for $1<p<3$. 
Recall the  expansion of $M^{-1}(\lambda)$ in the third kind resonance case:
\begin{align}
\nonumber
&M^{-1}(\lambda)=\lambda^{-4} S_3 A_{-4,1}^3 S_3+\lambda^{-3} S_2A^3_{-3,1}S_2
+\lambda^{-2}\Big(S_2A^3_{-2,1}S_1 + S_1A^3_{-2,2}S_2\Big)\\
	\label{third_resonance}
&+\lambda^{-1}\Big(S_2A^3_{-1,1}+ A^3_{-1,2}S_2+S_1A^3_{-1,3}S_1\Big)
	+S_1A_{0,1}^3 +A^3_{0,2}S_1 +QA^3_{0,3}Q+ \lambda A^3_{1,1}+\Gamma^3_2(\lambda).
\end{align}
All terms, except for $\lambda^{-4} S_3 A_{-4,1}^3 S_3$, can be dealt with the same argument as above.  Hence, it suffices to establish the following proposition. 
%We begin by recalling some notations (See Section \ref{subsection_notation}):\begin{itemize}	\item 	For $z=(z_1,z_2,z_3)\in \R^3$ and $k,l,s,t,m=1,2,3$, we set	\begin{align*}	&z_{kl}=z_kz_l,\  z_{kls}=z_kz_lz_s,\  z_{klst}=z_kz_lz_sz_t,\  z_{klstm}=z_kz_lz_sz_tz_m;\\			v_{k}(z)=z_kv(z),\ & v_{kl}(z)=z_{kl}v(z),\ v_{kls}(z)=z_{kls}v(z),\ v_{klst}(z)=z_{klst}v(z),\ v_{klstm}(z)=z_{klstm}v(z).	\end{align*}	\item  Denote  $\omega=(\omega_1,\omega_2,\omega_3)\in S^2$ and the Riesz transforms $\mathcal R=(\mathcal R_1,\mathcal R_2,\mathcal R_3)$. Moreover, we set	\begin{align*}		\mathcal{R}_{kl}=\mathcal{R}_k\mathcal{R}_l,\quad \mathcal{R}_{kls}=\mathcal{R}_k\mathcal{R}_l\mathcal{R}_s,		\quad \mathcal{R}_{klst}=\mathcal{R}_k\mathcal{R}_l\mathcal{R}_s\mathcal{R}_t,	\quad \mathcal{R}_{klstm}=\mathcal{R}_k\mathcal{R}_l\mathcal{R}_s\mathcal{R}_t\mathcal{R}_m.	\end{align*}
	%\item
%	For a fixed $\theta\in\R\setminus\{0\},$  we define the dilation operator $U_\theta$ by
	%$
%	U_\theta: u(x) \mapsto \theta^{-3} u(\theta^{-1} x).
%	$
	%\item For a fixed $\rho\in\R^3,$ we  define the translation operator $J_\theta$ by  $J_\theta: u(x)\mapsto u(x-\rho).$
	%\item We denote by $\mathop{\mathrm{AB}}(L^2)$ the family of absolutely bounded operators on $L^2$ (see Section \ref{subsection_notation}). \end{itemize}

\begin{proposition}	\label{S3A-413S3}
Let $B=\lambda^{-4} S_3 AS_3$ with $A\in \AB(L^2)$. Then $T_B\in \mathbb B(L^p(\R^3))$ for $1<p<3$.
\end{proposition}

\begin{proof}
As in the previous two sections, we will use the notations in Section \ref{subsection_notation} frequently. 	Let $f\in\mathcal{S}(\R^3)$. By the  Taylor expansion of $e^{i\lambda\omega\cdot z}$:
\begin{align}\label{taylor(1-theta)4}
e^{i\lambda\omega\cdot z}=\sum_{\ell=0}^4\frac{i^{\ell}}{\ell !}(\lambda\omega\cdot z)^{\ell}+\frac{i}{24}(\lambda\omega\cdot z)^{5} \int_0^1(1-\theta)^4e^{i\lambda \theta\omega\cdot z}d\theta,\quad z\in \R^3,\ \omega\in S^2,
\end{align}
the properties $S_3(v)=S_3(v_k)=S_3(v_{kl})=0$ for $k,l=1,2,3$, and  \eqref{free_resolvent2} 
yields
\begin{align*}
&S_3v(R_0^+(\lambda^4)- R_0^-(\lambda^4))f\\
&=\frac{\pi}{12}\sum_{k,l,s=1}^3\lambda^2S_3v_{kls}\int_{S^2}\widehat{\mathcal{R}_{kls}f}(\lambda\omega)d\omega
			+\frac{\pi i}{48}\sum_{k,l,s,t=1}^3\lambda^3S_3v_{klst}
\int_{S^2}\widehat{\mathcal{R}_{klst}f}(\lambda\omega)d\omega\\
			&-\frac{\pi}{48}\sum_{k,l,s,t,m=1}^3\lambda^4\int_0^1(1-\theta)^4S_3 v_{klstm}
			\int_{S^2} e^{i\lambda \theta\omega\cdot z}\widehat{\mathcal{R}_{klstm}f}(\lambda\omega)d\omega d\theta.
\end{align*}
Thus, $T_B$ is a linear combination of 
%	\begin{align*}%\label{K3-41unbound}		K^3_{-4,1}	=:-\frac{i}{6}\sum_{k,l,s=1}^3\Phi^3_{kls}+\frac{1}{24}\sum_{k,l,s,t=1}^3\Phi^3_{klst}+\frac{i}{24}\sum_{k,l,s,t,m=1}^3\Phi^3_{klstm},	\end{align*}where
\begin{align}
\label{TKLS3}
\Phi^3_{kls}f&=\int_0^\infty \lambda \chi_{<a}(\lambda)R_0^+(\lambda^4)vS_3AS_3v_{kls}\left(  \int_{S^2} \widehat{ \mathcal{R}_{kls}f}(\lambda\omega)d\omega\right)d\lambda,\\
\label{TKLTS3}
\Phi^3_{klst}f(x)&=\int_0^\infty \lambda^2 \chi_{<a}\lambda)R_0^+(\lambda^4)vS_3AS_3v_{klst}\left(\int_{S^2} \widehat{ \mathcal{R}_{klst}f}(\lambda\omega)d\omega\right )d\lambda,\\
\label{TKLTS4}
\Phi^3_{klstm}f(x)&=\int_0^1(1-\theta)^4\int_0^\infty \lambda \chi_{<a}(\lambda)R_0^+(\lambda^4)vS_3AS_3v_{klstm}\int_{S^2} e_{\lambda\theta\omega}\widehat{ \mathcal{R}_{klstm}f}(\lambda\omega)d\omega d\lambda d\theta,
\end{align}
where $k,l,s,t,m=1,2,3$. 
By Remark \ref{remark4.2}, $\Phi^3_{klst},\Phi^3_{klstm} \in \mathbb B(L^p(\R^3))$ for all  $1<p<\infty$ and $1\le k,l,s,t,m\le 3$. Moreover, the formulas of $\Phi^3_{kls}$  are analogous to  \eqref{Phi^2_{kl}}. Hence we can prove $\Phi^3_{kls} \in \mathbb B(L^p(\R^3))$ for all  $1<p<3$ by the same argument as in Proposition \ref{proposition_5_1}.
\end{proof}

% It is worth noting that we can replace $n_3$ with $n_2$ in the definition of $\Phi^3_{kls}$ and $\Phi^3_{klst}$ in \eqref{TKLS3} and \eqref{TKLTS3}, respectively. Indeed, since  $\{f_j\}_{j=1}^{n_2}$ and $\{f_j\}_{j=1}^{n_3}$are the  standard orthogonal bases on  $S_2L^2$ and  $S_3L^2$, respectively, then  for all $n_3<l\le n_2$,$$S_3f_l=\sum_{j=1}^{n_3}\langle S_3f_l, f_j \rangle f_j=\sum_{j=1}^{n_3}\langle f_l, S_3f_j \rangle f_j=\sum_{j=1}^{n_3}\langle f_l, f_j \rangle f_j=0.$$

\subsection{The boundedness and unboundedness for $p\ge3$}\label{section 6.2-un}
In this subsection we show the remaining part of Theorem \ref{theorem1.2} (iii), namely 
\begin{itemize}
\item If $S_2L^2= S_3L^2=S_4L^2$, then $W_{\pm}(H,\Delta^2) \in \mathbb{B}(L^p(\R^3))$ for all  $3\le p<\infty$,
\item If $S_2L^2\neq S_3L^2$  or  $S_3L^2\neq  S_4L^2$, then
  	     $W_{\pm}(H,\Delta^2) \notin \mathbb{B}(L^p(\R^3))$ for any $3\le p\leq\infty$,
\end{itemize}
where $S_j$ with $j=1,..,4$ are defined in Definition \ref{definition_resonance} and \eqref{projectionS_4} (see also the characterization of $S_jL^2$ in Lemma \ref{lemma_resonance}). 
 %$W_{\pm}(H, \Delta^2)\in \mathbb{B}(L^p(\R^3))$ for all $3 \le p < \infty$ if and only if $S_2L^2=S_3L^2=S_4L^2$. And $W_{\pm}(H, \Delta^2)\notin\mathbb{B}(L^p(\R^3))$ for any $3 \le p \leq \infty$ if and only if $S_2L^2 \neq S_3L^2$ or $S_3L^2 \neq S_4L^2$. Here, $W_-\notin \mathbb{B}(L^\infty(\R^3))$ followed from the Marcinkiewicz interpolation theorem to lead a contradiction that assume  $W_-\in \mathbb{B}(L^\infty(\R^3))$.
For all the entries of \eqref{third_resonance}, except for the four terms
$$
\lambda^{-4}S_3A^3_{-4,1}S_3,\quad 
\lambda^{-3}S_2A^3_{-3,1}S_2,\quad \lambda^{-2}S_2A^3_{-2,1}S_1,\quad \lambda^{-1}S_2A^3_{-1,1},
$$
the associated operators $T_B$ (defined in \eqref{T_B}) can be shown to be bounded on $L^p$ for all $3\le p<\infty$ by the previous arguments. Precisely, taking into account the cancelation properties
$$
Q(v) = S_2(v)=S_3(v) = S_2(v_k)=S_3(v_k) = S_3(v_{kl}) = 0,\quad k,l=1,2,3,
$$
one can apply Proposition \ref{proposition_5_3} to $\lambda^{-2}S_1A^3_{-2,2}S_2$ and $\lambda^{-1}A^3_{-1,2}S_2$, Proposition \ref{proposition_4_1} to $\lambda^{-1}S_1A^3_{-1,3}S_1$ and Proposition \ref{proposition_4_3} to $S_1A_{0,1}^3$ and $A^3_{0,2}S_1$, respectively. For the terms $QA^3_{0,3}Q$, $\lambda A^3_{1,1}$ and $\Gamma^3_2(\lambda)$, the same argument as in the regular case applies. Moreover, we will observe in the proof of Theorem \ref{thm-main-inver-M} in Appendix \ref{Appendix A} (see the end of Appendix \ref{Appendix A}) that 
$$
S_2A^3_{-3,1}S_2 = S_2H^3_{-3,1}S_2 + S_2H^3_{-3,2}S_3,\quad 
S_2A^3_{-1,1} = S_2H^3_{-1,1} + S_2H^3_{-1,2}Q
$$
with some explicit operators $H^3_{-3,1}$, $H^3_{-1,1}$ and $H^3_{-3,2},H^3_{-1,2}\in \mathop{\mathrm{AB}}(L^2)$. Again, one can prove that $T_B$ associated with $B=\lambda^{-3} S_2H^3_{-3,2}S_3,\  \lambda^{-1}S_2H^3_{-1,2}Q$ are bounded on $L^p$ for all $3\le p<\infty$ by the same argument as that in Propositions \ref{proposition_4_1} (with the Taylor expansion \eqref{taylor3} for $B=\lambda^{-3} S_2H^3_{-3,2}S_3$). 
%We can further simplify our focus by noting that $S_2A^3_{-3,1}S_2 = S_2H^3_{-3,1}S_2 + S_2H^3_{-3,2}S_3$ and $S_2A^3_{-1,1} = S_2H^3_{-1,1} + S_2H^3_{-1,2}Q$, as derived from \eqref{A4321-1biaojithird}.  By utilizing the properties $Q(v) = S_1(v) = 0$, $S_3(v) = S_3(z_jv) = S_3(z_jz_kv) = 0$ $( j, k= 1,2, 3)$, and drawing from previous analysis, we can confirm that both the two terms corresponding with $\lambda^{-1}S_2H^3_{-1,2}Q$ and $\lambda^{-3}S_2H^3_{-3,2}S_3$ are bounded on $L^p$ for all $1 < p < \infty$. This implies that we only need to address the two terms corresponding with $\lambda^{-1}S_2H^3_{-1,1}$  and $\lambda^{-3}S_2H^3_{-3,1}S_2$ instead of   $\lambda^{-1}S_2A^3_{-1,1}$ and $\lambda^{-3}S_2A^3_{-3,1}S_2$, respectively.
 Hence, it remains to determine whether the following operator $\Lambda^3$ is bounded on $L^p$ for $3 \leq p < \infty$:
\begin{align*}%\label{Lambda3}
	\Lambda^3=\frac{2}{\pi i}\int_0^\infty \lambda^3 \chi_{<a}(\lambda) & R_0^+(\lambda^4)v
	\Big(  \lambda^{-4}S_3A^3_{-4,1}S_3+\lambda^{-3}S_2H^3_{-3,1}S_2+\lambda^{-2}S_2A^3_{-2,1}S_1\nonumber\\
	&+\lambda^{-1}S_2H^3_{-1,1}\Big) v\left(R_0^+(\lambda^4)-R_0^-(\lambda^4)\right ) d\lambda.
\end{align*}

As in the proof of the unboundedness of $W_-$ in  the second kind resonance case, applying the same argument as that for deriving the decomposition \eqref{eq4.10} based on the Taylor expansion of $e^{i\omega\cdot z}$ to each parts of $\Lambda^3$, we can decompose $\Lambda^3$ into a main term and a harmless term, where the latter term is bounded on $L^p$ for all $1<p<\infty$. Moreover, as for $\Lambda^2$, we expand such a main term by using the real-valued  orthonormal basis $\{f_j\}_{j=1}^{n_2}$ of $S_2L^2$ (See Remark \ref{real-valued basis}). 
As a consequence, $\Lambda^3$ can be written in the form $\Lambda^3=\Lambda^3_1+\Lambda^3_2$
with some $\Lambda^3_2\in \mathbb B(L^p(\R^3))$ for all $1<p<\infty$ and 
 \begin{align}%\label{Lambda_1^3f(x)}
	&\Lambda_1^3f(x)=\sum_{j=1}^{n_2}\bigg[
	\int_0^\infty \lambda \chi_{<a}(\lambda)\big(R_0^+(\lambda^4)vf_j\big)(x)  \int_{S^2}\Big(-\frac{i}{6}\sum_{k,l,s=1}^3c_{j}^{kls} \widehat{ \mathcal{R}_{kls}f}(\lambda\omega)\nonumber \\
&\ \ \ \ \ \ \ \ \ \ \ \ \ \ \ \  -\frac{1}{2}\sum_{k,l=1}^3c_{j}^{kl}\widehat{\mathcal R_{kl}f}(\lambda\omega)
+i\sum_{k=1}^3c_{j}^{k}\widehat{ \mathcal{R}_{k}f}(\lambda\omega)+c_{j}\widehat{f}(\lambda\omega)\Big)d\omega d\lambda\bigg],\nonumber
\end{align}
where $c_{j}^{kls}$, $ c_{j}^{kl}$, $ c_{j}^{k}$, and $ c_{j}$ are given as follows:
 \begin{align}
\label{xishu3}
&c_{j}^{kls}=\<S_3A_{-4,1}^3S_3v_{kls}, f_j \>,\quad
c_{j}^{kl}=\<S_2H_{-3,1}^3S_2v_{kl}, f_j \>,\nonumber\\ 
& c_{j}^{k}=\<S_2A^3_{-2,1}S_1v_k, f_j \>, \quad
c_{j}=\<S_2H^3_{-1,1}v, f_j \>.
\end{align}
 The  following Proposition \ref{propositionLambda_131} and  the above argument, show $W_-\in \mathbb B(L^p(\R^3))$ for all $3\le p<\infty$ if $S_2L^2= S_3L^2=S_4L^2$. While the  following Proposition \ref{propositionLambda_132} and  the above argument, show  $W_-\notin \mathbb B(L^p(\R^3))$ for any $3\le p\leq\infty$ if $S_2L^2\neq S_3L^2$  or  $S_3L^2\neq  S_4L^2.$
 
Before stating Proposition \ref{propositionLambda_131}-\ref{propositionLambda_132}, we need to focus on   the coefficients $c_{j}^{kls}$, $ c_{j}^{kl}$, $ c_{j}^{k}$, and $ c_{j}.$ 
 
% Before stating Proposition \ref{propositionLambda_131}, we establish the following Lemma \ref{calcuate xishu3}, which addresses the computation of the coefficients $c_{j}^{kls}$, $ c_{j}^{kl}$, $ c_{j}^{k}$, and $ c_{j}$ defined by \eqref{xishu3}. This lemma plays an important role in the proofs of the following Propositions \ref {propositionLambda_131} and \ref{propositionLambda_132}.
\begin{lemma} \label{calcuate xishu3}
 Let the coefficients $c_{j}^{kls}$, $ c_{j}^{kl}$, $ c_{j}^{k}$, and $ c_{j}$ be defined by \eqref{xishu3}. Then
 \begin{align}
 	c_{j}^{kls}
 	&=\langle \D_3(v_{kls}), f_j\rangle,\nonumber\\
 	c_{j}^{kl}
 	&= 4\pi(1-i)5!\big\langle(\D_2-\D_3vG_4v\D_2)(v_{kl}), f_j\big\rangle \nonumber\\
 	&\quad -\frac{1-i}{2\pi\cdot3!}\sum_{m=1}^3
 	\langle \D_2(v_{kl}), |z|^2v\rangle \langle  \D_1(Tv), z_mv\rangle  \langle \D_3(z_m|z|^2v), f_j \rangle ,\nonumber\\
 	c_{j}^{k}
 	&=-i\frac{\|V\|_{L^1}}{8\pi^2\cdot5!}\sum_{m=1}^3  \langle \D_1(v_k), z_mv \rangle
 	\langle \D_3(z_m|z|^2v), f_j \rangle\nonumber\\
 	&\quad +20i \langle  \D_1(Tv), v_k \rangle \big\langle (\D_2-\D_3vG_4v\D_2)(|z|^2v), f_j \big\rangle\nonumber\\
 	&\quad-\frac{i}{8\pi^2(3!)^2}\sum_{m=1}^3\langle  \D_1(Tv), v_k \rangle \langle\D_2(|z|^2v), |z|^2v \rangle
 	\langle  \D_1(Tv), z_mv \rangle \langle \D_3(z_m|z|^2v), f_j \rangle,\nonumber
    \end{align}
    \begin{align*}
    %\label{xishu31}
    c_{j}
 	&=(1+i)\bigg[80\pi\Big(1-\frac{\langle \D_1(Tv), Tv \rangle}{\|V\|_{L^1}}\Big)\big\langle(\D_2-\D_3vG_4v\D_2)(|z|^2v), f_j\big\rangle\nonumber\\
 	&\quad -\frac{1}{2\pi(3!)^2}\Big(1-\frac{\langle \D_1(Tv), Tv \rangle}{\|V\|_{L^1}}\Big)\sum_{m=1}^3\langle \D_2(|z|^2v), |z|^2v\rangle
 	\langle  \D_1(Tv), z_mv \rangle \langle \D_3(z_m|z|^2v), f_j \rangle \nonumber\\
 	&\quad +\frac{1}{2\pi\cdot5!}\sum_{m=1}^3\langle  \D_1(Tv), z_mv \rangle \langle \D_3(z_m|z|^2v), f_j \rangle \bigg],
 \end{align*}
where the relevant operators above are defined in Definition \ref{definition_resonance} or Appendix \ref{Appendix A}.
\end{lemma}
 
\begin{proof}
 	In the following, we sequentially compute  the coefficients $c_{j}^{kls}$, $ c_{j}^{kl}$, $ c_{j}^{k}$, and $ c_{j}$.
 	
For $c_{j}^{kls},$  by \eqref{A_3_1} and $S_3\D_3=\D_3S_3=\D_3$ from \eqref{orthog-relation-1},  we have  $c_{j}^{kls}
 	=\langle \D_3(v_{kls}), f_j\rangle.$
 	
For $c_{j}^{kl},$ combining with \eqref {A_3_5} and $S_t\D_s=\D_ sS_t =\D_s$ for $ t \leq s$ from \eqref{orthog-relation-1}, in order to get the desired result for 
 	 $c_{j}^{kl}$, it suffices to prove that 
 \begin{align*}	 
 \big\langle	\D_3vG_3v\D_1TvG_1v\D_2(v_{kl}), f_j \big\rangle=-4\sum_{m=1}^3
 	\langle  \D_2(v_{kl}), |z|^2v \rangle \langle \D_1(Tv), z_mv\rangle \langle \D_3(z_m|z|^2v), f_j\rangle.
 \end{align*}	 
 	Indeed,  since  $G_1(x, z)=|x-z|^2$ given by \eqref {def-Gk} and  $S_2v=S_2x_\ell v=0, $ $S_2\D_2=\D_2S_2=\D_2$, we have
 	\begin{align}\label{cjkl1}	 
 	G_1v\D_2(v_{kl})=\int_{\R^3}\big( |x|^2-2x\cdot z+|z|^2\big)v(z)\big(\D_2(v_{kl})\big)(z)dz=	\langle  \D_2(v_{kl}), |z|^2v \rangle.
 	\end{align}	
Furthermore, note that  $G_3(x,z)=|x-z|^4,$   $S_t\D_s=\D_ sS_t =\D_s$ for $ t \leq s$ from \eqref{orthog-relation-1}  and $S_1v=S_3v=S_3x_\ell v=S_3x_ix_\ell v=0 $, we know
 	    	\begin{align}\label{cjkl2}	 	 
 	    	\D_3vG_3v\D_1Tv&=	\D_3v\int_{\R^3}\big( |x|^4-4x\cdot z|x|^2+2|x|^2|z|^2+4(x\cdot z)^2-4x\cdot z|z|^2+|z|^4\big)v(z)(\D_1Tv)(z)dz\nonumber\\
 	    	&=-4\sum_{m=1}^3
 	    	\langle  \D_1(Tv), z_mv\rangle \D_3(x_m|x|^2v)=-4\sum_{m=1}^3
 	    	\langle  \D_1(Tv), z_mv\rangle \D_3(z_m|z|^2v).
 	    \end{align}	
 By virtue of 	\eqref{cjkl1} and 	\eqref{cjkl2}, we derive the desired result for  $c_{j}^{kl}.$
 	    
For  $c_{j}^{k},$ by replacing $Tv$ with $v_k$ in \eqref{cjkl2},	 we have
 	\begin{align}\label{cjk1}	 	 
 	\D_3vG_3v\D_1(v_k)
 	&=-4\sum_{m=1}^3
 	\langle  \D_1(v_k), z_mv\rangle \D_3(z_m|z|^2v).
 \end{align}	
 Moreover, by $\D_2vG_1vT\D_1(v_k)=\big\langle \D_1(Tv), v_k\big\rangle  \D_2(|z|^2v)$ from \eqref{cjk2},	 we derive
 	\begin{align}\label{cjk3}	 	 
 	&(\D_2vG_1vT\D_1-\D_3vG_4v\D_2vG_1vT\D_1)(v_k)\nonumber\\&=\big(\Id-\D_3vG_4v\big)\D_2vG_1vT\D_1(v_k)
 	=\big\langle \D_1(Tv), v_k\big\rangle\big(\D_2-\D_3vG_4v\D_2\big)(|z|^2v).
 \end{align}	
    Use \eqref{cjk2},  $S_2v=S_2v_\ell=0, $ and  $S_2\D_2=\D_2S_2=\D_2$ to get that 
    \begin{align*}
    	%\label{cjk4}	 	 
   & \D_1TvG_1v\D_2vG_1vT\D_1(v_k)=\big\langle \D_1(Tv), v_k\big\rangle  \D_1TvG_1v \D_2(|z|^2v)\nonumber\\
  &=  \big\langle \D_1(Tv), v_k\big\rangle \D_1Tv\int_{\R^3}\big( |x|^2-2x\cdot z+|z|^2\big)v(z)\big(\D_2(|z|^2v)\big)(z)dz\nonumber\\
  &=\big\langle \D_1(Tv), v_k\big\rangle\langle \D_2(|z|^2v), |z|^2v\rangle \D_1Tv.
    \end{align*}	
 Combining with    \eqref{cjkl2},  it follows that
  \begin{align}\label{cjk5}	 	 
 	&\D_3vG_3v\D_1TvG_1v\D_2vG_1vT\D_1(v_k)\nonumber\\
 	&=\big\langle \D_1(Tv), v_k\big\rangle\langle \D_2(|z|^2v), |z|^2v, \rangle \D_3vG_3v\D_1Tv\nonumber\\
 	&=-4\sum_{m=1}^3\big\langle \D_1(Tv), v_k\big\rangle\langle \D_2(|z|^2v), |z|^2v\rangle
 	\langle \D_1(Tv), z_mv\rangle \D_3(z_m|z|^2v).
 \end{align}	
 Taking into account that \eqref{cjk1}, \eqref{cjk3}, \eqref{cjk5}	and  \eqref{A_3_3}, 	we derive the desired result for  $c_{j}^{k}.$ 
 
For $c_{j},$ by substituting $T\D_1v_k$ with $v$ and  substituting $v_k$ with $Tv$ in \eqref {cjk3}, respectively,  it is not hard to find that 
 \begin{align*}
&(\D_2vG_1v-\D_3vG_4v\D_2vG_1v)(v)=\langle v, v\rangle\big(\D_2-\D_3vG_4v\D_2\big)(|z|^2v),\\
&(\D_2vG_1vT\D_1T-\D_3vG_4v\D_2vG_1vT\D_1T)(v)
=\big\langle \D_1(Tv), Tv\big\rangle\big(\D_2-\D_3vG_4v\D_2\big)(|z|^2v).
    \end{align*}
    Hence, note that $\langle v, v\rangle=\|V\|_{L^1},$ then
    \begin{align}\label{cj1}
    	&\big(\D_2vG_1v-\D_3vG_4v\D_2vG_1v-\D_2vG_1vT\D_1T+\D_3vG_4v\D_2vG_1vT\D_1T\big)(v)\nonumber\\
    	&=\|V\|_{L^1}\Big(1-\frac{\langle \D_1(Tv), Tv \rangle}{\|V\|_{L^1}}\Big)(\D_2-\D_3vG_4v\D_2)(|z|^2v).
    	\end{align}
Additionally, by replacing $T\D_1v_k$ with $v$ and  replacing $v_k$ with $Tv$ in \eqref {cjk5},	 	respectively,    one has 
\begin{align*}
	&\D_3vG_3v\D_1TvG_1v\D_2vG_1v(v)=-4\sum_{m=1}^3\big\langle v, v\big\rangle\langle \D_2(|z|^2v), |z|^2v\rangle
	\langle \D_1(Tv), z_mv\rangle \D_3(z_m|z|^2v),\\
	&\D_3vG_3v\D_1TvG_1v\D_2vG_1vT\D_1T(v)\\
	&\quad\quad\quad=-4\sum_{m=1}^3\big\langle \D_1(Tv), Tv\big\rangle\langle \D_2(|z|^2v), |z|^2v\rangle
	\langle \D_1(Tv), z_mv\rangle \D_3(z_m|z|^2v).
\end{align*}
Since $\langle v, v\rangle=\|V\|_{L^1},$ it follows that 
   \begin{align}\label{cj2}
	&\Big(\D_3vG_3v\D_1TvG_1v\D_2vG_1v-\D_3vG_3v\D_1TvG_1v\D_2vG_1vT\D_1T\Big)(v)\nonumber\\
	&=-4\|V\|_{L^1}\Big(1-\frac{\langle \D_1(Tv), Tv \rangle}{\|V\|_{L^1}}\Big)\sum_{m=1}^3\langle \D_2(|z|^2v), |z|^2v\rangle
	\langle \D_1(Tv), z_mv\rangle \D_3(z_m|z|^2v).
\end{align}
Taking into account  \eqref{cj1}, \eqref{cj2}, \eqref{cjkl2}	and \eqref{A_3_6}, we derive the desired result for  $c_{j}.$ 
  \end{proof}

\begin{proposition}\label{propositionLambda_131}
If $S_2L^2= S_3L^2=S_4L^2$, then $\Lambda_1^3$ vanishes identically. 
\end{proposition}

\begin{proof}
Now, we recall that $\langle x_jx_kx_lv, f\rangle =0$ for all $f\in S_4L^2$ and $j,k,l=1,2,3$. Additionally, note that $\D_1, \D_2, \D_3$ are self-adjoint  and $S_t\D_s = \D_sS_t =S_t, $ if $t>s;$  $S_t\D_s= \D_ sS_t =\D_s,$ if $ t \leq s$ from \eqref{orthog-relation-1}. Therefore, if $S_2L^2=S_3L^2=S_4L^2$, then all of the inner products against $f_j$ in Lemma \ref {calcuate xishu3} must vanish, yielding that $c_{j}^{kls}\equiv c_{j}^{kl}\equiv c_{j}^{k}\equiv c_{j}\equiv 0$. This shows $\Lambda^3_1\equiv0$.
\end{proof}

%In particular, if $S_2L^2= S_3L^2=S_4L^2$, then $\Lambda^3=\Lambda_2^3+\Lambda_3^3,$ {\bf which implies that $\Lambda^3\in\mathbb{B}(L^p(\R^3))$ for all  $1< p<\infty$,} due to  $\Lambda_2^3, \Lambda_3^3\in\mathbb{B}(L^p(\R^3))$ for all  $1< p<\infty$.

It remains to deal with the case $S_2L^2\neq S_3L^2$  or  $S_3L^2\neq  S_4L^2$. 

\begin{proposition}\label{propositionLambda_132}
	If  $S_2L^2\neq S_3L^2$  or  $S_3L^2\neq  S_4L^2$, then $\Lambda_1^3\notin\mathbb B(L^p(\R^3))$ for any $3\le p<\infty$.
\end{proposition}

\begin{proof}
The proof is almost analogous to that of Proposition \ref{proposition_5_4}. It suffices to prove 
$\widetilde \chi_{\ge 4a}(D)\Lambda_1^3\notin \mathbb{B}(L^p(\R^3))$ for any $3\leq p<\infty$. As done for $\widetilde \chi_{\ge 4a}(D)\Lambda^2$, $\widetilde \chi_{\ge 4a}(D)\Lambda_1^3$ can be decomposed as the sum of $\widetilde \Lambda_{1}^3$ defined below, and an operator given as a (finite) linear combination of the operators:
\begin{equation*}
			\int_{\R^3}v(\rho)f_j(\rho)\big(m(D)A_{\rho,0} |D|^2\widetilde \chi_{<4a}(D)g_j(\mathcal{R})f\big)(x)d\rho
\end{equation*}
with some polynomial $g_j$, which are bounded on $L^p(\R^3)$ for all $1<p<\infty$. Here $\widetilde \Lambda_{1}^3$ is given by
	 \begin{align*}
			\widetilde \Lambda_{1}^3f(x)=\sum_{j=1}^{n_2}\ell_j(f) m(D)vf_j,\quad
			\ell_j(f)=\int_{\R^3}\F(h_j)(y)f(y)dy,\quad
			h_j(\eta)=\widetilde m(\eta)\phi_j(\eta),
	\end{align*}
where $m(\xi)=|\xi|^{-4}\widetilde\chi_{\ge 4a}(\xi),$  \  $\widetilde m(\eta)=|\eta|^{-1}\widetilde \chi_{<a}(\eta),$\  and
\begin{align*}%\label{gggggg_j}
  \phi_j(\eta)=-\frac{i}{6}\sum_{k,l,s=1}^3c_{j}^{kls}\frac{\eta_{k}\eta_{l}\eta_{s}}{|\eta|^3}
			-\frac{1}{2}\sum_{k,l=1}^3c_{j}^{kl}\frac{\eta_{k}\eta_{l}}{|\eta|^2}+i\sum_{k=1}^3c_{j}^{k}\frac{\eta_{k}}{|\eta|}
			+c_{j}.
  \end{align*}
 Thus, it remains to show $\widetilde \Lambda_{1}^3\notin\mathbb{B}(L^p(\R^3))$ for any $3\leq p<\infty$ provided $S_3L^2\neq S_2L^2$ or $S_3L^2\neq S_4L^2$. 
 
 To this end, we assume for contradiction that $\widetilde \Lambda_{1}^3\in\mathbb{B}(L^{p_0}(\R^3))$ for some $3\leq p_0<\infty$.
 
  As for $\widetilde \Lambda^{2}$, this assumption implies $\phi_j\equiv0$ for all $j=1,...,n_2$. Noticing that $ \D_3v_{kls}$, $ \D_2v_{kl}$ and $ \D_1v_k$  are real-valued for each $k,l,s=1,2,3,$ shown by the same way as Lemma \ref{Dg real} and the fact that $f_{j}$ are real-valued (see Remark \ref{real-valued basis}). Furthermore, by Lemma \ref{calcuate xishu3}, one can calculate
 	\begin{equation*}%\label{Img_j(eta)}
		\Im \phi_j(\eta)=-\frac{1}{6}\sum_{k,l,s=1}^3\<\D_3v_{kls}, f_j\> \frac{\eta_{k}\eta_{l}\eta_{s}}{|\eta|^3}
		+\sum_{k,l=1}^3b_{j,kl}\frac{\eta_{k}\eta_{l}}{|\eta|^2},
	\end{equation*}
	where $j=1,\cdots,n_2$, and  the coefficients $b_{j,kl}$ for $k\neq l$ and $b_{j,kk}$ are given as follows:
	\begin{align*}%\label{b_j,kl}
b_{j,kl}&=2\pi\cdot5!\langle(\D_2-\D_3vG_4v\D_2)(vz_{kl}), f_j\>\nonumber\\
&\quad  -\frac{1}{4\pi\cdot3!}\sum_{m=1}^3 \langle|z|^2v, \D_2(vz_{kl}) \> \<z_mv, \D_1(Tv) \> \<\D_3(z_m|z|^2v), f_j \>,\\
%		\label{b_j,kk}
b_{j,kk}&=2\pi\cdot5!\langle(\D_2-\D_3vG_4v\D_2)(vz_k^2), f_j\>\nonumber\\
&\quad -\frac{1}{4\pi\cdot3!}\sum_{m=1}^3\langle|z|^2v, \D_2(vz_k^2) \> \<z_mv, \D_1(Tv) \> \<\D_3(z_m|z|^2v), f_j \>\nonumber\\
&\quad +80\pi\left(1-\frac{\langle \D_1(Tv), Tv \rangle}{\|V\|_{L^1}}\right)\langle(\D_2-\D_3vG_4v\D_2)(|z|^2v), f_j\>\nonumber\\
&\quad -\frac{1}{2\pi(3!)^2}\left(1-\frac{\langle \D_1(Tv), Tv \rangle}{\|V\|_{L^1}}\right)\sum_{m=1}^3\langle|z|^2v, \D_2(|z|^2v) \>\<z_mv, \D_1(Tv) \>\<\D_3(z_m|z|^2v), f_j \>\nonumber\\
&\quad +\frac{1}{2\pi\cdot5!}\sum_{m=1}^3\<z_mv, \D_1(Tv) \>\<\D_3(z_m|z|^2v), f_j \>.
	\end{align*}
By $\Im\phi_j(\eta)\equiv0$ and Lemma \ref{equiv0},  we conclude that $\<\D_3(vz_{kls}), f_j\>=\<vz_{kls}, \D_3f_j\>\equiv 0$ and $b_{j,kl}=0$ for all $j, k, l, s.$ Since $\D_3$ is invertible on $S_3L^2$ and $\{f_j\}_{j=1}^{n_3}$ is a basis of $S_3L^2$ (See Remark \ref{real-valued basis}), we have $S_3(vz_{kls})=0$ for all $k,l,s=1,2,3,$ and thus, by Lemma \ref{lemma_resonance} and \eqref{projectionS_4},  we derive that $$ S_3L^2=S_4L^2.$$ Furthermore, since $b_{j,kl}=0$ for all $j,k,l$, we have
\begin{align}
\label{bbb_j,kl}
b_{j,kl}&=2\pi\cdot5!\langle(\D_2-\D_3vG_4v\D_2)(vz_{kl}), f_j\>=0,\quad k\neq l;\\
%	\label{bbb_j,kk}
\nonumber 
b_{j,kk}&=2\pi\cdot5!\langle(\D_2-\D_3vG_4v\D_2)(vz_k^2), f_j\>\\
\nonumber
& \ \ \ \ \ +80\pi\Big(1-\frac{\langle \D_1(Tv), Tv \rangle}{\|V\|_{L^1}}\Big)\langle(\D_2-\D_3vG_4v\D_2)(|z|^2v), f_j\>=0.
\end{align}
In particular, for all $k=1,2,3$ and $j=1,..,n_2$, 
	\begin{equation}
\label{proof_proposition_6_3_3}
			3\langle(\D_2-\D_3vG_4v\D_2)(vz_k^2), f_j\> =-\Big(1-\frac{\langle \D_1(Tv), Tv \rangle}{\|V\|_{L^1}}\Big)\langle(\D_2-\D_3vG_4v\D_2)(|z|^2v), f_j\>,
	\end{equation}
which implies that $\langle(\D_2-\D_3vG_4v\D_2)(vz_k^2), f_j\>$
is independent of $k$. Thus, for all $k=1,2,3$, 
\begin{align}
\label{proof_proposition_6_3_4}
\langle(\D_2-\D_3vG_4v\D_2)(|z|^2v), f_j\>=3\langle(\D_2-\D_3vG_4v\D_2)(vz_k^2), f_j\>. 
\end{align}
Take into account that $ 1-\|V\|_{L^1}^{-1}\langle \D_1Tv, Tv\rangle\ge 0$ from Lemma \ref{DTv}, then
\eqref{proof_proposition_6_3_3} and \eqref{proof_proposition_6_3_4} imply  $\langle(\D_2-\D_3vG_4v\D_2)(vz_k^2), f_j\>\equiv0$. This, combined with \eqref{bbb_j,kl}, shows 
$$
\<vz_{kl},(\D_2-\D_2vG_4v\D_3)f_j\>=\langle(\D_2-\D_3vG_4v\D_2)(vz_{kl}), f_j\>\equiv 0,\ \text{for all}\  j,k,l,
$$
which by the definition of $S_3L^2$, implies that  $\tilde{f}_j:=(\D_2-\D_2vG_4v\D_3)f_j\in S_3L^2$ for all $j$. Furthermore, since $S_3L^2=\Ker T_2$ (see Definition \ref{definition_resonance}), one has
	\begin{equation*}
		\begin{split}
			0=T_2S_3\tilde{f}_j=T_2\tilde{f}_j&=T_2(\D_2-\D_2vG_4v\D_3)f_j.
		\end{split}
	\end{equation*}
Note that by  Definition \ref{definition_resonance} and \eqref{orthog-relation-1}, we have  $(T_2+S_3)\D_2=S_2$, $S_3\D_2=S_3$, $S_3\D_3=\D_3$ and  $S_3\D_2vG_4v\D_3=S_3vG_4vS_3\D_3=S_3$.  Hence  it follows that
	\begin{equation*}
		\begin{split}
			0&=T_2(\D_2-\D_2vG_4v\D_3)f_j\\
			&=(T_2+S_3)\D_2f_j-S_3\D_2f_j-(T_2+S_3)\D_2vG_4v\D_3f_j+S_3\D_2vG_4v\D_3f_j\\
			&=f_j-S_3f_j-S_2vG_4v\D_3f_j+S_3f_j=f_j-S_2vG_4v\D_3f_j.
		\end{split}
	\end{equation*}
	Hence $f_j=S_2vG_4v\D_3f_j$ for each $j=1,\cdots,n_2$, which implies that  ${\rm Ran}(S_2vG_4v\D_3)=S_2L^2(\R^3)$, since $\{f_j\}_{j=1}^{n_2}$ is a basis of $S_2L^2$.  Thus we have
	$$\dim S_2L^2=\dim S_2vG_4v\D_3L^2=\dim S_2vG_4vS_3\D_3L^2\leq\dim S_3L^2<\infty,$$
which yields  $S_2L^2=S_3L^2$ since $S_3L^2\subset S_2L^2$ by definition.

As a result, we conclude $S_2L^2=S_3L^2=S_4L^2$, leading a contradiction. Hence $\widetilde \Lambda^3\notin \mathbb B(L^p(\R^3))$ for any   $3\le p<\infty$, which means that $\Lambda_1^3\notin\mathbb B(L^p(\R^3))$ for any $3\le p<\infty$.
\end{proof}

%In summary, we have finished the proof of Theorem \ref{theorem1.2} for case (iii).
%Indeed, on the one hand, if $S_2L^2=S_3L^2=S_4L^2$, then
%by Proposition \ref{propositionLambda_131},  operator $\Lambda^3=\Lambda_2^3\in \mathbb{B}(L^p(\R^3))$  for  all $1 < p < \infty$, which simultaneously  yields $W_-$ is also.
%On the other hand, if $S_3L^2\neq S_2L^2$ or $S_3L^2\neq S_4L^2$, then  by Proposition \ref{propositionLambda_132},  operator $\Lambda^3\notin \mathbb{B}(L^p(\R^3))$ for  any $3\leq p < \infty,$ which simultaneously  yields  $W_- $ is also. Additionally,  $W_-\notin \mathbb{B}(L^\infty(\R^3)).$ Otherwise, by the Marcinkiewicz interpolation theorem, $W_-\in \mathbb{B}(L^p(\R^3))$ for all $1< p<\infty$, a contradiction.

\section{The proofs of Lemma \ref{lem2.2} and two technical results}
\label{Section_technical_lemma}

In this section,  we prove  Lemma \ref{lem2.2} and provide two technical  Lemmas \ref{lemma_technical} and \ref{equiv0} used in the proof of the unboundedness of $W_-$.

%%%%%%%%%%%%%%%%%%
\subsection{The proof of Lemma \ref{lem2.2}}

Here we prove Lemma \ref{lem2.2} by establishing the following four propositions. Recalling that $\widetilde \chi$ was defined in Section \ref{subsection_notation}, we set
$$m_1(\xi,\eta)=\frac{\widetilde \chi_{<4a}(\xi)\widetilde \chi_{<a}(\eta)}{|\xi|^2+|\eta|^2},\quad   m_2(\xi,\eta) =\delta_\xi \otimes\frac{\widetilde \chi_{<a}(\eta)}{|\eta|},\quad m_3(\xi,\eta)= \frac{\widetilde \chi_{<4a}(\xi)\widetilde \chi_{<a}(\eta)}{|\xi|+|\eta|}.$$

\begin{proposition}\label{L1,L2,L3} Let $$L_j(x,y)=\int_{\mathbb{R}^6}e^{i\xi x-i\eta y}m_j(\xi,\eta)d\xi d\eta, \ j=1,2,3. $$
Then  $|L_1(x,y)|\lesssim \langle|x|+|y|\rangle^{-4},$ \ $|L_2(x,y)|\lesssim\langle y\rangle^{-2}$
	and  $|L_3(x,y)|\lesssim \langle x\rangle^{-1}\langle y\rangle^{-1}\langle|x|+|y|\rangle^{-3}.$
\end{proposition}
\begin{proof}
(i) $L_1(x,y)$ can be written as
	\begin{equation*}
		\begin{split}
			L_1(x,y)=\int_0^\infty\Big(\int_{\mathbb{R}^3}e^{i\xi x-s|\xi|^2}\widetilde \chi_{<4a}(\xi)d\xi\Big)
			\Big(\int_{\mathbb{R}^3}e^{-i\eta y-s|\eta|^2}\widetilde\chi_{<a}(\eta)d\eta\Big)ds.
		\end{split}
	\end{equation*}
		It is not hard to see that the functions inside parentheses above are convolutions of Gauss-Weierstrass kernels with Schwartz functions
	 $\F\widetilde \chi_{<a}$  and  $\F\widetilde \chi_{<4a}$ ,  respectively.
	
	 Hence, for any $N>0$, 
	\begin{equation*}
		\begin{split}	
 \left|L_1(x,y)\right|&\lesssim\left|\int_{\mathbb{R}^6}(\F\widetilde\chi_{<4a})(x-w)(\F\widetilde\chi_{<a})(y-z)
		\Big(	\int_0^\infty\frac{1}{s^3}e^{-\frac{|w|^2+|z|^2}{4s}}ds\Big)dwdz\right|\\
			&\lesssim\bigg|\int_{\mathbb{R}^6}
			\frac{(\F\widetilde\chi_{<4a})(x-w)(\F\widetilde\chi_{<a})(y-z)}{(|w|^2+|z|^2)^2}dwdz\bigg|\\
          & \lesssim\int_{\mathbb{R}^6}\<x-w\>^{-N}\<y-z\>^{-N}(|w|^2+|z|^2)^{-2}dwdz \\
	&\lesssim\int_{\mathbb{R}^6}\<(x-w,y-z)\>^{-N}|(w,z)|^{-4}dwdz
			\lesssim\<|x|+|y|\>^{-4}.		
		\end{split}
	\end{equation*}
(ii) For $L_2(x,y)$, similarly we have
	\begin{equation*}
		\begin{split}
			\big|L_2(x,y)\big|&=\bigg|\int_0^\infty
			\Big(\int_{\mathbb{R}^3}e^{-i\eta y-s|\eta|}\widetilde \chi_{<a}(\eta)d\eta\Big)ds\bigg|\\
			&\lesssim\bigg|\int_0^\infty\int_{\mathbb{R}^3}\frac{s}{(s^2+|z|^2)^2}\mathcal{F}\big(\widetilde \chi_{<a}\big)(y-z)dzds\bigg|
			\lesssim\int_{\mathbb{R}^3}|z|^{-2}\<y-z\>^{-N}dz\lesssim\<y\>^{-2}.
		\end{split}
	\end{equation*}
%	It is not hard to find that the function inside parenthesis is a convolution of the Poisson kernel with Schwartz function
%	$\mathcal{F}\big(\widetilde \chi_{<a}\big)(y)$.
%	Then, we have
%	\begin{equation*}
%		\begin{split}
%			| L_2(x,y)|&\lesssim\big|\int_0^\infty\int_{\mathbb{R}^3}\frac{s}{(s^2+|z|^2)^2}\mathcal{F}\big(\widetilde \chi_{<a}\big)(y-z)dzds\big|\\
%			&\lesssim\int_{\mathbb{R}^3}\frac{1}{|z|^2}\big|\mathcal{F}\big(\widetilde \chi_{<a}\big)(y-z)\big|dz.
%		\end{split}
%	\end{equation*}
%	Moreover note that $\mathcal{F}\big(\widetilde \chi_{<a}\big)(y)\in S(\mathbb{R}^3)$ and  $ |  L_2(x,y)|$ is bounded, it is easy to check that
%	\begin{equation*}
%		\begin{split}
%			| L_2(x,y)|&\lesssim\frac{1}{\langle y\rangle^2}.
%		\end{split}
%	\end{equation*}
(iii) $L_3(x,y)$ can be written as
	\begin{equation}\label{L_3}
		\begin{split}
			L_3(x,y)=\int_0^\infty\Big(\int_{\mathbb{R}^3}e^{i\xi x-s|\xi|}\widetilde \chi_{<4a}(\xi)d\xi\Big)
		\Big	(\int_{\mathbb{R}^3}e^{-i\eta y-s|\eta|}\widetilde\chi_{<a}(\eta)d\eta\Big)ds.
		\end{split}
	\end{equation}
By computing $\F\big(e^{-s|\cdot|}\big)(w)\cdot\F\big(e^{-s|\cdot|}\big)(z)$, we can see that $L_3(x,y)$ is dominated by
	\begin{equation*}
		\begin{split}
			&\bigg|\int_{\mathbb{R}^6}\mathcal{F}\big(\widetilde \chi_{<4a}\big)(x-w)\mathcal{F}\big(\widetilde \chi_{<a}\big)(y-z)
			\bigg(\int_0^\infty\frac{s^2}{\big(s^4+(|w|^2+|z|^2)s^2+|w|^2|z|^2\big)^2}ds\bigg)dwdz\bigg|.
		\end{split}
	\end{equation*}
	Let $\rho=(|w|^2+|z|^2)^{-\frac{1}{2}}s$ and $\gamma=|w||z|(|w|^2+|z|^2)^{-1}$.  Then
	\begin{equation*}
		\begin{split}
			\int_0^\infty&\frac{s^2}{\big(s^4+(|w|^2+|z|^2)s^2+|w|^2|z|^2\big)^2}ds
			=\frac{1}{(|w|^2+|z|^2)^{\frac{5}{2}}}
			\int_0^\infty\frac{\rho^2}{(\rho^4+\rho^2+\gamma^2)^2}d\rho\\
		& \lesssim\frac{1}{(|w|^2+|z|^2)^{\frac{5}{2}}}\Big(\int_0^\gamma\frac{\rho^2}{\gamma^4}d\rho+\int_\gamma^\infty\frac{\rho^2}{\rho^4}d\rho\Big)
			\lesssim\frac{1}{\gamma(|w|^2+|z|^2)^{\frac{5}{2}}}\lesssim\frac{1}{|w||z|(|w|+|z|)^3}.
		\end{split}
	\end{equation*}
	Hence choosing a enough large $N$,  we obtain
	\begin{equation}\label{L_3'}
		\begin{split}
			\big|  L_3(x,y)\big|\lesssim\int_{\mathbb{R}^6}\frac{1}{\langle x-w\rangle^N\langle y-z\rangle^N|w||z|(|w|+|z|)^3}dwdz.
		\end{split}
	\end{equation}
	To estimate the integral in \eqref{L_3'}, we decompose $\mathbb{R}^6$ into the four parts $\{\Sigma_i\times\Sigma _j^{'}\}_{i,j=1}^2,$ where
	\begin{equation*}
		\begin{split}
			&\Sigma_1=\{w\in\mathbb{R}^3\ |\ |w|\leq1/2\},\quad \Sigma_2=\{w\in\mathbb{R}^3\ |\ |w|>1/2\},\\
			&\Sigma_1^{'}=\{z\in\mathbb{R}^3\ |\ |z|\leq1/2\},\quad \Sigma_2^{'}=\{z\in\mathbb{R}^3\ |\ |z|>1/2\}.
		\end{split}
	\end{equation*}
	Then one has
	\begin{equation*}
		\begin{split}
			\big|  L_3(x,y)\big|\lesssim \sum_{i,j=1}^2\int_{\Sigma_i\times\Sigma _j^{'}}\frac{1}{\langle x-w\rangle^N\langle y-z\rangle^N|w||z|(|w|+|z|)^3}dwdz=:\sum_{i,j=1}^2 L_3^{ij}(x,y).
		\end{split}
	\end{equation*}
Thus, it is enough to show that each of $\{L_3^{ij}(x,y)\}_{i,j=1}^2$ is dominated by
	$C\langle x\rangle^{-1}\langle y\rangle^{-1}\langle|x|+|y|\rangle^{-3}$.
	
For $L_3^{11}(x,y)$, we use $\langle x-w\rangle\sim\langle x\rangle, \langle y-z\rangle\sim\langle y\rangle$ for $|w|\leq\frac{1}{2}$ and $|z|\leq\frac{1}{2}$, to obtain
	\begin{equation*}
		\begin{split}
			L_3^{11}(x,y)&=\int_{\Sigma_1\times\Sigma _1^{'}}\frac{1}{\langle x-w\rangle^N\langle y-z\rangle^N|w||z|(|w|+|z|)^3}dwdz\\
			&\lesssim\frac{1}{\langle x\rangle^N\langle y\rangle^N}\int_{|z|\leq\frac{1}{2}}\int_{|w|\leq\frac{1}{2}}\frac{1}{|w||z|(|w|+|z|)^3}dwdz\\
			&\lesssim\frac{1}{\langle x\rangle^N\langle y\rangle^N}\int_{|z|\leq\frac{1}{2}}\int_{|w|\leq\frac{1}{2}}\frac{1}{|w|^{\frac{5}{2}}|z|^{\frac{5}{2}}}dwdz
		\lesssim\frac{1}{\langle x\rangle^N\langle y\rangle^N}\lesssim\frac{1}{\langle x\rangle \langle y\rangle\<|x|+|y|\>^3},
		\end{split}
	\end{equation*}
for $N$ large enough, where we have used the bound $|w|+|z|\gtrsim\sqrt{|w||z|}$ in the third line.

Next we deal with $L_3^{12}(x,y)$. It is easy to obtain that
	\begin{equation*}
		\begin{split}
			L_3^{12}(x,y)&=\int_{\Sigma_1\times\Sigma _2^{'}}\frac{1}{\langle x-w\rangle^N\langle y-z\rangle^N|w||z|(|w|+|z|)^3}dwdz\\
			&\lesssim\frac{1}{\langle x\rangle^N}\int_{|z|>\frac{1}{2}}\int_{|w|\leq\frac{1}{2}}\frac{1}{\langle y-z\rangle^N|w||z|^4}dwdz
			\lesssim \frac{1}{\langle x\rangle^N}\int_{|z|>\frac{1}{2}}\frac{1}{\langle y-z\rangle^N|z|^4}dz.
		\end{split}
	\end{equation*}
	Since $|z|>\frac{1}{2}$, one has
	\begin{equation*}
		\begin{split}
			\frac{1}{\langle y-z\rangle^N|z|^4}\lesssim\frac{1}{\langle y-z\rangle^{N-4}(1+|y-z|+|z|)^4}
			\lesssim\frac{1}{\langle y-z\rangle^{N-4}\langle y\rangle^4}.
		\end{split}
	\end{equation*}
Therefore, 
	\begin{equation*}
		\begin{split}
			L_3^{12}(x,y)\lesssim\frac{1}{\langle x\rangle^N\langle y\rangle^4}\int_{|z|>\frac{1}{2}}\frac{1}{\langle y-z\rangle^{N-4}}dz
			\lesssim\frac{1}{\langle x\rangle^N\langle y\rangle^4}\lesssim\frac{1}{\langle x\rangle\langle y\rangle\langle|x|+|y|\rangle^3}.
		\end{split}
	\end{equation*}
	
For $L_3^{21}(x,y)$,  it is enough to observe $L_3^{21}(x,y)=L_3^{12}(y,x)$ by the symmetry. 

Finally, we focus on $L_3^{22}(x,y)$. Noting that $|w|, |z|>\frac{1}{2}$, one has
	\begin{equation*}
		\begin{split}
			L_3^{22}(x,y)&=\int_{\Sigma_2\Sigma_2^{'}}\frac{1}{\langle x-w\rangle^N\langle y-z\rangle^N|w||z|(|w|+|z|)^3}dwdz\\
			&\lesssim\frac{1}{\langle x\rangle\langle y\rangle}
			\int_{|z|>\frac{1}{2}}\int_{|w|>\frac{1}{2}}\frac{1}{\langle x-w\rangle^{N-1}\langle y-z\rangle^{N-1}(|w|+|z|)^3}dwdz\\
			&\lesssim \frac{1}{\langle x\rangle\langle y\rangle}\int_{\mathbb{R}^6}\<(x-w,y-z)\>^{-N+1}|(w,z)|^{-3}dwdz
	\lesssim\frac{1}{\langle x\rangle\langle y\rangle\langle|x|+|y|\rangle^3}.
		\end{split}
	\end{equation*}
Summing up (i)--(iii) above, we complete the proof of  Proposition \ref{L1,L2,L3}.
\end{proof}

\begin{proposition}\label{K}
	Let $$E_0(x,y)=\frac{1}{(2\pi)^3}\int_{\mathbb{R}^6}e^{i\xi x-i\eta y}\frac{\widetilde \chi_{<4a}(\xi)\widetilde \chi_{<a}(\eta)}{(|\xi|+|\eta|)(|\xi|^2+|\eta|^2)}d\xi d\eta.$$  Then
			$\big|E_0(x,y)\big|\lesssim \langle|x|+|y|\>^{-3}.$
	\end{proposition}
\begin{proof}
	Combining with the estimates of $L_1(x,y)$ and $L_3(x,y)$ in Proposition \ref{L1,L2,L3},  we have
	\begin{equation*}
		\begin{split}
			|E_0(x,y)|&\lesssim \int_{\mathbb{R}^6}|L_1(x-w,y-z)||L_3(w,z)|dwdz\\
			&\lesssim\int_{\mathbb{R}^6}
			\frac{dwdz}{\langle|x-w|+|y-z|\rangle^4\langle w\rangle\langle z\rangle\langle|w|+|z|\rangle^3}.
		\end{split}
	\end{equation*}
	Now we decompose $\mathbb{R}^6$ into nine parts $\{\Sigma_i\times\Sigma _j^{'}\}_{i,j=1}^3,$ where
	\begin{equation*}
		\begin{split}
			&\Sigma_1=\{w\in\mathbb{R}^3\ |\ |x-w|\leq|x|/2\},\ \ \ \ \Sigma_2=\{w\in\mathbb{R}^3\ |\ |x|/2<|x-w|\leq2|x|\},\\
			&\Sigma_3=\{w\in\mathbb{R}^3\ |\ |x-w|>2|x|\},\ \ \ \
	\Sigma_1^{'}=\{z\in\mathbb{R}^3\ |\ |y-z|\leq|y/2|\},\\ &\Sigma_2^{'}=\{z\in\mathbb{R}^3\ |\ |y|/2<|y-z|\leq2|y|\},\ \ \ \
	\Sigma_3^{'}=\{z\in\mathbb{R}^3\ |\ |y-z|>2|y|\}.
		\end{split}
	\end{equation*}
Then it follows that
	\begin{equation*}
		\begin{split}
			|E_0(x,y)|\lesssim \sum_{i,j=1}^3\int_{\Sigma_i\times\Sigma _j^{'}}
			\frac{dwdz}{\langle|x-w|+|y-z|\rangle^4\langle w\rangle\langle z\rangle\langle|w|+|z|\rangle^3}=:\sum_{i,j=1}^3K_{ij}(x,y).
		\end{split}
	\end{equation*}
	Hence, it suffices to show that all of $K_{ij}(x,y)$ with $i,j=1,2,3$ are dominated by $\langle|x|+|y|\rangle^{-3}$.\\
(i) For $K_{11}$, it is easy to obtain that
\begin{equation*}
	\begin{split}
		K_{11}(x,y)&\lesssim\frac{1}{\langle x\rangle\langle y\rangle\langle|x|+|y|\rangle^3}
		\int_{\Sigma_1\times\Sigma _1^{'}}\frac{dwdz}{\langle|x-w|+|y-z|\rangle^4}\\
		&\lesssim\frac{1}{\langle x\rangle\langle y\rangle\langle|x|+|y|\rangle^3}\int_0^{\frac{|x|}{2}}
		\int_0^{\frac{|y|}{2}}\frac{r^2\rho^2}{\langle r+\rho\rangle^4}d\rho dr\lesssim \langle|x|+|y|\rangle^{-3},
	\end{split}
\end{equation*}
where we have used the bound $r^2\rho^2\lesssim \<r+\rho\>^4$ for $r,\rho\ge0$. \\
(ii)  For $K_{12}$, we have
\begin{align}\nonumber
		K_{12}(x,y)&\lesssim\frac{1}{\langle x\rangle}\int_{\Sigma_1\times\Sigma _2^{'}}\frac{dwdz}{\langle|x-w|+|y|\rangle^4
			\langle z\rangle\langle|x|+|z|\rangle^3}\\\nonumber
		&\lesssim\frac{1}{\langle x\rangle}\int_0^{3|y|}\int_0^{\frac{|x|}{2}}\frac{r^2\rho^2}{\langle r+|y|\rangle^4
			\langle \rho\rangle\langle|x|+\rho\rangle^3}drd\rho\\
\label{K_12}
		&\lesssim\frac{1}{\langle x\rangle}\Big(\int_0^{|x|/2}\frac{r^2}{(1+ r+|y|)^4} dr\Big) \Big(\int_0^{3|y|}	\frac{\rho} {(1+|x|+\rho)^3}d\rho\Big).
\end{align}
Plugging the explicit formulas
\begin{align}
\label{rho}
\int_0^{\frac{|x|}{2}}\frac{r^2}{(1+ r+ |y| )^4}dr&=\frac{|x|^3}{24(1+|y|)(1+\frac{|x|}{2}+|y|)^3},\\
\label{r}
\int_0^{3|y|}\frac{\rho}{(1+|x|+\rho)^3}d\rho&=\frac{9|y|^2}{2(1+|x|)(1+|x|+3|y|)^2}, 
\end{align}
into \eqref{K_12} yields
$$K_{12}(x,y) \lesssim\frac{|x|^3|y|^2}{\langle x\rangle^2\langle y\rangle(1+\frac{|x|}{2}+|y|)^3(1+|x|+3|y|)^2}\lesssim\frac{1}{\langle|x|+|y|\rangle^3}.$$
(iii) For $K_{13}$, we have
\begin{equation*}
	\begin{split}
		K_{13}(x,y)&\lesssim\frac{1}{\langle x\rangle}\int_{\Sigma_1\times\Sigma _3^{'}}\frac{dwdz}{\langle|x-w|+|z|\rangle^4
			\langle z\rangle\langle|x|+|z|\rangle^3}\\
		&\lesssim\frac{1}{\langle x\rangle}\int_{|y|}^{\infty}\int_0^{\frac{|x|}{2}}\frac{r^2\rho^2}{\langle r+\rho\rangle^4
			\langle \rho\rangle\langle|x|+\rho\rangle^3}drd\rho\\
		&\lesssim\frac{|x|^3}{\langle x\rangle}\int_{|y|}^{\infty}\frac{\rho^2}{\langle \rho\rangle^2(1+|x|+\rho)^3(1+\frac{|x|}{2}+\rho)^3}d\rho\\
		&\lesssim\frac{|x|^2}{\langle|x|+|y|\rangle^3}\int_{|y|}^{\infty}\frac{1}{(1+\frac{|x|}{2}+\rho)^3}d\rho\lesssim\frac{1}{\langle|x|+|y|\rangle^3},
	\end{split}
\end{equation*}
where we have used \eqref{rho} in the third line. \\
(iv) For $K_{21}$, one has  $K_{21}(x,y)=K_{12}(y,x)$ by the symmetry and hence
$$K_{21}(x,y)\lesssim \langle|x|+|y|\rangle^{-3}$$
by the item (ii) above. \\
(v) For $K_{22}(x,y)$, by using the integral \eqref{r}, we  have
\begin{equation*}
	\begin{split}
		K_{22}(x,y)&\lesssim\frac{1}{\langle|x|+|y|\rangle^4}\int_{\Sigma_2\times\Sigma _2^{'}}\frac{dwdz}{\langle w\rangle
			\langle z\rangle\langle|w|+|z|\rangle^3}\\
		&\lesssim\frac{1}{\langle|x|+|y|\rangle^4}\int_0^{3|x|}\int_0^{3|y|}\frac{r^2\rho^2}{\langle r\rangle
			\langle \rho\rangle\langle r+\rho\rangle^3}d\rho dr\\
		&\lesssim\frac{|y|^2}{\langle|x|+|y|\rangle^4}\int_0^{3|x|}\frac{r^2}{\langle r\rangle^2(1+r+3|y|)^2}dr.
	\end{split}
\end{equation*}
Since
\begin{align*}
	\int_0^{3|x|}\frac{1}{(1+r+3|y|)^2}dr=\frac{3|x|}{(1+3|y|)(1+3|x|+3|y|)}, 
	\end{align*}
we have 
\begin{align*}
		K_{22}(x,y)\lesssim\frac{|x||y|^2}{\langle|x|+|y|\rangle^4(1+3|y|)(1+3|x|+3|y|)}\lesssim\frac{1}{\langle|x|+|y|\rangle^3}.
\end{align*}
(vi) For $K_{23}$, we again use the integral  \eqref{r} to obtain
\begin{equation*}
	\begin{split}
		K_{23}(x,y)&\lesssim\int_{\Sigma_2\times\Sigma _3^{'}}\frac{dwdz}{\langle|x|+|z|\rangle^4\langle w\rangle
			\langle z\rangle\langle|w|+|z|\rangle^3}
		\lesssim\int_{|y|}^\infty\int_0^{3|x|}\frac{r\rho}{\langle|x|+\rho\rangle^4\langle r+\rho\rangle^3}drd\rho\\
		&\lesssim|x|^2\int_{|y|}^\infty\frac{1}{\langle|x|+\rho\rangle^4(1+\rho+3|x|)^2}d\rho\lesssim|x|^2\int_{|y|}^\infty\frac{1}{(1+\rho+|x|)^6}d\rho
		\lesssim\frac{1}{\langle|x|+|y|\rangle^3}.
	\end{split}
\end{equation*}
(vii) For $K_{31}$ and $ K_{32}$, we have $K_{31}(x,y)=K_{13}(y,x)$ and $K_{32}(x,y)=K_{23}(y,x)$ by symmetry, which deduces that both
$K_{31}(x,y)$ and $K_{32}(x,y)$ are dominated by $C\langle|x|+|y|\rangle^{-3}$ with some $C>0$. \\
(ix) It remains to deal with $K_{33}$. We can obtain that
\begin{equation*}
	\begin{split}
		K_{33}(x,y)&\lesssim\int_{\Sigma_3\times\Sigma _3^{'}}\frac{dwdz}{\langle|w|+|z|\rangle^4\langle w\rangle
			\langle z\rangle\langle|w|+|z|\rangle^3}
		\lesssim\int_{|y|}^\infty\int_{|x|}^\infty\frac{r^2\rho^2}{\langle r+\rho\rangle^7\langle r\rangle\langle \rho\rangle}drd\rho\\
		& \lesssim\int_{|y|}^\infty\int_{|x|}^\infty\frac{1}{\langle r+\rho\rangle^5}drd\rho\lesssim\int_{|y|}^\infty
		\frac{1}{(1+|x|+\rho)^4}d\rho
		\lesssim\frac{1}{\langle|x|+|y|\rangle^3}.
	\end{split}
\end{equation*}
	Hence, we complete the proof of  Proposition \ref{K}.
\end{proof}

\begin{proposition}\label{E}
	Let $$E_1(x,y)=\frac{1}{(2\pi)^3}\int_{\mathbb{R}^6}e^{i\xi x-i\eta y}\frac{\widetilde \chi_{<4a}(\xi)\widetilde \chi_{<a}(\eta)}{|\eta|(|\xi|^2+|\eta|^2)}d\xi d\eta.$$  Then
			$\big|E_1(x,y)\big|\lesssim\langle x\rangle^{-1}\langle|x|+|y|\rangle^{-2}.$

\end{proposition}
\begin{proof}
	Using the estimates of $L_1(x,y)$ and $L_2(x,y)$ in Proposition \ref{L1,L2,L3}, one has
	\begin{equation*}
		\begin{split}
			\big|E_1(x,y)\big|&\lesssim \int_{\mathbb{R}^3}\big|L_1(x,y-z)\big|\big|L_2(z)\big|dz
			\lesssim\int_{\mathbb{R}^3}
			\frac{dwdz}{\langle|x|+|y-z|\rangle^4\langle z\rangle^2}.
		\end{split}
	\end{equation*}
We decompose $\mathbb{R}^3$ into three  parts $\{\Sigma_i\}_{i=1}^3,$ where
	\begin{equation*}
		\begin{split}
			&\Sigma_1=\{z\in\mathbb{R}^3\ |\ |y-z|\leq |y|/2\},\ \ \ \ \Sigma_2=\{z\in\mathbb{R}^3\ |\ |y|/2<|y-z|\leq2|y|\},\\
			&\Sigma_3=\{z\in\mathbb{R}^3\ |\ |y-z|>2|y|\}.
		\end{split}
	\end{equation*}
Then we have
	\begin{equation*}
		\begin{split}
			\big|E_1(x,y)\big|\lesssim \sum_{i=1}^3\int_{\Sigma_i}
			\frac{1}{\langle|x|+|y-z|\rangle^4\langle z\rangle^2}dz=:\sum_{i=1}^3E_i(x,y).
		\end{split}
	\end{equation*}
	(i) For $E_1$, by (\ref{rho})  one has
\begin{equation*}
	\begin{split}
		E_1(x,y)&\lesssim\frac{1}{\langle y\rangle^2}
		\int_{\Sigma_1}\frac{1}{\langle|x|+|y-z|\rangle^4}dz
		\lesssim\frac{1}{\langle y\rangle^2}\int_0^{\frac{|y|}{2}}\frac{r^2}{\langle|x|+ r\rangle^4} dr\\
		&\lesssim\frac{1}{\langle y\rangle^2}\frac{|y|^3}{(1+|x|)(1+|x|+\frac{|y|}{2})^3}\lesssim\frac{1}{\langle x\rangle\langle|x|+|y|\rangle^2}.
	\end{split}
\end{equation*}
(ii) For $E_2$, it is easy to get that
\begin{equation*}
	\begin{split}
		E_2(x,y)\lesssim\frac{1}{\langle|x|+|y|\rangle^4}
		\int_{\Sigma_2}\frac{1}{\langle z\rangle^2}dz
		\lesssim\frac{1}{\langle|x|+|y|\rangle^4}\int_0^{3|y|}\frac{r^2}{\langle r\rangle^2}dr\lesssim\frac{1}{\langle|x|+|y|\rangle^3}.
	\end{split}
\end{equation*}

(iii) Finally, for $E_3$,  one has
\begin{equation*}
	\begin{split}
		E_3(x,y)\lesssim\int_{\Sigma_3}\frac{1}{\langle|x|+|z|\rangle^4\langle z\rangle^2}dz
		\lesssim\int_{|y|}^\infty\frac{r^2}{\langle|x|+r\rangle^4\langle r\rangle^2}dr\lesssim\frac{1}{\langle|x|+|y|\rangle^3}.
	\end{split}
\end{equation*}
Thus we complete  the proof of  Proposition \ref{E}.
\end{proof}

\begin{proposition}\label{prop5.2} Let $K_1$ (resp. $K_2$) be the integral operator with kernel $\<|x|+|y|\>^{-3}$ (resp. $ \langle x\rangle^{-1}\langle|x|+|y|\rangle^{-2}$). Then $K_1\in \mathbb B(L^p(\R^3))$ for all $1<p<\infty$ and $K_2\in \mathbb B(L^p(\R^3))$ for all $1<p<3$. 
\end{proposition}
\begin{proof}Let $f\in \mathcal S(\R^3)$. Since
	\begin{equation*}
		\begin{split}
			K_1f(x)=\int_{\mathbb{R}^3}\frac{1}{(1+|x|+|y|)^3}f(y)dy=\int_{\mathbb{R}^3}\frac{1}{(1+|z|)^3}f\big((1+|x|)z\big)dz, 
		\end{split}
	\end{equation*}
by the polar coordinate $z=r\omega$ with $r>0$ and  $\omega\in S^2$, one has
	\begin{equation}\label{K1f}
		\begin{split}
			K_1f(x)=\int_0^\infty\frac{r^2}{(1+r)^3}\Big(\int_{S^2}f\big((1+|x|)r\omega\big)d\omega\Big)dr=\int_0^\infty\frac{r^2}{(1+r)^3}M_f\big((1+|x|)r\big)dr,
		\end{split}
	\end{equation}
	where $M_f(t)=\int_{S^2}f(t\omega)d\omega$. By the polar coordinate $x=\rho\sigma$, $\rho>0$, $\sigma\in S^2$ and H\"{o}lder's inequality,
	\begin{equation*}
		\begin{split}
			\big\|M_f\big((1+|\cdot|)r\big)\big\|_{L^p} ^p
			%&=\int_0^\infty\rho^2\int_{S^2}\big|M_f\big((1+\rho)r\big)\big|^pd\sigma d\rho\\
			&\lesssim\int_0^\infty \rho^2\big|M_f\big((1+\rho)r\big)\big|^pd\rho\lesssim\int_0^\infty \rho^2\big|M_f(\rho r)\big|^pd\rho\\
			&\lesssim\int_0^\infty\rho^2\int_{S^2}\big|f(\rho r\omega)\big|^pd\omega d\rho\lesssim \frac{1}{r^3}\|f\|_{L^p}^p.
		\end{split}
	\end{equation*}
	Combining with \eqref{K1f} and Minkowski's inequality, one has for all $ 1<p<\infty$
	\begin{equation*}
		\begin{split}
			\|K_1f\|_{L^p}\lesssim\Big(\int_0^\infty\frac{r^{2-\frac{3}{p}}}{(1+r)^3}dr\Big)\|f\|_{L^p}\lesssim\|f\|_{L^p}.
		\end{split}
	\end{equation*}
Similarly, it is easy to check that
$$
			K_2f(x)=\int_0^\infty\frac{r^2}{(1+r)^2}M_f\big((1+|x|)r\big)dr. 
$$
Thus, we have for $1<p<3,$
	\begin{equation*}
		\begin{split}
			\|K_2f\|_{L^p}\lesssim\int_0^\infty\frac{r^{2-\frac{3}{p}}}{(1+r)^2}dr\|f\|_{L^p}\lesssim\|f\|_{L^p}.
		\end{split}
	\end{equation*}
	Thus we finish the proof of the proposition.
\end{proof}
\begin{proof}[Proof of Lemma \ref{lem2.2}]  In fact, in view of  Proposition \ref{K},\ \ref{E} and \ref{prop5.2}, we can immediately conclude that $E_0\in \mathbb{B}(L^p)(\R^3)$ for all $1<p<\infty$, and $E_1\in \mathbb{B}(L^p)(\R^3)$ for all $1<p<3$.
\end{proof}

%subsection
\subsection{Two technical lemmas}
\begin{lemma}
\label{lemma_technical}
Let $\widetilde \chi_{\ge a}$ be defined in Section \ref{subsection_notation}, $m_a(\xi)=|\xi|^{-4}\widetilde\chi_{\ge a}(\xi)$, $v=|V|^{1/2}$ and $\{f_j\}_{j=1}^{n_2}$ be an orthonormal basis of $S_2L^2$.  Then $\{m_a(D)vf_j\}_{j=1}^{n_2}$  are   mutually linearly independent in the distributional sense  for any $a>0$ small enough. 
\end{lemma}

%proof
\begin{proof}
We show the lemma by contradiction. Assume that $\{m_a(D)vf_j\}_{j=1}^{n_2}$ are  linearly dependent  for sufficiently small $a>0$. Then their Fourier transforms $\{|\xi|^{-4}\widetilde\chi_{\geq4a}(\xi)\widehat{vf_j}(\xi)\}_{j=1}^{n_2}$ are also linearly dependent  in the distributional sense.

Choose an arbitrarily decreasing sequence $\{a_m\}$ with $a_m\to 0$ as $m\to \infty$.
For each $m\in \N$, there exists a vector $\beta_m=(b_{m1},\cdots,b_{mn})$ with $|\beta_m|=1$ and a null measure set $\Theta_m$ such that
\begin{equation*}
	\begin{split}
		\sum_{j=1}^{n_2}b_{mj}|\xi|^{-4}\widetilde \chi_{\geq4a_m}(\xi)\widehat{vf_j}(\xi)=0,\quad \xi\in \Theta_m^c.
	\end{split}
\end{equation*}
In particular $\sum_{j=1}^{n_2}b_{mj}\widehat{vf_j}(\xi)=0$ for $\xi\in\big(\bigcup_{k=1}^\infty \Theta_k\big)^c$ satisfying $|\xi|\geq4a_m.$ Let
\begin{equation*}
	B_m:=\Big\{ \beta_m=(b_{m1}, \cdots, b_{mn})\ \Big|\  |\beta_m|=1, \sum_{j=1}^n b_{mj}\widehat{vf_j}(\xi)=0,\ \xi \in \Big(\bigcup_{k=1}^\infty \Theta_k \Big)^c,\ |\xi|\geq4a_m  \Big\}.
\end{equation*}
Note that $B_m\supset B_{m+1}$ and  each $B_m$ is a non-empty compact set. Hence
\begin{equation*}
	B=\bigcap_{m=1}^\infty B_m\neq\emptyset.
\end{equation*}
Moreover, for any $\beta=(b_1, \cdots, b_n)\in B$, one has
\begin{equation*}
	\sum_{j=1}^nb_j\widehat{vf_j}(\xi)=0,\quad \xi\in\big(\bigcup_{k=1}^\infty \Theta_k\big)^c,
\end{equation*}
which implies that $\{vf_j\}_{j=1}^n$ are linearly dependent and, for all $\beta=(b_1, \cdots, b_n)\in B$
$$\sum_{j=1}^nb_jvf_j=0.$$
Since $f_j\in S_2L^2\subset S_1L^2\subset QL^2$, we observe by Lemma \ref{lemma_resonance} (ii) and (iii) that
$$
0=U QTf_j=U(U+vG_0v)f_j-UPTf_j=(I+UvG_0v)f_j. 
$$
Thus, for all $\beta=(b_1, \cdots, b_n)\in B$, 
$$\sum_{j=1}^nb_jf_j=-UvG_0\Big(\sum_{j=1}^nb_jvf_j\Big)=0.$$
Hence $\{f_j\}_{j=1}^{n_2}$ are  linearly dependent, which is a contradiction.
\end{proof}

\begin{lemma}\label{equiv0}
Let $c_{kls},b_{kl}\in \C$ and 
\begin{equation*}%\label{homo}
\phi(\eta)=\sum_{k,l,s=1}^3c_{kls}\frac{\eta_{k}\eta_{l}\eta_{s}}{|\eta|^3}
		-\sum_{k,l=1}^3b_{kl}\frac{\eta_{k}\eta_{l}}{|\eta|^2},\quad \eta\in \R^3\setminus\{0\}. 
\end{equation*}
If $\phi$ vanishes identically and the values of coefficients $c_{kls}, b_{kl}$ do not change by arranging the order of  the variables $k,l,s$ for all $k,l,s=1,2,3$.  Then $c_{kls}=b_{kl}=0$ for all $k,l,s=1,2,3$. 
\end{lemma}

\begin{proof}By the homogeneity of $\phi$, we may assume $|\eta|=1$ without loss of generality. 
By the assumption, $f(\eta)=\sum_{k,l=1}^3b_{kl}\eta_{k}\eta_{l}$ and $g(\eta)=\sum_{k,l,s=1}^3c_{kls}\eta_{k}\eta_{l}\eta_{s}$ are of the following forms:
\begin{align*}
f(\eta)&=b_{11}\eta_{1}^2+b_{22}\eta_{2}^2+b_{33}\eta_{3}^2+2b_{12}\eta_{1}\eta_{2}+2b_{13}\eta_{1}\eta_{3}+2b_{23}\eta_{2}\eta_{3}\\
		g(\eta)&=  c_{111}\eta_{1}^3+c_{222}\eta_{2}^3+c_{333}\eta_{3}^3+3c_{112}\eta_{1}^2\eta_{2}+3c_{113}\eta_{1}^2\eta_{3}+3c_{122}\eta_{1}\eta_{2}^2\nonumber\\		&+3c_{133}\eta_{1}\eta_{3}^2+3c_{223}\eta_{2}^2\eta_{3}+3c_{233}\eta_{2}\eta_{3}^2+6c_{123}\eta_{1}\eta_{2}\eta_{3}. 
\end{align*}
Since each polynomials of entries of $f,g$ are linearly independent, if $f-g\equiv0$ then $c_{kls}\equiv b_{kl}\equiv0$. 
\end{proof}

\appendix

\section{Asymptotic expansions of $M^{-1}(\lambda)$  with threshold singularities}
\label{Appendix A}
In this Appendix, we show Theorem \ref{thm-main-inver-M}, the low energy expansions of the inverse of
\begin{align}
\label{M^pm}
M^\pm(\lambda)=U+ vR^\pm_0(\lambda^4)v. 
\end{align}

%%%%%%%%%%%%%%%%%	
\subsection{The  asymptotic expansion of $M^\pm(\lambda)$}
\label{subsection_M^pm}
We first derive the asymptotic expansions of  $M^\pm(\lambda)$.  The formula \eqref{free_resolvent} gives the following  asymptotic expansion as $\lambda\to 0$:
\begin{equation}\label{resolvent-expansion-R0lambda4}
	\begin{split}
		R^\pm_0(\lambda^4)(x,y)=& \frac{a^\pm}{\lambda}+G_0(x,y) + \lambda a_1^\pm G_1(x,y) +\lambda^3 a_3^\pm  G_3(x,y)\\
		&+\lambda^4G_4(x,y)+ \sum_{k=5}^N  \lambda^k a_k^\pm G_k(x,y) +O\big( \lambda^{N+1}|x-y|^{N+2}\big),
	\end{split}
\end{equation}
where 
\begin{equation}\label{def-Gk}
	\begin{split}
		&G_0(x,y)=-\frac{|x-y|}{8\pi},\quad G_1(x,y)=|x-y|^2,\quad G_3(x,y)=|x-y|^4,\\
		& G_4(x,y)=-\frac{|x-y|^5}{4\pi\cdot 6!},\quad
		G_k(x,y)= |x-y|^{k+1}, \quad k\geq5,
	\end{split}
\end{equation}
and
$$\displaystyle a^\pm= \frac{1\pm i}{8\pi}, \ a_1^\pm=\frac{1\mp i}{8\pi \cdot3!},
\ a_3^\pm=\frac{1\pm i}{8\pi \cdot 5!},
\ a_k^\pm= \frac{ (-1)^{k+1}+ (\pm i)^{k+2}}{8\pi\cdot (k+2)!},\quad k\ge5.$$ 
Let $G_k$ be the integral operator with kernel $G_k(x,y)$ (note that  $G_0=(\Delta^2)^{-1}$), $T= U+vG_0v$ and $P= \|V\|^{-1}_{L^1(\mathbb{R}^3)} \langle\cdot, v \rangle v$.  
Plugging \eqref{resolvent-expansion-R0lambda4} into \eqref{M^pm} implies the following: 
\begin{lemma}\label{lem-M}
	Let $ \tilde{a}^\pm= a^\pm \|V \|_{L^1}$. Assume that $|V(x)|\lesssim (1+|x|)^{-\mu}$ with some $\mu > 0$. Then the following expansions of $M^\pm(\lambda)$ hold in $\mathbb{B}(L^2(\mathbb{R}^3))$ as $\lambda\to 0$:
	\begin{itemize}
		\item[(i)] If $\mu > 9$, then
		\begin{equation}\label{Mpm-1}
			\begin{split}
				M^\pm(\lambda)= \frac{\tilde{a}^\pm}{\lambda}P +T+a_1^\pm \lambda vG_1v+\Gamma_2(\lambda);
			\end{split}
		\end{equation}
		
		\item[(ii)]If $\mu > 13$, then
		\begin{equation}\label{Mpm-2}
			\begin{split}
				M^\pm(\lambda)= \frac{\tilde{a}^\pm}{\lambda}P +T+a_1^\pm \lambda vG_1v+a_3^\pm \lambda^3 vG_3v+\Gamma_4(\lambda);
			\end{split}
		\end{equation}
		
		\item[(iii)] If $\mu > 21$, then
		\begin{equation}\label{Mpm-3}
			\begin{split}
				M^\pm(\lambda)= &\frac{\tilde{a}^\pm}{\lambda}P +T+a_1^\pm \lambda vG_1v+a_3^\pm \lambda^3 vG_3v \\
				&+\lambda^4vG_4v + a_5^\pm\lambda^5 vG_5v + a_6^\pm \lambda^6vG_6v +a_7^\pm \lambda^7vG_7v +\Gamma_8(\lambda);
			\end{split}
		\end{equation}
		
		\item[(iv)] If $ \mu > 25$, then
		\begin{equation}\label{Mpm-4}
			\begin{split}
				M^\pm(\lambda)= \frac{\tilde{a}^\pm}{\lambda}P &+T+a_1^\pm \lambda vG_1v+a_3^\pm \lambda^3 vG_3v
				+\lambda^4vG_4v +\sum_{k=5}^9a_k^\pm \lambda^kvG_kv + \Gamma_{10}(\lambda).
			\end{split}
		\end{equation}
	\end{itemize}
Here $\Gamma_{k}(\lambda)$ are $\lambda$-dependent Hilbert-Schmidt operators satisfying
	\begin{equation}\label{Gamma_k}
		\big\| \Gamma_k(\lambda)\big\|_{L^2\rightarrow L^2}+\lambda\big\|\partial_\lambda\Gamma_k(\lambda)\big\|_{L^2\rightarrow L^2}
		+\lambda^2\big\|\partial^2_\lambda\Gamma_k(\lambda)\big\|_{L^2\rightarrow L^2}
		\lesssim \lambda^k.
	\end{equation}
\end{lemma}
For the sake of simplicity, any Hilbert-Schmidt operator-valued functions  satisfying the property \eqref{Gamma_k}  will be indiscriminately denoted by $\Gamma_k(\lambda)$. 

%%%%%%%%%%%%
\subsection {The asymptotic expansions of $M^\pm(\lambda)^{-1}$}
\label{expansion 6.2}
We next deal with $M^\pm(\lambda)^{-1}$. The following lemma will be used frequently in the proof of Theorem \ref{thm-main-inver-M}. 
\begin{lemma}[{\cite[Lemma 2.1]{JN}}]
\label{lemma-JN}
Let $A$ be a closed operator and $S$ be a projection. Suppose $A+S$ has a bounded inverse. Then $A$ has
	a bounded inverse if and only if
	\begin{equation*}		
		a:= S-S(A+S)^{-1}S
	\end{equation*}
	has a bounded inverse in $SH$, and in this case
	\begin{equation*}		
		A^{-1}= (A+S)^{-1} + (A+S)^{-1}S a^{-1} S(A+S)^{-1}.
	\end{equation*}	
\end{lemma}

%%%%%%%%%%%%%%%%%%%%%%%%%%%%%%%%%%%%%%%%%%

Now we are  in a position to prove Theorem \ref{thm-main-inver-M}.

\begin{proof}[Proof of Theorem \ref{thm-main-inver-M}]
The proof basically follows a similar line to that of \cite[Theorem 4.4]{Erdogan-Green-Toprak}. For simplicity, we drop the letter ``$+$", as well as the variable $\lambda$, denoting for instance 
$$M=M^+(\lambda),\quad \tilde a=\tilde a^+, \quad \Gamma_k=\Gamma_k(\lambda)$$
 and so on. Let $\widetilde M=\lambda \tilde a^{-1}M$ and recall that $Q=I-P$. 
%$ \displaystyle M^+(\lambda)= \frac{\tilde{a}^+}{\lambda} \widetilde{M}^+(\lambda)$. % and calculate the asymptotic expansion of  $\big(\widetilde{M}^\pm(\lambda)\big)^{-1}$ as $\lambda \rightarrow 0$. 
By virtue of Lemma \ref{lem-M}, the Neumann series expansion of $(\widetilde{M}+Q)^{-1}$ for sufficiently small $\lambda>0$ is of the form
%	\begin{equation*}\begin{split}\widetilde{M}+Q			= I + \frac{\lambda}{\tilde{a}}T+ \frac{a_1}{\tilde{a}}\lambda^2vG_1v			+ \frac{a_3}{\tilde{a}}\lambda^4vG_3v			+ \frac{1}{\tilde{a}}\lambda^5vG_4v			+\sum_{k=5}^9 \frac{a_k}{\tilde{a}}\lambda^{k+1}vG_kv +\Gamma_{11}.		\end{split}	\end{equation*}Then the Neumann series expansion of $(\widetilde{M}+Q)^{-1}$ is of the form
\begin{align}
\label{M+Q}
(\widetilde{M}+Q)^{-1}= I-\sum_{k=1}^{N-1} \lambda^kB_k+ \Gamma_{N},
\end{align}
where $N$ depends on the decaying condition $|V(x)|\lesssim \<x\>^{-\mu}$ as 
$$
N\le \begin{cases}3&\text{if $\mu>9$,}\\5&\text{if $\mu>13$,}\\
9&\text{if $\mu>21$,}\\
11& \text{if $\mu>25$.}
\end{cases}$$
$B_k$ are independent of $\lambda$ and $B_k\in \AB(L^2)$ since all of the entries of the expansions of $M$ in Lemma \ref{lem-M} are absolutely bounded. Moreover, $B_k$ can be computed explicitly in terms of combinations of $T,v,G_k$. In particular, $B_1,...,B_5$ are of the forms
\begin{align*}
B_1=& \frac{1}{\tilde{a}} T,\quad
			B_2 = \frac{a_1}{\tilde{a}} vG_1v -\frac{1}{\tilde{a}^2}T^2,\quad
			B_3= -\frac{a_1}{\tilde{a}^2} (TvG_1v +vG_1vT )
			+ \frac{1}{\tilde{a}^3}T^3,\\
			B_4=& \frac{a_3}{\tilde{a}}vG_3v
			- \frac{(a_1)^2}{\tilde{a}^2}(vG_1v)^2
			+ \frac{a_1}{\tilde{a}^3} ( T^2vG_1v+TvG_1vT+vG_1vT^2)
			-\frac{1}{\tilde{a}^4}T^4,\\
			B_5=& \frac{1}{\tilde{a}}vG_4v -\frac{a_3}{\tilde{a}^2}(TvG_3v+vG_3vT )
			+ \frac{a_1^2}{\tilde{a}^3}\left(T(vG_1v)^2+(vG_1v)^2T+vG_1vTvG_1v \right)\\
			&\ \ \ \ \  -\frac{a_1}{\tilde{a}^4}\big( T^3vG_1v +vG_1vT^3 +T^2vG_1vT +TvG_1vT^2 \big)
			+\frac{1}{\tilde{a}^5}T^5.
\end{align*}					%%\end{split}
%%	\end{equation}
%%	\begin{equation}\label{B+--}
%		\begin{split}
Here we list the orthogonality relations of the operators defined in Definition \ref{definition_resonance}: 
	\begin{equation}
		\begin{split}
\label{orthog-relation-1}
Q\D_0 =& \D_0Q=\D_0,\quad \D_i\D_j = \D_j\D_i =\D_i,\quad S_i\D_j = \D_jS_i =S_i,\quad\text{for}\quad i >  j,\\
			S_i\D_j=& \D_jS_i =\D_j\quad\text{for}\quad i \leq j;
			\end{split}
	\end{equation}
	%\vskip0.005cm
	\begin{equation}\label{orthog-relation-2}
		\begin{split}
			& S_2TP=PTS_2=S_2T=TS_2=QvG_1vS_2=S_2vG_1vQ=0,\\
			&QTS_1=S_1TQ= 0,\,vG_1vS_3=S_3vG_1v=0, \, S_2vG_3vS_3=S_3vG_3vS_2=0;
		\end{split}
	\end{equation}
	%\vskip0.005cm
	\begin{equation}\label{orthog-relation-3}
		\begin{split}
			&B_1S_2=S_2B_1=0,\quad QB_2 S_2=S_2B_2 Q=B_2S_3=S_3B_2=0,\\
			&S_2B_3S_2=B_3 S_3=S_3B_3=0, \quad S_2B_4S_3=S_3B_4S_2=0.
		\end{split}
	\end{equation}
	
	 We give some details about  \eqref{orthog-relation-1}-\eqref{orthog-relation-3}. Note that $S_j$ is an orthogonal projection  with $S_{j+1}L^2 \subset S_jL^2 $ for all $j=0,1,2,3,$ here let $Q=S_0,$ then $S_iS_j = S_j $ for $i \leq j$.
	 Furthermore,  $\D_j=(T_j+S_{j+1})^ {-1}$ defined in Definition \ref{definition_resonance} is invertible on $S_jL^2$ and 
	 $\hbox{ker} T_j=S_{j+1}L^2.$ 
	 Then $\D_jS_j=S_j\D_j=\D_j$ and 
	 $$ \D_jS_i =\D_jS_{j+1}S_i =\D_j(T_j+S_{j+1})S_{j+1}S_i =S_jS_{j+1}S_i=S_i \ \text{for}\  i >  j.$$
	 Since $S_i, \D_j$ are self-adjoint, then $ S_i\D_j=S_i $ for $  i >  j.$
	Moreover,  $$\D_i\D_j=\D_iS_i\D_j=\D_iS_i=\D_i\ \text{for}\ i >  j.$$ We can also easily obtain that $S_i\D_j=S_iS_j\D_j=S_j\D_j=\D_j$ for $i \leq j.$ And by the self-adjointness of 
	  $S_i, \D_j$, one has $\D_jS_i=\D_j$ for $i \leq j.$ Thus we obtain \eqref{orthog-relation-1}.
	  
	  As for \eqref{orthog-relation-2}-\eqref{orthog-relation-3}, it follows from the orthogonality of $S_j$ (see Lemma \ref{lemma_resonance} ).  We take to show
	  $QvG_1vS_2=0$ as an  example.  Since $G_1(x,y)=|x-y|^2$ 
	  (see \eqref{def-Gk}), $Qv=0$ and $S_1v=S_1x_jv=0, j=1,2,3,$ (see Lemma \ref{lemma_resonance}), then taken $f\in L^2,$
	  \begin{align*}
	  QvG_1vS_2f(x)&=Qv(x)\int_{\R^3}\big(|x|^2-2x\cdot y+|y|^2\big)v(y)S_2f(y)dy\\
	  	&=\<v, S_2f\>Q|x|^2v-2\sum_{j=1}^{3}\<y_jv, S_2f\>Qx_jv
	  	+\<|y|^2v, S_2f\>Qv=0.
	  	\end{align*}
	  Similarly, by the orthogonality of $S_j$ (see Lemma \ref{lemma_resonance} ), we can check the remaining orthogonality relations \eqref{orthog-relation-2}-\eqref{orthog-relation-3}.

%	 Now we turn to the proof of Theorem \ref{thm-main-inver-M} case by case.
	%For the convenience,
	%we shall use the notation $A^k_{i,j}$ denote $\lambda$-independent bounded operator on $L^2(\mathbb{R}^3)$,  which these operators may vary from line to line.

By Lemma \ref{lemma-JN} with $A=\widetilde{M}$, $S=Q$, and \eqref{M+Q}, $\widetilde{M}$ is invertible on $L^2(\mathbb{R}^3)$ if and only if
	\begin{equation}\label{id-M1M1tabu}
		\begin{split}
			M_1
			&:=Q-Q(\widetilde{M}+Q)^{-1}Q=\frac{\lambda}{\tilde{a}}T_0+\sum_{k=2}^{N-1}\lambda^kQB_kQ +\Gamma_{N}
		\end{split}
\end{equation}
is invertible on $QL^2(\mathbb{R}^3)$ in which case
\begin{align}
M^{-1}
&=\frac{\lambda}{\tilde{a}}\big( \widetilde{M}^+(\lambda)\big)^{-1}
=\frac{\lambda}{\tilde{a}}\left\{(\widetilde{M} +Q)^{-1}+ ( \widetilde{M}+Q)^{-1}Q M_1^{-1}Q(\widetilde{M}+Q)^{-1}\right\}\nonumber\\
\label{M^{-1}}
&=\frac{\lambda}{\tilde{a}}\Big\{I-\sum_{k=1}^{N-1} \lambda^kB_k+ \Gamma_{N}+\Big(I-\sum_{k=1}^{N-1} \lambda^kB_k+ \Gamma_{N}\Big)QM_1^{-1}Q\Big(I-\sum_{k=1}^{N-1} \lambda^kB_k+ \Gamma_{N}\Big)\Big\}
\end{align}
where $T_0=QTQ$. It thus is enough to derive the asymptotic expansion of  $M_1^{-1}$. 
\bigskip\\
%%%%%%%%%%%%%%%%%%%%%%%%%%%%%%%%%%%%%%%%%%%%%%%%
%%%%%%%%%%%%%%%%%%%%%%%%%%%%%%%%%%%%%%%%%%%%%%%%
\underline{{\bf(i) The regular case}}. In this case, $T_0$ is invertible on $QL^2$ and 
	$\D_0=T_0^{-1}$ exists on $QL^2$ (see Definition \ref{definition_resonance}). Moreover, it is known by \cite[Lemma 4.3]{Erdogan-Green-Toprak} and \eqref{orthog-relation-1} that $\D_0=Q\D_0Q\in \AB(L^2)$. %Using  one has	$$M_1	=\frac{\lambda}{\tilde{a}}\left(QTQ+\tilde{a}\lambda QB_2Q+\Gamma_2\right). $$
%Let $\widetilde{M}_1^+(\lambda)=\frac{\tilde a^+}{\lambda}M_1^+(\lambda)$. 

It follows from \eqref{id-M1M1tabu} with $N=3$,  the Neumann series expansion that 
\begin{align*}
M_1^{-1}=\left\{\frac{\lambda}{\tilde a}T_0\Big(I+\lambda \tilde a \D_0QB_2Q+\Gamma_2\Big)\right\}^{-1}
=\frac{\tilde{a}}{\lambda}\D_0-\tilde{a}^2 \D_0B_2\D_0+\Gamma_1
\end{align*}
for sufficiently small $\lambda>0$, 
%	$$\big(\widetilde{M}_1^+(\lambda)\big)^{-1} =D_0-\tilde{a}^+\lambda D_0B^+_2D_0+\Gamma_2(\lambda).$$
%	Since $\big(M_1^+(\lambda)\big)^{-1} = \frac{\tilde{a}^+}{\lambda}\big(\widetilde{M}_1^+(\lambda)\big)^{-1}$, then
%	$$\big(M_1^+(\lambda)\big)^{-1} =\frac{\tilde{a}^+}{\lambda}D_0-(\tilde{a}^+)^2 D_0B^+_2D_0+\Gamma_1(\lambda).$$
%Since Lemma \ref{lemma-JN} with $A=\widetilde{M}^+(\lambda)$ and $S=Q$ yields\begin{align}\label{M^{-1}}\big(M^+(\lambda)\big)^{-1}
%&=\frac{\lambda}{\tilde{a}^+}\big( \widetilde{M}^+(\lambda)\big)^{-1}\\=\frac{\lambda}{\tilde{a}^+}\left\{\big( \widetilde{M}^+(\lambda) +Q\big)^{-1}+ \big( \widetilde{M}^+(\lambda)+Q\big)^{-1}Q\big(M_1^+(\lambda)\big)^{-1}Q\big( \widetilde{M}^+(\lambda) +Q\big)^{-1}\right\},\end{align}
which, together with $Q\D_0=\D_0Q=\D_0,$ \eqref{id-M1M1tabu} and \eqref{M^{-1}}, shows
\begin{align*}
M^{-1}
%&=\frac{\lambda}{\tilde{a}^+}\big( \widetilde{M}^+(\lambda)\big)^{-1}\\&=\frac{\lambda}{\tilde{a}^+}\left\{\big( \widetilde{M}^+(\lambda) +Q\big)^{-1}			+ \big( \widetilde{M}^+(\lambda)+Q\big)^{-1}Q\big(M_1^+(\lambda)\big)^{-1}Q			\big( \widetilde{M}^+(\lambda) +Q\big)^{-1}\right\}\\
&=\frac{\lambda}{\tilde{a}}\left\{ I+\Gamma_1+(I+\lambda B_1+\Gamma_2)QM_1^{-1}Q(I+\lambda B_1+\Gamma_2)\right\}\\
&=\frac{\lambda}{\tilde{a}}\left(\frac{\tilde{a}}{\lambda}\D_0+ I-\D_0T-T\D_0-\tilde{a}a_1\D_0vG_1v\D_0+\D_0T^2\D_0+\Gamma_1\right)\\
			&=QA^0_{0,1}Q +\lambda A_{1,1}^0+\Gamma_2,
\end{align*}
%	Furthermore, since $ \displaystyle \big( M^+(\lambda)\big)^{-1}= \frac{\lambda}{\tilde{a}^+}\big( \widetilde{M}^+(\lambda)\big)^{-1} $, we have$$\big(M^+(\lambda)\big)^{-1}= QA^0_{0,1}Q +\lambda A_{1,1}^0+\Gamma_2(\lambda)$$
where  $A^0_{0,1},A^0_{1,1}\in \AB(L^2)$ are given by
	\begin{equation*}
		\begin{split}
			A^0_{0,1}=\D_0,\quad A_{1,1}^0=\frac{1}{\tilde a}-\frac{1}{\tilde{a}}\D_0T-\frac{1}{\tilde{a}}T\D_0-a_1\D_0vG_1v\D_0
			+\frac{1}{\tilde{a}}\D_0T^2\D_0.
		\end{split}
	\end{equation*}
	%%%%%%%%%%%%%%%%%%%%%%%%%%%%%%%%%%%%%%%%%%%%%%%%%%%%%%%%%%%%%%%%%%%%%%%%%%%%%
\noindent
\underline{{\bf (ii) The first kind resonance case}}. In this case, $T_0$ is not invertible on $QL^2$, while $T_0+S_1$ is invertible on $QL^2$ with  $S_1$ being the orthogonal (Riesz) projection onto $\Ker (T_0)$.

 Set $\D_0:= (T_0+S_1)^{-1}$ and $\widetilde{M}_1:=\lambda^{-1}\tilde a M_1$. Then we use \eqref{id-M1M1tabu} with $N=5$, the fact $Q\D_0=\D_0Q=\D_0$ and the Neumann series expansion to derive that
\begin{align}
\nonumber
(\widetilde{M}_1+S_1)^{-1}
&=\left\{(T_0+S_1)\Big(I+\sum_{k=2}^4\lambda^{k-1}\tilde a \D_0QB_kQ+\Gamma_4\Big)\right\}^{-1}\\
%&=\left\{I-\lambda \tilde a D_0QB_2Q-\lambda^2\tilde a(D_0QB_2Q)^2-\lambda^3\tilde a(D_0QB_2Q)^3+\Gamma_4\right\}D_0\\
\label{id-M1tuba+S1}
&=\D_0- \lambda B_1^0 -\lambda^2B_2^0-\lambda^3B_3^0+ \Gamma_4,
\end{align}
as an operator on $QL^2$, where  $B_k^0\in \AB(L^2)$ are $\lambda$-independent, and
	\begin{equation}\label{Bjo}
		\begin{split}
			B_1^0=& \tilde{a}\D_0B_2\D_0, \quad
		 	B_2^0= \tilde{a}\D_0B_3\D_0- \tilde{a}^2\D_0(B_2\D_0)^2,\\
			B_3^0=& \tilde{a}\D_0B_4\D_0 -\tilde{a}^2( \D_0B_2\D_0B_3\D_0 + \D_0B_3\D_0B_2\D_0)
			+\tilde{a}^3\D_0(B_2\D_0)^3.
		\end{split}
	\end{equation}
By Lemma \ref{lemma-JN}, $ \widetilde{M}^+_1$ has a bounded inverse on $QL^2$
	if and only if
	\begin{align}\label{M2,1st}
	 M_2:= S_1- S_1(\widetilde{M}_1+S_1)^{-1}S_1=\lambda S_1B_1^0S_1+\lambda^2S_1B_2^0S_1+\lambda^3S_1B_3^0S_1+\Gamma_4
	 \end{align}
	 has a bounded inverse on $S_1L^2$, where we used \eqref{id-M1tuba+S1}  in the second equality. By using \eqref{Bjo}, the definition of $B_2$, \eqref{orthog-relation-1} and \eqref{orthog-relation-2}, and the fact $T_0=0$ on $S_1L^2$, we compute
\begin{align*}
S_1B_1^0S_1
&=\tilde a S_1\D_0\left(\frac{a_1}{\tilde{a}} vG_1v -\frac{1}{\tilde{a}^2}T^2\right)\D_0S_1=a_1S_1vG_1vS_1-\frac{1}{\tilde a}S_1\D_0T(P+Q)T\D_0S_1\\
&=-\frac{1}{\tilde a}\left(-\tilde a a_1S_1vG_1vS_1+S_1TPTS_1\right)-\frac{1}{\tilde a}S_1\D_0TQT\D_0S_1
=-\frac{1}{\tilde a}T_1,
\end{align*}
where $T_1$ was given in Definition \ref{definition_resonance} and $S_1\D_0TQT\D_0S_1$ vanishes identically since
$$
S_1\D_0TQT\D_0S_1=S_1\D_0(QTQ)^2\D_0S_1=S_1T_0^2S_1=0.
$$ 
Since $T_1$ is invertible as an operator on $S_1L^2$ in the first kind resonance case, so is $M_2$. Furthermore, by \eqref{M2,1st} and the Neumann series expansion, we obtain
\begin{align*}
M_2^{-1}=\left\{ -\frac{\lambda}{\tilde a}T_1\left(I-\lambda \tilde a\D_1S_1B_2^0S_1-\lambda^2\tilde a\D_1S_1B_3^0S_1+\Gamma_3\right)\right\}^{-1}=
-\frac{\tilde{a}}{\lambda}\D_1+\tilde{a} B_1^1+\tilde{a}\lambda B_2^1+\Gamma_2,
\end{align*}
%	\begin{equation*}		\begin{split}			\big( \widetilde{M}_2^+(\lambda)\big)^{-1}:=-\lambda \tilde a^+M_2^+(\lambda)			=D_1-\lambda B_1^1-\lambda^2 B_2^1+\Gamma_3(\lambda),		\end{split}	\end{equation*}
where  $\D_1=T_1^{-1}$, $B_1^1=-\tilde{a}^+\D_1B_2^0\D_1$ and $B_2^1=-\tilde{a}\D_1B_3^0\D_1- \tilde{a}^2\D_1(B_2^0\D_1)^2$ and we have used \eqref{orthog-relation-1} and \eqref{orthog-relation-2}.  % are absolutely bounded operators on $S_1L^2$ given by$$			B_1^1= -\tilde{a}^+D_1B_2^0D_1,\quad			B_2^1= -\tilde{a}^+D_1B_3^0D_1- (\tilde{a}^+)^2D_1(B_2^0D_1)^2.$$	
%	Since $\big( M_2^+(\lambda)\big)^{-1}=-\frac{\tilde{a}^+}{\lambda}\big( \widetilde{M}_2^+(\lambda)\big)^{-1}$, we have	\begin{equation*}\begin{split}\big( M_2^+(\lambda)\big)^{-1}=-\frac{\tilde{a}^+}{\lambda}D_1+\tilde{a}^+B_1^1+\tilde{a}^+\lambda B_2^1+\Gamma_2(\lambda).\end{split}	\end{equation*}
Moreover, 
Lemma \ref{lemma-JN} with $A=\widetilde{M}_1$ and $S=S_1$ and \eqref{id-M1tuba+S1} then yields
\begin{align*}
\label{M_1^{-1}}
\widetilde{M}_1^{-1}=
			(\widetilde{M}_1+S_1)^{-1}
			+ (\widetilde{M}_1+S_1)^{-1}S_1M_2^{-1}S_1
			( \widetilde{M}_1+S_1)^{-1}
= -\frac{\tilde{a}}{\lambda}\D_1+C^1_{1,0}+\lambda C^1_{1,1}+\Gamma_2,
\end{align*}
	where  $C^1_{1,0}=\tilde{a}\D_1B^0_1+\tilde{a} S_1B^1_1S_1+\tilde{a} B^0_1\D_1+\D_0$ and 	$$
C^1_{1,1}=\tilde{a} \D_1B^0_2-\tilde{a} S_1B^1_1S_1B^0_1+\tilde{a} S_1B^1_2S_1-\tilde{a} B^0_1\D_1B^0_1
			-\tilde{a} B^0_1S_1B^1_1S_1+\tilde{a} B^0_2\D_1-B^0_1.
$$
	Furthermore, since $M_1^{-1} = {\tilde{a}}{\lambda}^{-1}\widetilde{M}_1^{-1}$, then
\begin{equation*}
	\begin{split}
		M_1^{-1}=-\frac{\tilde{a}^2}{\lambda^2}\D_1+\frac{\tilde{a}}{\lambda}C^1_{1,0}+\tilde{a}C^1_{1,1}+\Gamma_1.
	\end{split}
\end{equation*}
Plugging this expansion $M_1^{-1}$ above  into \eqref{M^{-1}} with $N=3$, we arrive at
\begin{align*}
M^{-1}
&=\frac{\lambda}{\tilde{a}}\left\{I+ \Gamma_{1}+(I-\lambda B_1-\lambda^2 B_2+ \Gamma_{3})QM_1^{-1}Q(I-\lambda B_1-\lambda^2 B_2+ \Gamma_{3})\Big)\right\}\\
%&=-\frac{\tilde{a}^+}{\lambda}D_1+ \frac{1}{\tilde{a}^+}C^1_{0,-1}+  \frac{1}{\tilde{a}^+}\lambda C^1_{0,0} +\Gamma_2(\lambda)\\
			&= \lambda^{-1}S_1A^1_{-1,1}S_1 + S_1A^1_{0,1}+A^1_{0,2}S_1+QA^1_{0,3}Q+\lambda A_{1,1}^1+\Gamma_2,
\end{align*}
	where  $A^1_{i,j}\in \AB(L^2)$ are $\lambda$-independent and given as follows:
\begin{align*}
A^1_{-1,1}&=- \tilde{a} \D_1,\quad A^1_{0,1}=\D_1T,\quad A^1_{0,2}=T\D_1,\quad A^1_{0,3}=\tilde{a}\D_1B^0_1+\tilde{a}S_1B^1_1S_1+\tilde{a}B^0_1\D_1+\D_0,\\
A_{1,1}^1&=\frac{1}{\tilde a}+\tilde{a}\D_1B_2-QC^1_{1,0}QB_1+QC^1_{1,1}Q-\tilde{a}B_1\D_1B_1-B_1QC^1_{1,0}Q+\tilde{a}B_2\D_1.
\end{align*}
%%%%%%%%%%%%%%%%%%%%%%%%%%%%%%%%%%%%%%%%%%%%%%%%%%%%%%%%%%%%%%%%%%%%%%%%%%%%%
\underline{{\bf(iii) The second  kind  resonance case}}. The general strategy is essentially same as above. For the case with the second kind resonance, $T_1$ is not invertible on $S_1L^2$. Let $S_2$ be the orthogonal projection onto $\mathop{\mathrm{Ker}}T_1$ and  $\D_1= (T_1+S_2)^{-1}$ be the inverse of $T_1+S_2$ on $S_1L^2$.
	By the same argument as above with the expansion of $M_1$ given in \eqref{id-M1M1tabu} with $N=9$, we have \begin{align*}
M_1%=&\frac{\lambda}{\tilde{a}^+}QTQ + \sum_{k=2}^8\lambda^k QB_k^+Q +\Gamma_9(\lambda)\\
		&	=\frac{\lambda}{\tilde{a}}T_0 +\sum_{k=2}^8 \lambda^{k}Q B_kQ+ \Gamma_9:=\frac{\lambda}{\tilde{a}}\widetilde M_1, \quad	(\widetilde M_1+S_1)^{-1}=\D_0- \sum_{k=1}^7\lambda^kB_k^0 + \Gamma_8,\\
		%	:=\frac{\lambda}{\tilde{a}^+} \widetilde{M}_1^+(\lambda).
		M_2&:= S_1- S_1(\widetilde{M}_1+S_1)^{-1}S_1=-\frac{\lambda}{\tilde{a}}T_1+ \sum_{k=2}^7\lambda^{k} S_1B_k^0S_1
			+ \Gamma_8.
			%,\quad
		%	(\widetilde M_2+S_2)^{-1}=D_1- \sum_{k=1}^6\lambda^kB_k^1 + \Gamma_7,
\end{align*}
%where as stated  to get   \eqref{id-M1tuba+S1} and \eqref{Bjo}, we  employ the Neumann series expansion to derive the  expansion for $(\widetilde M_1+S_1)^{-1},$ and we can  calculate  that  $B_k^0\in \AB(L^2)$  (independent of $\lambda$) satisfy
%\begin{equation}\label{B0}
	%\begin{split}
	%	B_1^0=& \tilde{a}D_0B_2D_0, \quad
	%B_2^0= \tilde{a}D_0B_3D_0- \tilde{a}^2D_0(B_2D_0)^2,\\
%	B_3^0=& \tilde{a}D_0B_4D_0 -\tilde{a}^2( D_0B_2D_0B_3D_0 + D_0B_3D_0B_2D_0)
%	+\tilde{a}^3D_0(B_2D_0)^3.\\
	%	B_4^0=& \widetilde{a}D_0B_5D_0 -\widetilde{a}^2\big( D_0B_2D_0B_4D_0 + D_0B_4D_0B_2D_0
	%	+D_0(B_3D_0)^2 \big)\\
	%	&+ \widetilde{a}^3\big( D_0(B_2D_0)^2B_3D_0+ D_0B_2D_0B_3D_0B_2D_0
	%	+ D_0B_3D_0(B_2D_0 )^2 \big) -\widetilde{a}^4D_0(B_2D_0)^4.
%	\end{split}
%\end{equation}
Let $\widetilde M_2=-\lambda^{-1}\tilde a M_2$. Then we utilize the Neumann series expansion to derive
\begin{align}
	\nonumber
	(\widetilde{M}_2+S_2)^{-1}
	&=\left\{(T_1+S_2)\Big(I-\tilde a \sum_{k=2}^7\lambda^{k-1} \D_1S_1B_k^0S_1
	+ \Gamma_7\Big)\right\}^{-1}\\
	%&=\left\{I-\lambda \tilde a D_0QB_2Q-\lambda^2\tilde a(D_0QB_2Q)^2-\lambda^3\tilde a(D_0QB_2Q)^3+\Gamma_4\right\}D_0\\
	\label{id-M1tuba+S2}
	&= \D_1- \sum_{k=1}^6\lambda^kB_k^1 + \Gamma_7(\lambda),
\end{align}
as an operator on $S_1L^2$, where  $B_k^1\in \AB(L^2)$ are $\lambda$-independent, and
\begin{equation}\label{B1}
		\begin{split}
			B_1^1=& -\tilde{a}\D_1B_2^0\D_1, \quad
			B_2^1= -\tilde{a}\D_1B_3^0\D_1- \tilde{a}^2\D_1(B_2^0\D_1)^2,\\
			B_3^1=& -\tilde{a}\D_1B_4^0\D_1 -\tilde{a}^2( \D_1B_2^0\D_1B_3^0\D_1 + \D_1B_3^0\D_1B_2^0\D_1)
			-\tilde{a}^3\D_1(B_2^0\D_1)^3.
		\end{split}
	\end{equation}
By Lemma \ref{lemma-JN}, $ \widetilde{M}^+_2$ has a bounded inverse on $S_1L^2$
if and only if
\begin{align}
\label{M_3}
			M_3:= S_2- S_2(\widetilde{M}_2+S_2)^{-1}S_2
			=\sum_{k=1}^6\lambda^kS_2B_k^1S_2+\Gamma_7,
\end{align}
has a bounded inverse on $S_2L^2$.
Using \eqref{Bjo} and the properties \eqref{orthog-relation-1}-\eqref{orthog-relation-3}, we find $S_2B_1^1S_2=0$ and
	\begin{equation*}%\label{T3}
		\begin{split}
			S_2B_2^1S_2
			=& -\tilde{a}a_3\Big( S_2vG_3vS_2 +\frac{10}{3\|V\|_{L^1}}S_2(vG_1v)^2S_2
			-\frac{10}{3\|V\|_{L^1}} S_2vG_1vT\D_1TvG_1vS_2\Big)
			=-\tilde{a}a_3T_2,
		\end{split}
	\end{equation*}
where $T_2$ has been defined in Definition \ref{definition_resonance} and is invertible on $S_2L^2$ in the second kind resonance case.  Thus $	M_3$ is invertible on $S_2L^2.$ Specifically, by the Neumann series expansion, we obtain
\begin{align*}
	M_3^{-1}&=\left\{ -{\tilde a}a_3\lambda^2 T_2\left(I-     {\tilde a}^{-1}a_3^{-1}\sum_{k=3}^6\lambda^{k-2}\D_2S_2B_k^1S_2+\Gamma_5\right)\right\}^{-1}\\
	&=-({\tilde a}a_3)^{-1}\lambda^{-2}\D_2+({\tilde a}a_3)^{-1}\sum_{k=1}^4\lambda^{k-2} B_k^2+ \Gamma_3.
\end{align*} 
%	Furthermore, one has	\begin{equation}\label{M3-M3tuba}		\begin{split}			M_3^+(\lambda)
%=& -\tilde{a}^+a_3^+ \lambda^2 T_2 + \sum_{k=3}^6 \lambda^k S_2B_k^1S_2 +\Gamma_7(\lambda)\\			=&-\tilde{a}^+a_3^+\lambda^2 \Big(  T_2-			\frac{1}{\tilde{a}^+a_3^+ }\sum_{k=3}^6\lambda^{k-2} S_2B_k^1S_2+ \Gamma_5(\lambda)\Big)			:=-\tilde{a}^+a_3^+ \lambda^2 \widetilde{M}_3^+(\lambda).		\end{split}	\end{equation}
where $\D_2 = T_2^{-1}$, $B_k^2\in \AB(L^2)$ are independent of $\lambda$,  and
	\begin{equation}\label{B2}
		\begin{split}
			&B_1^2= -\frac{1}{\tilde{a}a_3}\D_2B_3^1\D_2,\quad B_2^2= -\frac{1}{\tilde{a}a_3}\D_2B_4^1\D_2
			-\frac{1}{(\tilde{a}a_3)^2}\D_2(B_3^1\D_2)^2.
			%&B_3^2= -\frac{1}{\tilde{a}^+a_3^+}D_2B_5^1D_2-\frac{1}{(\tilde{a}^+a_3^+)^2}(D_2B_3^1D_2B_4^1D_2+D_2B_4^1D_2B_3^1D_2)-
%			\frac{1}{(\tilde{a}^+a_3^+)^3}D_2(B_3^1D_2)^3,\\
%			&B_4^2= -\frac{1}{\tilde{a}^+a_3^+}D_2B_6^1D_2-\frac{1}{(\tilde{a}^+a_3^+)^2}\big(D_2B_3^1D_2B_5^1D_2
%			+D_2B_5^1D_2B_3^1D_2+D_2(B_4^1D_2)^2\big)\\
%			&-\frac{1}{(\tilde{a}^+a_3^+)^3}\big(D_2(B_3^1D_2)^2B_4^1D_2+D_2B_3^1D_2B_4^1D_2B_3^1D_2
%			+D_2B_4^1D_2(B_3^1D_2)^2\big)
%			-\frac{1}{(\tilde{a}^+a_3^+)^4}D_2(B_3^1D_2)^4.
		\end{split}
	\end{equation}
As in the previous case, we use Lemma \ref{lemma-JN} repeatedly to obtain	\begin{align*}
M^{-1}&=\lambda^{-3}S_2A^2_{-3,1}S_2+\lambda^{-2}(S_2A^2_{-2,1}S_1 + S_1A^2_{-2,2}S_2)
			+\lambda^{-1}(S_2A^2_{-1,1}+A^2_{-1,2}S_2+S_1A^2_{-1,3}S_1)\\
			&+(S_1A_{0,1}^2 +A^2_{0,2}S_1 +QA^2_{0,3}Q)
			+\lambda A_{1,1}^2+ \Gamma_2,
\end{align*}
where  all the $A^2_{i,j}\in \AB(L^2)$ are independent of $\lambda$. To show Theorem \ref{theorem1.2}, for  most operators $A_{i,j}^2$, we utilize the absolute boundedness only and do not need their explicit formulas, except for  operators $A^2_{-3,1}$, $A^2_{-2,1}$ and  $A^2_{-1,1}$. Thus for simplicity, we only give more details about  $A^2_{-3,1}$, $A^2_{-2,1}$ and  $A^2_{-1,1}$, which can be obtained by using the properties \eqref{orthog-relation-1}-\eqref{orthog-relation-3}:
\begin{align}
\label{A^2}
A^2_{-3,1}=4\pi\cdot5!(1-i)\D_2,\quad A^2_{-2,1}=20i\D_2vG_1vT\D_1,\quad A^2_{-1,1}=S_2H^2_{-1,1}+S_2H^2_{-1,2}Q,
\end{align}
where $H^2_{-1,2}\in \AB(L^2)$ and
\begin{align}
\label{H^2}
H^2_{-1,1}=80\pi(1+i)\|V\|^{-1}_{L^1}(\D_2vG_1v-\D_2vG_1vT\D_1T).
\end{align}
	%%%%%%%%%%%%%%%%%%%%%%%%%%%%%%%%%%%%%%%%%%%%%%%%%%%%%%%%%%%%%%%%%%%%%%%%%%%%%%%%%%%%%%%%%%%%
\underline{{\bf (iv) The third kind  resonance case}}. By the same argument as above, we obtain the more detail expansion of
	$M_3(\lambda)$ than \eqref{M_3} as follows:
\begin{align*}
M_3&=-\tilde{a}a_3\lambda^2T_2+\sum_{k=3}^8\lambda^{k} S_2B_k^1S_2+ \Gamma_9\\
&=-\tilde{a}a_3\lambda^2\Big(T_2-\tilde{a}^{-1}a_3^{-1}\sum_{k=3}^8\lambda^{k-2} S_2B_k^1S_2+ \Gamma_7\Big)
:=-\lambda^2\tilde{a}a_3\widetilde{M}_3,
\end{align*}
where $T_2$ is not invertible on $S_2L^2$ in the third kind resonance case. We denotes by $S_3$ the orthogonal projection onto $\Ker (T_2)$ and $\D_2= (T_2+S_3)^{-1}$. Then by the Neumann series expansion, 
$$
(\widetilde{M}_3+S_3)^{-1}= \D_2- \sum_{k=1}^6\lambda^kB_k^2 + \Gamma_7,
$$
where  $B_k^2\in \AB(L^2)$ are independent of $\lambda$. In particular, 
$B_k^2 \  (k=1,2)$ can be seen in \eqref{B2}.  Moreover, 
by Lemma \ref{lemma-JN}, $ \widetilde{M}^+_3$ has a bounded inverse on $S_2L^2$
if and only if
\begin{align*}
	M_4:= S_3- S_3(\widetilde{M}_3+S_3)^{-1}S_3
	=\sum_{k=1}^6 \lambda^kS_3B_k^2S_3+\Gamma_7(\lambda),
\end{align*}
has a bounded inverse on $S_3L^2$.
Using the orthogonality relationship \eqref{orthog-relation-1}-\eqref{orthog-relation-3}, we have
\begin{equation}%\label{T3}
	S_3B_1^2S_3= \frac{1}{a_3^+}S_3vG_4vS_3:=\frac{1}{a_3^+}T_3 .
\end{equation}
 Since $T_3$ is always invertible on $S_3L^2,$ 
 it follows  $	M_4$ is invertible on $S_3L^2.$
 Specifically, by the Neumann series expansion, we obtain
 \begin{align*}
 	M_4^{-1}&=\left\{ \frac{\lambda}{a_3} T_3\left(I+  \sum_{k=2}^6a_3\lambda^{k-1}\D_3S_3B_k^2S_3+\Gamma_6\right)\right\}^{-1}\\
 	&=a_3\lambda^{-1}\D_3+\sum_{k=1}^5a_3\lambda^{k-1} C^3_{4,k}+\Gamma_5,
 \end{align*} 
	where $\D_3=T_3^{-1}$ and $C^3_{4,k}$
	 are absolutely bounded operators on $S_3L^2$ and independent of $\lambda,$ whose explicit formulas can be calculated in details. We now apply Lemma \ref{lemma-JN} repeatedly to obtain
\begin{align*}
M^{-1}
&=\lambda^{-4}S_3 A^3_{-4,1} S_3+\lambda^{-3}S_2A^3_{-3,1}S_2+\lambda^{-2}(S_2A^3_{-2,1}S_1+ S_1A^3_{-2,2}S_2)\\
&+\lambda^{-1}(S_2A^3_{-1,1}+ A^3_{-1,2}S_2 +S_1A^3_{-1,3}S_1) +(S_1A_{0,1}^3 +A^3_{0,2}S_1 +QA^3_{0,3}Q) +\lambda A_{1,1}^3+\Gamma_2,
\end{align*}
where  $A^3_{i,j}\in \AB(L^2)$ are independent of $\lambda$. Here we list the explicit formulas of a part of $\{A^3_{i,j}\}$ needed in the proof of Theorem \ref{theorem1.2}, which can be obtained by using \eqref{orthog-relation-1}-\eqref{orthog-relation-3}: 
\begin{align}
\label{A_3_1}
A^3_{-4,1}&=\D_3=(S_3vG_4vS_3)^{-1},\\
\label{A_3_2}
A^3_{-3,1}&=S_2H^3_{-3,1}S_2+S_2H^3_{-3,2}S_3,\\
\label{A_3_3}
A^3_{-2,1}&=i(32\pi^2\cdot5!)^{-1}\|V\|_{L^1}\D_3vG_3v\D_1+20i(\D_2vG_1vT\D_1-\D_3vG_4v\D_2vG_1vT\D_1)\\&+i(32\pi^2\cdot3!^2)^{-1}\D_3vG_3v\D_1TvG_1v\D_2vG_1vT\D_1,\\
\label{A_3_4}
A^3_{-1,1}&=S_2H^3_{-1,1}+S_2H^3_{-1,2}Q,
\end{align}
where $H^3_{j,k}$ are absolutely bounded operators on $L^2$ and 
 \begin{align}
\label{A_3_5}
 H^3_{-3,1}&=4\pi\cdot5!(1-i)(\D_2-\D_3vG_4v\D_2)+(1-i)(8\pi\cdot3!)^{-1}\D_3vG_3v\D_1TvG_1v\D_2,\\
\label{A_3_6}
 H^3_{-1,1}&=\frac{80\pi(1+i)}{\|V\|_{L^1}}\Big(\D_2vG_1v-\D_3vG_4v\D_2vG_1v-\D_2vG_1vT\D_1T+\D_3vG_4v\D_2vG_1vT\D_1T\Big)\nonumber\\
			&\ \ \ \  +\frac{1+i}{8\pi(3!)^2\|V\|_{L^1}}\Big(\D_3vG_3v\D_1TvG_1v\D_2vG_1v-\D_3vG_3v\D_1TvG_1v\D_2vG_1vT\D_1T\Big)\nonumber\\
			&\ \ \ \  -\frac{1+i}{8\pi\cdot5!}\D_3vG_3v\D_1T.
\end{align}
The proof of Theorem \ref{thm-main-inver-M} is completed.	
\end{proof}

\section*{Acknowledgments}
H. Mizutani is partially supported by JSPS KAKENHI Grant Numbers JP21K03325 and JP24K00529. Z. Wan and X. Yao are partially supported by NSFC grants No.11771165 and 12171182. The authors would like to express their thanks to Professor Avy Soffer for his interests and insightful discussions about topics on higher-order operators.

%%%%%%%%%% Bibliography %%%%%%%%%%%%%%%%%%%

\end{document}